\documentclass[11pt, leqno]{article}
\usepackage{graphicx}
\usepackage{amsmath, amsthm, amssymb, euler}

\newcommand{\R}{\mathbb{R}}
\newcommand{\Z}{\mathbb{Z}}
\newcommand{\bes}{\dot{B}_p}
\newcommand{\besfull}{B_p}
\newcommand{\WR}{L^{2,p}(\R^2)}

\newcommand{\WE}{L^{2,p}(\R^2)|_E}
\newcommand{\CZ}{\{Q_\nu\}_{\nu=1}^{K}}

\newtheorem{defn}{Definition}
\newtheorem{thm}{Theorem}
\newtheorem{lem}{Lemma}

\newtheorem{prop}{Proposition}
\newtheorem{rek}{Remark}

\newcommand\SLASH{\char`\~
}

\numberwithin{equation}{section}
\numberwithin{thm}{section}
\numberwithin{defn}{section}
\numberwithin{lem}{section}
\numberwithin{cor}{section}
\numberwithin{prop}{section}
\numberwithin{rek}{section}

\begin{document}
\begin{titlepage}

\title{A Bounded Linear Extension Operator for $L^{2,p}(\R^2)$}
\author{Arie Israel}
\date{}
\maketitle{}

\begin{abstract}
For a finite $E \subset \R^2$, $f:E \rightarrow \R$, and $p>2$, we produce a continuous $F:\R^2 \rightarrow \R$ depending linearly on $f$, taking the same values as $f$ on $E$, and with $\WR$ semi-norm minimal up to a factor $C=C(p)$. This solves the Whitney extension problem for the Sobolev space $\WR$. A standard method for solving extension problems is to find a collection of local extensions, each defined on a small square, which if chosen to be mutually consistent can be patched together to form a global extension defined on the entire plane. For Sobolev spaces the standard form of consistency is not applicable due to the (generically) non-local structure of the trace norm. In this paper, we define a new notion of consistency among local Sobolev extensions and apply it toward constructing a bounded linear extension operator. Our methods generalize to produce similar results for the $n$-dimensional case, and may be applicable toward understanding higher smoothness Sobolev extension problems.
\end{abstract}
\end{titlepage}
\section{Introduction}
\label{intro}
The study of extension problems in its first form began with Whitney's work dating back to 1934. Whitney was concerned with an intrinsic description of the restriction of a ``smooth" function $F$ defined on $\R^n$ to a closed subset $E \subset \R^n$. Of course, we have many useful definitions for ``smooth", and for each we find an interesting variation of the above question. For a modern description of the problem at hand, we must fix a Banach (resp. semi-normed complete vector) space $X(\R^n)$ consisting of smooth functions on $\R^n$. These include, e.g.,  $X=C^m,C^{m,\omega}$, and $W^{m,p}$; respectively, these are the spaces of continuously differentiable, differentiable with modulus of continuity $\omega$, and Sobolev functions. These spaces consist of functions $F: \R^n \rightarrow \R$ with finite norm - and in the case of $C^m(\R^n)$, continuous $m$'th derivatives - as below:
\begin{align*}
& \|F\|_{C^m(\R^n)} = \sup_{x \in \R^n} \max_{k \leq m} |\nabla^k F(x)| < \infty, \; \mbox{with} \;  \nabla^m F \; \mbox{continuous};\\
& \|F\|_{C^{m,\omega}(\R^n)} = \|F\|_{C^m(\R^n)} + \sup_{x,y \in \R^n} \frac{|\nabla^m F(x)- \nabla^m F(x)|}{\omega(|x-y|)} < \infty; \\
& \|F\|_{W^{m,p}(\R^n)} = \sum_{k \leq m} \left( \int_{\R^n} |\nabla^k F|^p dx \right)^{1/p} < \infty.
\end{align*}
We also denote the homogeneous versions of these spaces by $\dot{C}^m(\R^n), \dot{C}^{m,\omega}(\R^n)$, and $L^{m,p}(\R^n)$, which are defined via finiteness of the analogous semi-norm where all but the highest order derivatives are dropped. Here $\omega:\R_+ \rightarrow \R_+$ is a specified modulus of continuity, which is usually taken to satisfy some mild regularity assumptions. When $\omega(t)=t^\alpha$ for $0 < \alpha \leq 1$, we write $C^{m,\alpha}(\R^n)$, and $\dot{C}^{m,\alpha}(\R^n)$ for the respective spaces. For technical reasons it will be convenient to restrict attention to $L^{m,p}(\R^n)$ (or $W^{m,p}(\R^n)$) for the range $n<p<\infty$. We recall that for $p>n$, the Sobolev embedding theorem implies that $L^{m,p}(\R^n) \subset \dot{C}^{m-1,\alpha}(\R^n)$ ($\alpha=1-n/p$), with functions on the left space identified up to equality on a set of measure zero with functions on the right space. Thus we may assume that pointwise evaluation of derivatives through order $m-1$ is well-defined for any $F \in L^{m,p}(\R^n)$.

For $X$ any Banach (resp. semi-normed complete vector) space, and $E \subset \R^n$ arbitrary, we define the \textit{trace space} $X(\R^n)|_E := \{F|_E : F \in X(\R^n)\}$. This vector space carries the natural \textit{trace norm} (resp. \textit{semi-norm}): $\|f\|_{X(\R^n)|_E} = \inf \{\|F\|_{X(\R^n)} , F|_E = f\}$. Even for finite $E$, the finite-dimensional norm $\|f\|_{X(\R^n)|_E}$ is often non-trivial to calculate to within a factor of $C=C(X)$ (independent of the number of points in $E$). For a given $X$, and $E$ now arbitrary, we formulate the \textit{Whitney extension problem}.

\textbf{Question 1.} Given $E \subset \R^n$ arbitrary, and $f:E \rightarrow \R$, does $f$ \textit{extend} to $F \in X(\R^n)$ with $F|_E = f$? Can we take this extension to depend linearly on the data $f \in X(\R^n)|_E$?

For spaces that possess some form of compactness (e.g. the Arzel\`a-Ascoli theorem), such as $C^{m,\omega}(\R^n)$ and $L^{m,p}(\R^n)$, the quantitative finitary version of Question 1 is equivalent to Question 1 itself. Thus we focus on the finite version of the problem, which are Questions 2 and 3 below, but advise the reader that there are additional technical problems that arise in the solution of Question 1 when compactness fails, as it does for $C^m(\R^n)$; see \cite{F3,F4} for a solution of Question 1 for $C^m(\R^n)$.

\textbf{Question 2.} Given a \textit{finite} $E \subset \R^n$, and $f:E \rightarrow \R$, compute a real number $M(f) \geq 0$ so that $M(f) \approx \|f\|_{X(\R^n)|_E}$.

\textbf{Question 3.} For $E \subset \R^n$ finite, does there exist $T: X(\R^n)|_E = \{f:E \rightarrow \R\} \rightarrow X(\R^n)$, with $T(f)|_E = f$, and $\|T(f)\|_{X(\R^n)} \lesssim \|f\|_{X(\R^n)|_E}$ for all $f:E \rightarrow \R$?

For two non-negative real-valued quantities $A,B$, we write $A \lesssim B$ (resp. $A \approx B$) if and only if there exists a constant $C=C(X)$ (independent of $A$ and $B$) so that $A \leq CB$ (resp. $\frac{1}{C} B \leq A \leq CB$). In Question 2, the word ``compute" may mean, e.g., an explicit formula for $M(f)$ in terms of the function $f$. We call $T$ as in Question 3 a bounded linear extension operator for $X(\R^n)|_E$.

In this paper, we begin study of the structure of the semi-normed (finite-dimensional) vector space $L^{2,p}(\R^2)|_E$ for finite sets $E$, in particular solving Questions 2 and 3 with the following theorem.

\begin{thm}
\label{mainthm}
Fix $2<p<\infty$. Let $E \subset \R^2$, with cardinality $\#(E)=N$. Then there exists a bounded linear operator $T: \WE \rightarrow \WR$, and a non-negative real number $M(\cdot)$, which satisfy:
\begin{enumerate}
\item $Tf|_E = f$ for all $f: E \rightarrow \R$, and 
\item $\|T(f)\|_{L^{2,p}(\R^2)} \approx \|f\|_{L^{2,p}(\R^2)|_E} \approx M(f)$ for some constant $C = C(p)$.
\item There exist linear functionals $\{\lambda_i\}_{i=1}^{N_0}$, with $N_0 \lesssim N^2$, so that $M(f) = \left(\sum_{i=1}^{N_0} |\lambda_i(f)|^p\right)^{1/p}$.
\end{enumerate}
\end{thm}

In \cite{FIL1}, a forthcoming paper joint with C. Fefferman and G.K. Luli, we show that these methods generalize to prove Theorem \ref{mainthm} for $L^{2,p}(\R^n)$. Unfortunately, for dimension $n \geq 3$ the current proof of Theorem \ref{mainthm} is non-constructive. We focus here on the two dimensional case for sake of clarity.

We now recall some of the history of extension theory, and outline a few important differences in our approach.

Question 1 was originally introduced by Whitney, who solved it for the space $C^m(\R)$ in \cite{W1} through the method of finite differences. Moreover, for $C^m(\R^n)$, Whitney proved the classical Whitney extension theorem (see \cite{St1,W2}), which solved the following variant of Question 1: Let $P_E=(P_x)_{x \in E}$ be a collection of $m$'th degree polynomials indexed by points of $E$ which satisfy $|\partial^\alpha(P_x-P_y)(x)| \cdot |x-y|^{m-|\alpha|} \leq M \cdot o(|x-y|)$ and $|\partial^\alpha P_x(x)| \leq M$ for some real number $M\geq0$, all multi-indices $\alpha$ with $0 \leq |\alpha| \leq m$, and all $x,y \in E$. Then there exists a function $F=T(P_E) \in C^m(\R^n)$ depending linearly on $P_E$ with $m$'th degree Taylor polynomials specified by $J^m_x F = P_x$ for each $x \in E$; moreover, the $C^m$-norm of $T(P_E)$ is comparable to the least possible value of $M$ as above. Thus, Whitney found that the corresponding extension problem for $C^m$ is easier once one adds additional constraints on the derivatives through order $m$ on the set $E$, and one can recover optimal interpolants through a linear operator. All future work on Question 1 has relied on the ability to ``guess" the full jets of the sought extension on $E$ through examination of the function values $f:E\rightarrow \R$.

In 1985, Y. Brudnyi and P. Shvartsman proposed the finiteness conjecture for $C^{m-1,\omega}(\R^n)|_E$ (see Theorem \ref{finprin} for a statement), which if proven would offer a solution to Questions 2 and 3 for this space; moreover, Brudnyi and Shvartsman proved their conjecture for the case $m=2$, thereby answering Questions 2 and 3 affirmatively for $C^{1,\omega}(\R^n)$ (see \cite{BS1,BS2,BS3,BS4,S2,S3,S4} for these and related results). Further progress came in \cite{F1,F2,F3,F4} where Fefferman proved the finiteness conjecture and introduced bounded linear extension operators for the class of spaces $C^m(\R^n)|_E$ and $C^{m,\omega}(\R^n)|_E$ for all $m,n \geq 1$, $E \subset \R^n$ closed, and $\omega$ satisfying some mild regularity conditions; thus Questions 2 and 3 have been resolved for these spaces as well.

This raises similar questions for the next classical space in line: the Sobolev spaces $L^{m,p}(\R^n)$ and $W^{m,p}(\R^n)$. In \cite{L1}, Luli proves that Whitney's one-dimensional construction also gives a bounded linear extension operator $T: L^{m,p}(\R)|_E \rightarrow L^{m,p}(\R)$ for finite $E \subset \R$. A recent interesting development came when Shvartsman showed that the classical Whitney extension operator (for $C^{m-1}(\R^n)$) also produces a function $F = T(P_E) \in L^{m,p}(\R^n)$, with $J^{m-1}_x(F)=P_x$ $\forall x \in E$, and $L^{m,p}(\R^n)$-norm minimal up to a constant $C(m,p,n)$. As a consequence, Shvartsman produces a bounded linear extension operator $T: L^{1,p}(\R^n)|_E \rightarrow L^{1,p}(\R^n)$, and solves Questions 2 and 3 for this space. We refer the reader to \cite{S1} for a statement of these and other results.

The inherent gap in difficulty between $L^{1,p}(\R^2)$ and $L^{2,p}(\R^2)$ comes from the fact that pointwise evaluation for the gradients of an $L^{1,p}(\R^2)$ function makes little sense, and so the information required to apply the classical Whitney extension operator (for $L^{1,p}(\R^2)$) are precisely the function values $(f(x))_{x \in E}$. On the other hand, for $L^{2,p}(\R^2)$ functions ($p>2$), pointwise evaluation of the gradient makes sense, and from the Sobolev embedding theorem we find a simple notion of consistency among gradients given by
$$|\nabla F(x) - \nabla F(y)| \lesssim |x-y|^{1-2/p} \|F\|_{L^{2,p}(\R^2)}$$
That is, when choosing an extension $F$ of $f$, we must ensure that its gradient vectors are consistent enough so that the Sobolev embedding theorem does not force it to have large norm; the choice of gradient for our extension at certain points of $\R^n$ is key to our discussion, and is first necessary when dealing with the class of Sobolev spaces $L^{2,p}(\R^2)$ ($p>2$).

A few comments on the proof of Theorem \ref{mainthm} are in order: For certain sets $E$ which appear ``flat" (e.g., $E$ lies on a line), the corresponding interpolation problems are easier, and we provide a bounded linear extension operator $T$ as well as an approximate formula for the trace semi-norm. Since we are seeking an interpolant of $f$ in the homogeneous space $\WR$ that has close to optimal semi-norm, when $E$ lies on a line segment we have the freedom to add a large multiple of an affine function that vanishes on $E$ to any proposed optimal interpolant $F \in \WR$ without affecting the size of its semi-norm. That is, any bounded linear extension operator $T$ is certainly not close to unique, since $\tilde{T}(f) := T(f) + L$, for $L$ affine and vanishing on $E$, is also a bounded linear extension operator. There exists a corresponding remark when $E$ is merely ``flat", and there will still be freedom in choosing $T$ in this case as well.

Now, for an arbitrary finite set $E \subset \R^2$, through use of a Calder\'on-Zygmund decomposition we partition $E$ into local pieces $\{E_\nu\}^K_{\nu=1}$ given by taking the intersection of $E$ with a family of CZ squares. $E_\nu$ are essentially disjoint, with $E_\nu \subset E$, $\bigcup_\nu E_\nu = E$, and $E_\nu$ ``flat". For each of these sets $E_\nu$, we form an associated local interpolation problem for $f|_{E_\nu}$, which we can solve thanks to the above remarks. This local extension is not uniquely determined, and there is no obvious choice which eliminates the inherent freedom. One of the key aspects of our construction is a resolution of this troubling non-uniqueness that we achieve by making extra assumptions on our local interpolants, making them globally consistent. All of this follows from a closer examination of the geometry of our Calder\'on-Zygmund decomposition; this leads to the notion of Keystone squares described in \ref{alg}. Finally, upon patching together these local solutions using a partition of unity, we arrive at an interpolant $T(f) \in L^{2,p}(\R^2)$ of $f:E\rightarrow \R$ which is provably optimal, and obviously linear.

Calder\'on-Zygmund decompositions first entered the picture in \cite{F1}, where Fefferman used them in an intricate induction procedure to prove the finiteness principle for $X=C^m, C^{m-1,1}$. Though, in his proof, each of the Calder\'on-Zygmund squares were treated as an equal, no square more important than its fellow. This is in stark contrast to the picture for $L^{2,p}(\R^2)$, where certain Calder\'on-Zygmund squares - which we coin Keystone squares - are used to determine the behavior of our interpolant on far away regions of $\R^2$. This non-local behavior makes its first appearance in our solution to the Whitney extension problem for $L^{2,p}(\R^2)$. In fact, in \cite{FIL1} we show that this is, in some sense, a necessary feature of Sobolev extension operators.

For comparative purposes, we recall from \cite{F2} a theorem of Fefferman that resolves Questions 2 and 3 for the space $C^{m-1,1}(\R^n)$.

\begin{thm}[Fefferman '03/'05]
\label{finprin}
There exists a positive integer $k^\#(m,n)$ so that the following holds: Let a finite $E \subset \R^n$ be given. Then for any $f:E \rightarrow \R$ we have
$$\|f\|_{C^{m-1,1}(\R^n)|_E} \approx \max_{\substack{S\subset E \\ \#(S) \leq k^\#} } \{\|f|_S\|_{C^{m-1,1}(\R^n)|_S}\}.$$

Moreover, there exists a bounded linear operator $T: C^{m-1,1}(\R^n)|_E \rightarrow C^{m-1,1}(\R^n)$ which satisfies
\begin{enumerate}
\item $Tf|_E = f$ for all $f:E \rightarrow \R$;
\item $\|Tf\|_{C^{m-1,1}(\R^n)} \leq C \|f\|_{C^{m-1,1}(\R^n)|_E}$, for $C = C(m,n)$;
\item $Tf(y) = \sum_{x \in E} f(x) \lambda_x(y)$, so that for all $y \in \R^n$, $\# \{x \in E: \lambda_x(y) \neq 0\} \leq k$, with $k=k(m,n)$.
\end{enumerate}
\end{thm}

The depth of a bounded linear operator $T$ is defined to be the smallest $k$ such that $P3$ holds. An extension operator $T$ which satisfies $P3$ is said to be of \textit{bounded depth}. Notice that the constants $k$, $k^\#$, and $C$ depend only on $m$ and $n$. This property is important, since the above theorem is trivial if we allow them to depend on $\#(E)$. Limiting behavior of the linear operators $T$ for growing finite sets, along with compactness properties of $C^{m-1,1}(\R^n)$, allow one to recover a variant of Theorem \ref{finprin} for arbitrary closed sets $E \subset \R^n$. The difficulty is that in the limit one often loses control on the bounded depth property of the operator. In \cite{L2}, these concerns were answered by Luli when he produced bounded depth $C^{m,\omega}(\R^n)$ extension operators for $E$ an arbitrary closed set.

To understand the statement of Theorem \ref{finprin}, we should consider $m,n$ to be fixed integers, but imagine that $\# E = N$ with $N$ very large (compared to $k^\#(m,n)$). Theorem \ref{finprin} then states: If for all small sets $S \subset E$ with at most $k^\#$ elements there exists an interpolant of $f|_S : S \rightarrow \R$ with $C^{m-1,1}(\R^n)$-norm less than or equal to $M$, then there exists an interpolant of $f$ with $C^{m-1,1}(\R^n)$-norm less than or equal to $CM$, for $C=C(m,n)$.

One may wonder if Theorem \ref{finprin} actually provides a formula for the trace norm. In fact, it is simple to find explicit linear functionals $\{\lambda_S^i\}^C_{i=1} \in (\WE)^*$ ($C=C(m,n)$) with $\|f|_S\|_{C^{m-1,1}(\R^n)|_S} \approx \max_i |\lambda^i_S(f|_S)|$, as long as $\#(S) \leq k^\#(m,n)$. This follows as a consequence of the classical Whitney extension theorem, whereby one is able to formulate the trace norm as the global minimum of a quadratic form (see Section 3 in \cite{F2}). Thus, from Theorem \ref{finprin},
\begin{equation}
\label{cmform0}
\|f\|_{C^{m-1,1}(\R^n)|_E} \approx \max_{\substack{S \subset E \\ \#(S) \leq k^\#}} \max_{1\leq i \leq C} |\lambda_S^i(f)|,
\end{equation}
with $\lambda^i_S$ a linear functional of $f$ depending only on the restriction of $f$ to a subset $S$ of size at most $k^\#$.

It should be noted that this formula is still somewhat undesirable. In fact, $O(N^{k^\#})$ terms appear on the RHS of \eqref{cmform0}. So, even if each term requires little work to calculate, the sheer number of terms makes this formula somewhat unpleasing from a computational perspective. In \cite{FK1}, Fefferman and B. Klartag remedy this situation with the much improved result: $\exists S_1,\cdots,S_L \subset E$ ($L \leq CN$), with $\#(S_i) \leq k^\#$, and for which 
\begin{equation}
\label{cmform}
\|f\|_{C^{m-1,1}(\R^n)|_E} \approx \max_{1 \leq i \leq L} \|f|_S\|_{C^{m-1,1}(\R^n)|_S} \approx \max_{1 \leq i \leq L'} |\lambda^i(f)|,
\end{equation}
for $\{\lambda^i\}^{L'}_{i=1}$ ($L' \leq C' N$) linear functionals each depending on the restriction of $f$ to a set of size at most $k^\#$.

Note the close analogy between Theorem \ref{mainthm} and Theorem \ref{finprin}. In particular, our formula for $\|f\|_{\WR|_E}$ corresponds with \eqref{cmform}, but with an $l^p$ norm replacing the $l^\infty$ norm present in \eqref{cmform}. Upon comparing Theorem \ref{mainthm} with Theorem \ref{finprin} and \eqref{cmform}, we find a few differences, which are listed below:

\begin{itemize}
\item In \eqref{cmform}, $\lambda^i(f)$ each depend on only $k^\#(m,n)$ function values, but no such property is satisfied for the $\lambda_i$ in Theorem \ref{mainthm}.
\item The total number of linear functionals in \eqref{cmform} is $O(N)$; whereas in Theorem \ref{mainthm} we use $O(N^2)$ linear functionals.
\item We have no bound on the depth of the linear map $T$ in Theorem \ref{mainthm}.
\end{itemize}
 
In \cite{FIL1} we exhibit a finite set $E \subset \R^2$ so that the depth of \textit{any} linear extension operator $T: \WR|_E \rightarrow \WR$ necessarily depends on $N$; this highlights a key difference between the Sobolev and $C^m$ versions of the problem, and resolves the third bullet-point above. In \cite{FIL1} we also sharpen the conclusions of Theorem \ref{mainthm} as follows:

Fix $\omega_1, \cdots, \omega_L \in (L^{2,p}(\R^2)|_E)^*$. We say that a linear functional $\lambda \in (L^{2,p}(\R^2)|_E)^*$ is of \textit{assisted bounded depth} with \textit{assists} $\omega_1,\cdots \omega_L$ if
\begin{align*}
&\lambda(f) = \sum_{x \in E} \alpha^x f(x) + \sum_{k=1}^{L} \beta^k \omega_k(f), \; \mbox{with} \\
&\#\{x \in E : \alpha^x \neq 0\} + \#\{1 \leq k \leq L: \beta^k \neq 0\} \leq k, \; \mbox{and} \\
& \displaystyle \sum_{l=1}^L ( \# \; \mbox{of non-zero coefficients of} \; \omega_l ) \leq C N,
\end{align*}
for some constants $k,C$ depending only on $p$. We prove that a simple modification to $T$ from Theorem \ref{mainthm} gives us the following improved properties:
\begin{align*}
&\mbox{There exist assists} \; \omega_1,\cdots,\omega_L \in (L^{2,p}(\R^2)|_E)^*, \;\mbox{so that} \\
&\mbox{for all} \; x \in \R^2, Tf(x) \; \mbox{is of assisted bounded depth. Moreover,} \\
&\mbox{There exist linear functionals} \; \lambda_1, \cdots, \lambda_{L'} \; \mbox{of assisted bounded depth with} \; L' \leq C N, \; \mbox{and} \\
&\|Tf\|^p_{L^{2,p}(\R^2)} \approx M(f)^p := \sum_{i=1}^{L'} |\lambda_i(f)|^p. \\
\end{align*}
Finally, we pose an open question regarding the existence of a formula for the norm which is closer in spirit to \eqref{cmform}. There are two formulations to this question, and we present both for the sake of completeness.

\noindent\textbf{Open Problem (Form 1):} Given $E \subset \R^2$ with $\#(E)=N$, do there exist linear functionals $\{\lambda_i\}^{L}_{i=1} \subset (\WR)^*$ ($L \leq CN$), each depending on only $O(1)$ function values, and for which $$\|f\|^p_{L^{2,p}(\R^2)|_E} \approx \sum^L_{i=1} |\lambda_i(f)|^p \; \mbox{?}$$
\textbf{Open Problem (Form 2):} Given $E \subset \R^2$ with $\#(E)=N$, do there exist $S_1,S_2,\cdots,S_L \subset E$, and $B_1,\cdots,B_L > 0$ with $\#(S_i) \leq O(1)$, $L \leq CN$, and for which $$\|f\|^p_{L^{2,p}(\R^2)|_E} \approx \sum^L_{i=1} B_i\|f|_{S_i}\|^p_{L^{2,p}(\R^2)|_{S_i}} \; \mbox{?}$$

By an argument similar to the one producing \eqref{cmform0}, a positive answer to Form 2 would imply one for Form 1. These two problems are conjectured analogues of the finiteness principle for $C^m(\R^n)$. It is still one of our goals to resolve either Form 1 or Form 2 listed above.

\subsection{Acknowlegements}

This paper forms the main body of the author's doctoral dissertation. I would like to thank my thesis advisor, Charles Fefferman, for proposing this problem to me, for his always thoughtful remarks, and seemingly infinite patience.

Recently, during the '10 Whitney Workshop held at AIM, P. Shvartsman announced his proof of existence of bounded linear extension operators for $L^{2,p}(\R^2)$. There, I heard about the techniques he developed for studying Sobolev extensions; it would be interesting for me to study the relationship between his methods and the ones put forth here. We look forward to hearing more of his work in the near future.

Open Problem (Form 2) as listed in the introduction was first conjectured as a direct analogue of the finiteness principle for $C^m(\R^n)$ by participants of the '09 Whitney Workshop at William \& Mary College sponsored by the ONR. I am pleased to thank ONR, as well as the organizers and participants of this Workshop for helping to bring about my interest in these problems. It is also a great pleasure to thank AIM for hosting the '10 Whitney Workshop and providing a fruitful environment for research; work contained in this paper, along with ideas that are currently being developed, was carried out on their grounds this past Summer. In particular, I am grateful to Bo'az Klartag, Kevin Luli, Assaf Naor, Pavel Shvartsman, and Nahum Zobin for many stimulating discussions over the past years.

\section{Notation and Definitions}
\label{notation}
For $x \in \R^2$, $|x|$ is to denote the standard Euclidean norm of $x$. For $\Omega_1,\Omega_2 \subset \R^2$, we denote the distance between them by $d(\Omega_1,\Omega_2) = \inf\{|x-y| : x \in \Omega_1,\; y \in \Omega_2\}$. For a multi-index $\alpha = (\alpha_1,\alpha_2) \in \Z_+^2$, we denote its \textit{order} by $|\alpha| = \alpha_1 + \alpha_2$. By a \textit{Universal Constant} $C$ we mean a positive real number depending only on $p$. Two non-negative real numbers $A$ and $B$ are said to be $C$-equivalent, or rather $A \approx B$, if $A/C \leq B \leq CA$.

Often we will construct an object $\mathcal{O}$, and satisfying certain properties, numbered, e.g., $(1)$, $(2)$, $(3)$. These properties will be referred to within the body of text where this object is defined (be it a certain section, lemma or proposition) as $P1$ of $\mathcal{O}$, $P2$ of $\mathcal{O}$, and $P3$ of $\mathcal{O}$.

A \textit{Euclidean coordinate} system on $\R^2$ consists of two affine functions $z_1,z_2 : \R^2 \rightarrow \R$ with $\nabla z_1 \cdot \nabla z_2 = 0$, whereby we can uniquely represent $x \in \R^2$ in coordinates: $x = (z_1(x),z_2(x))_{z_1z_2}$. Alternatively, given two orthonormal vectors $e_1,e_2 \in \R^2$, and a basepoint $x_0 \in \R^2$, we may set $z_i(x) = e_i \cdot(x-x_0)$, which forms a Euclidean coordinate system.

For a line $l \subset \R^2$, and $\Omega \subset \R^2$, we define the projection of $\Omega$ onto $l$ by $\mbox{proj}_{l} \Omega = \bigcup_{x \in \Omega} \mbox{proj}_l x$, with $\mbox{proj}_l x$ the standard orthogonal projection of $x \in \R^2$ onto $l$.

Given a map $\Phi: \Omega \rightarrow \R^2$, and a choice of Euclidean coordinates $(z_1,z_2)$, the components of $\Phi$ (relative to $(z_1,z_2)$) are denoted by $\Phi^1$, $\Phi^2$. That is, $\Phi^i$ is defined to satisfy: $\Phi(x) = (\Phi^1(x),\Phi^2(x))_{z_1z_2}$. Of course this notation is dependent on $(z_1,z_2)$, but the dependence of $\Phi^i$ on the coordinate system is dropped when it is clear from the context.

For the following definition we must fix a base Euclidean coordinate system on $\R^2$ for the remainder of the paper; this base coordinate system is left unnamed. A \textit{square} $Q \subset \R^2$ is of the form $[a_1,b_1]\times[a_2,b_2]$ (where $b_1 - a_1 = b_2 - a_2 > 0$). Note from the definition that an arbitrary rotation of a square is not necessarily a square. We denote the sidelength of $Q$ by $\delta_Q = b_1 - a_1$, and the center of $Q$ by $c_Q$. For a real number $A>0$, we define $AQ = \{A(x-c_Q)+c_Q: x \in Q\}$, which is the $A$-dilate of $Q$ about its center. We set $\tilde{Q} = 1.3 Q$. When $G \subset \R^2$ is any set that fails to be a square, we choose to abuse notation and define $A G = \{A y : y \in G\}$ for $A>0$. In particular the preceding definition will be used whenever $G$ is a finite set.

Given a square $Q$, it can be decomposed into its dyadic children as follows: $Q^1_0,Q^1_1,Q^1_2,Q^1_3$ are disjoint except for their boundaries, their union is $Q$, and $\delta_{Q^1_i} = \frac{1}{2} \delta_Q$. We say that $Q$ is the parent of $Q^1_i$, or equivalently that $Q^1_i$ is a child of $Q$. As one continues to cut the children of some base square $Q^\circ$, we form a quadtree whose leaves are a collection of squares that are disjoint except for their boundaries, and with union equal to $Q^\circ$. Let $\CZ$ be a labeling of these squares. If $Q_\nu \cap Q_{\nu'} \neq \emptyset$ we say that $Q_\nu$ and $Q_{\nu'}$ are neighbors, and often write $Q_\nu \leftrightarrow Q_{\nu'}$ (or rather $\nu \leftrightarrow \nu'$). Note that our definition of neighbor allows a square to be neighbors with itself. Unless we are in the trivial setting where $\CZ = \{Q^\circ\}$, every square $Q \in \CZ$ arose from bisection of its parent. We denote by $Q^+$ the parent of $Q$.

We denote the space of affine functions by $P=\{A \cdot x+b: A \in \R^2,b \in \R\}$, and the space of \textit{Whitney Fields} on $E_1 \subset \R^2$ by $Wh(E_1) = \{(L_x)_{x \in E_1}\}$ (with each $L_x \in P$). Given $F \in C^1(\R^2)$, and $x_0 \in \R^2$, we define the 1-jet of $F$ at $x_0$ by 
$$J_{x_0} F = F(x_0) + \nabla F(x_0) \cdot (x-x_0),$$
as well as the 1-jet of $F$ on $E_0$ by 
$$J_{E_0} F = (J_x F)_{x \in E_0} \in Wh(E_0),$$
furthermore we say that $F$ is an interpolant of $f:E \rightarrow \R$ if and only if $F|_E = f$.

Analogously, for $G \in C^1(\R)$, we define the 1-jet of $G$ at $x_0 \in \R$ by $J_{x_0} G = G(x_0) + G'(x_0) (x-x_0)$.

Suppose $X$ is some vector space of functions defined on $\Omega \subset \R^2$ equipped with a semi-norm $\|\cdot\|_X$, and $F \in X$ is given. For a universal constant $C>0$, we say that $\|F\|_X$ is  C-optimal with respect to properties $p_1,p_2,\ldots,p_m$ if $F$ satisfies $p_1,p_2,\ldots,p_m$, and also 
$$\|F\|_X \leq C \inf\{\|G\|_X : G \; \mbox{satisfies} \; p_1,p_2,\ldots,p_m\}.$$
A C-optimal $F \in X$ satisfying $p_1,p_2,\cdots$ is to be understood in the same way.

For a domain $\Omega \subset \R^2$, we define the homogeneous Sobolev space $L^{2,p}(\Omega)$, as well as the non-homogeneous $W^{2,p}(\R^2)$, to consist of functions $F:\Omega\rightarrow \R$ for which the respective semi-norm is finite:
\begin{align*}
& \|F\|_{L^{2,p}(\Omega)} := \|\nabla^2 F\|_{L^p(\Omega)} < \infty,\\
& \|F\|_{W^{2,p}(\Omega)} := \|F\|_{L^p(\Omega)} + \|\nabla F\|_{L^p(\Omega)} + \|\nabla^2 F\|_{L^p(\Omega)} < \infty.
\end{align*}
In particular, $\|\cdot\|_{W^{2,p}(\Omega)}$ induces a norm on the Banach space $W^{2,p}(\Omega)$, while $L^{2,p}(\Omega)$ forms a complete semi-normed vector space under $\|\cdot\|_{L^{2,p}(\Omega)}$. The preceding semi-norms can be extended to vector valued mappings $\Phi:\Omega \rightarrow \R^2$ in an obvious way.

For an interval $I \subset \R$ (bounded or unbounded), we define the homogeneous Besov space $\bes(I)$, as well as the non-homogeneous $\besfull(\R)$ to consist of functions $\varphi:I \rightarrow \R$ for which the respective semi-norm is finite:
\begin{align*}
& \|\varphi\|_{\bes(I)} := \left(\int_I\int_I \frac{|\varphi'(x) - \varphi'(y)|^p}{|x-y|^p} dx dy \right)^{1/p} < \infty, \\
& \|\varphi\|_{\besfull(I)} := \|\varphi\|_{L^p(I)} + \|\varphi'\|_{L^p(I)} + \|\varphi\|_{\bes(I)} < \infty.
\end{align*}
As before, $\besfull(I)$ (resp. $\bes(I)$) forms a Banach (resp. complete semi-normed) space under $\|\cdot\|_{\besfull(I)}$ (resp. $\|\cdot\|_{\bes(I)}$).

We now define a geometric quantity which is used to measure the flatness of subsets of $\R^2$. In particular, this notion will be specially adapted to the problem of extension within a Sobolev space in the following way: It will be the case that subsets of $\R^2$ for which this quantity is small will have corresponding Sobolev extension problems that are easy to solve. Given an arbitrary set $\Omega \subset \R^2$, we define the Besov semi-norm of $\Omega$ by:
\begin{equation*}
\|\Omega\|_{\bes} = \inf \{\|\varphi\|_{\bes(\R)}: (z_1,z_2) \; \mbox{Euclidean coordinate system}, \; \Omega \subset \{(z_1,\varphi(z_1))_{z_1z_2} : z_1 \in \R\} \}.
\end{equation*}
\begin{rek} When $\Omega$ is not contained in the graph of any Besov function as above,  $\|\Omega\|_{\bes} = \infty$. Alternatively, when $\Omega$ lies on a line, $\|\Omega\|_{\bes} = 0$.
\end{rek}
Given a square $Q \subset \R^2$, a finite set $E_0 \subset Q$, a function $f:E_0 \rightarrow \R$, a point $x_0 \in Q$, and a real number $M \geq 0$. We define the local $\Gamma$'s and $\sigma$'s by:
\begin{align*}
&\Gamma_Q(f,x_0,M) = \{J_{x_0}F : F|_{E_0} = f, \; \|F\|_{L^{2,p}(Q)} \leq M\}, \; \mbox{and}\\
&\sigma_Q(E_0, x_0) = \{J_{x_0}F :F|_{E_0} = 0, \; \|F\|_{L^{2,p}(Q)} \leq 1\};
\end{align*}
as well as the global $\Gamma$'s and $\sigma$'s:
\begin{align*}
&\Gamma(f,x_0,M) = \{J_{x_0}F : F|_E = f, \; \|F\|_{L^{2,p}(\R^2)} \leq M\}, \; \mbox{and}\\
&\sigma(E_0,x_0) = \{J_{x_0}F : F|_{E_0} = 0, \; \|F\|_{L^{2,p}(\R^2)} \leq 1\}.
\end{align*}
We set aside $0 < c_1, c_2, \cdots < 1/100$ for small universal constants whose precise values are fixed throughout the paper. We use $c,C,\tilde{c},\tilde{C}, C_1, C_2, \cdots$ for universal constants whose values are independent of $c_i$. Unless otherwise stated, the values of these constants may change from one occurrence to the next, though we promise to fix their meaning within a particular theorem, lemma, or section in order to avoid confusion. We always use capitalized letters for ``large'' constants, while uncapitalized letters are to denote constants that are sufficiently small. 

\section{Background Material}

We start by recalling the well known Sobolev embedding theorem, which we state for the function space $L^{2,p}(Q)$, with $Q$ a square, and $p>2$. 

\begin{lem}[Sobolev embedding theorem (SET)]
\label{sobimbthm}
Let $Q$ be a square, and $x_0 \in Q$. For $F \in L^{2,p}(Q)$ we define $L^\circ = J_{x_0} F$. Then for $x \in Q$,
\begin{align*}
&|\nabla F(x) - \nabla F(x_0)| \lesssim |x-x_0|^{1-2/p} \|F\|_{L^{2,p}(Q)};  \\
&|F(x) - L^\circ(x)| \lesssim |x-x_0|^{2-2/p} \|F\|_{L^{2,p}(Q)}; \\
&\|\nabla F - \nabla F(x_0)\|_{L^{p}(Q)} \lesssim \delta_Q \|F\|_{L^{2,p}(Q)}; \\
&\|F - L^\circ\|_{L^{p}(Q)} \lesssim \delta^2_Q \|F\|_{L^{2,p}(Q)}.
\end{align*}
\end{lem}

\begin{rek}
Recall that for $0 < \alpha \leq 1$, the semi-norm given by
$$\|F\|_{\dot{C}^{1,\alpha}(Q)} = \sup_{x,y \in Q} \frac{|\nabla F(x) - \nabla F(y)|}{|x-y|^\alpha}$$
defines the semi-normed complete vector space $\dot{C}^{1,\alpha}$ to be the collection of all continuously differentiable $F$ for which the preceding quantity is finite. Then, the SET is often succinctly stated as the inequality of semi-norms: $\|F\|_{\dot{C}^{1,1-2/p}(Q)} \lesssim \|F\|_{L^{2,p}(Q)}$ for all $F \in L^{2,p}(Q)$. Note that the first two inequalities above follow from this inequality as well as an application of Taylor's Theorem for $\dot{C}^{1,1-2/p}(Q)$ functions, while the last two inequalities follow by integrating the first two inequalities over the square $Q$.
\end{rek}

\begin{rek}
\label{twosquares}
The SET generalizes easily to domains which are the union of two intersecting squares $\Omega = Q_1 \cup Q_2$. For points $x,x_0 \in Q_i$ ($i=1$ or $2$), the first two inequalities extend directly. When $x \in Q_1$ and $x_0 \in Q_2$, we can introduce a third point $x_1 \in Q_1 \cap Q_2$, and use that $|x_1-x_0|,|x_1-x| \leq |x-x_0|$ to deduce similar inequalities.  The third and fourth inequalities follow in the same way, except $\delta_Q$ is replaced by $\mbox{diam}(\Omega) \approx \max\{\delta_{Q_1},\delta_{Q_2}\}$. This argument fails for domains $\Omega= R_1 \cup R_2$, where $R_1,R_2$ are arbitrary intersecting rectangles, e.g., in the case when $\Omega$ looks like a thin 'V', the first two inequalities in Lemma \ref{sobimbthm} fail to hold uniformly over such choices of $\Omega$.
\end{rek}

We will also make use of the Besov embedding theorem for univariate functions $\varphi \in \bes(\R)$. Recall that the homogeneous Besov space arose in our definition of the Besov semi-norm of a set $E \subset \R^2$. Similar in form to the last result, we have the following bounds for Besov functions.

\begin{lem}[Besov embedding theorem (BET)]
Let $I$ be an interval, and $r_0 \in I$. Consider $\varphi \in \bes(I)$, and set $l_0 = J_{r_0} \varphi$. Then for $r \in I$,
\begin{align*}
&|\varphi'(r)-\varphi'(r_0)| \lesssim |r-r_0|^{1-2/p}\|\varphi\|_{\bes(I)}; \\
&|\varphi(r)-l_0(r)| \lesssim |r-r_0|^{2-2/p}\|\varphi\|_{\bes(I)}.
\end{align*}
\end{lem}

The following trace/extension theorem is well known in the literature (e.g., see \cite{T1,M1}), and is the main connection between Sobolev spaces and Besov spaces that we draw upon in this paper.

\begin{prop}[Trace/Extension theorem]
Let $G \in L^{2,p}(\R^2)$, and define $g(x) = G(x,0)$. Then $g \in \bes(\R)$ with $\|g\|_{\bes(\R)} \lesssim \|G\|_{L^{2,p}(\R^2)}$; moreover, if $G \in W^{2,p}(\R^2)$ then $g \in \besfull(\R)$ with $\|g\|_{\besfull(\R)} \lesssim \|G\|_{W^{2,p}(\R^2)}$. Conversely, there exists a linear extension operator $T_1: \bes(\R) \rightarrow L^{2,p}(\R^2)$ satisfying
\begin{enumerate}
\item $T_1 g (x,0) = g(x,0)$, for all $x \in \R$;
\item $\|T_1 g\|_{L^{2,p}(\R^2)} \lesssim \|g\|_{\bes(\R)}$;
\item $\|T_1 g\|_{W^{2,p}(\R^2)} \lesssim \|g\|_{\besfull(\R)}$.
\end{enumerate}
\label{trthm}
\end{prop}

We now present a technical lemma relating to the definition of the Besov semi-norm of a set: Through rotating coordinates appropriately, we can arrange a set with small Besov semi-norm to lie on the graph of a function with small Besov \textit{norm}. We use this Lemma in our proof of the implicit function theorem.

\begin{lem}
\label{nicecurve}
There exists a universal constant $c>0$ so that for any real numbers $0 \leq \kappa_1,\kappa_2 \leq c$ the following holds: Let $\Omega \subset \R^2$ be an arbitrary set with $\mbox{diam}(\Omega) \lesssim 1$, and $\|\Omega\|_{\bes} \leq \kappa_1$. Let $(u,v)$ be Euclidean coordinates on $\R^2$, with the following property: if $\#(\Omega) \geq 2$, then there exist distinct $x_0,y_0 \in \Omega$ with $|v(x_0)|,|v(y_0)|, |v(y_0)-v(x_0)|/|u(y_0)-u(x_0)| \leq \kappa_2$, whereas for $\#(\Omega) \leq 1$ we have $|v(x_0)| \leq \kappa_2$ if $x_0 \in \Omega$. Then there exists a function $\tilde{\varphi} \in \besfull(\R)$ with 
\begin{enumerate}
\item $\|\tilde{\varphi}\|_{\besfull(\R)} \lesssim \kappa_1 + \kappa_2$, and
\item $\Omega \subset \{(u,\tilde{\varphi}(u)) : u \in \R\}$.
\end{enumerate}
\end{lem}

\textit{Remark on proof of Lemma \ref{nicecurve}:} The proof is tedious but straightforward. Here we give a sketch of the proof: One fixes a curve $\gamma = \{(s,\varphi(s))_{st} : s \in \R\}$ with $\Omega \subset \gamma$ and $\|\varphi\|_{\bes(\R)} \lesssim \kappa_1$, and then calculates the Besov semi-norm of $\varphi_1$ upon writing $\gamma=\{(u,\varphi_1(u))_{uv}: u \in \R\}$. We find that $\|\varphi_1\|_{\bes(\R)} \lesssim \|\varphi\|_{\bes(\R)} \lesssim \kappa_1$ as long as $\varphi'_1(u)$ remains uniformly bounded by an absolute constant. Now, from the hypotheses relating $\Omega$ to $(u,v)$, $\varphi_1$ must have slope and value smaller than $\kappa_1 + \kappa_2$ in a fixed neighborhood surrounding $\Omega$. After choosing an appropriate unit cutoff function $\theta$ and setting $\tilde{\varphi} = \varphi_1 \theta$, we find that $\tilde{\varphi}$ has uniformly bounded slope and value, bounded Besov semi-norm, support contained within a unit-scale interval, and still interpolates the set $\Omega$. The result follows.

The well known implicit function theorem for the continuous class of spaces $C^m(\R^n)$ has a quantitative analogue for Sobolev functions. We would like to thank Kevin Luli for helpful ideas concerning the proof of this lemma. We first require a form of the inverse function theorem. For a $2\times2$ matrix $M$, $|M|$ will denote the maximal absolute value of the entries of $M$.

\begin{lem}[Inverse function theorem]
\label{ift}
Let $\Omega \subset \R^2$ be a domain, and $\Phi = (\Phi^1,\Phi^2) : \Omega \rightarrow \R^2$ with $\Phi^i \in L^{2,p}(\Omega)$ be given. Suppose that $\Phi$ is injective, and that $|\nabla \Phi(x)|$, $|(\nabla \Phi(x))^{-1}| \lesssim 1$ for all $x \in \Omega$. Then $\Phi^{-1} \in L^{2,p}(\Phi(\Omega))$ with $\|\Phi^{-1}\|_{L^{2,p}(\Phi(\Omega))} \lesssim \|\Phi\|_{L^{2,p}(\Omega)}$.
\end{lem}

\textit{Remark on proof of Lemma \ref{ift}:} Upon differentiating the identity $\Phi^{-1} \circ \Phi(x) = x$ twice, solving the resulting equation for the Hessian of $\Phi^{-1}$, and using the boundedness of $|(\nabla \Phi)^{-1}|$, we find that $|\nabla^2 \Phi^{-1}(x)| \lesssim |\nabla^2 \Phi(\tilde{x})|$ for $\tilde{x} = \Phi^{-1}(x)$. After raising both sides to the $p$'th power, integrating, and using the boundedness of the Jacobian of $\Phi$, the conclusions of Lemma \ref{ift} follow.

The implicit function theorem and proceeding lemma explain a certain duality between Sobolev functions and Besov curves: Besov curves are precisely the level sets of Sobolev functions.

\begin{lem}[Implicit function theorem]
\label{impft}
Let $Q = [-1/2,1/2]^2$, and fix a point $x_0 \in 0.9 Q$. There exists a sufficiently small universal constant $c>0$ so that the following holds: Let $h \in L^{2,p}(Q)$ satisfy $\|h\|_{L^{2,p}(Q)} \leq c$ and $|\nabla h(x_0)| \geq 1$. Consider $\gamma = \{x \in 0.9 Q: h(x) = 0\}$. Then $\|\gamma\|_{\bes} \lesssim \|h\|_{L^{2,p}(Q)}$. Conversely, suppose that $\gamma \subset 0.9 Q$ is given with $\|\gamma\|_{\bes} \leq c$. Then there exists a function $h \in L^{2,p}(Q)$ satisfying $h|_{\gamma} = 0$, $\|h\|_{L^{2,p}(Q)} \lesssim \|\gamma\|_{\bes}$, and $|\nabla h(x_0)| \geq 1$.
\end{lem}

\begin{proof}
The lemma is trivial if $\#(\gamma) \leq 1$, and so we assume without loss that $\#(\gamma) \geq 2$. We now establish the first half of the lemma: Upon rescaling by a factor of $1/|\nabla h(x_0)|$ (a factor less than $1$), without loss the given $h \in L^{2,p}(Q)$ satisfies: $(1)$ $\|h\|_{L^{2,p}(Q)} \leq c$, and $(2)$ $|\nabla h(x_0)| = 1$. We now fix $A = \|h\|_{L^{2,p}(Q)} \leq c$, with $c$ a sufficiently small universal constant. We now proceed to bound the Besov semi-norm of the zero set of $h$ by $A$.

Let $\theta \in C_c^\infty(Q)$ satisfy the following properties: $(1)$ $\theta \equiv 1$ on $0.9 Q$, and $(2)$ $|\partial^\alpha \theta| \lesssim 1$ for all $|\alpha| \leq 2$. We set $L_0 = J_{x_0} h$, and through the SET find that $\|L_0 - h\|_{W^{2,p}(Q)} \lesssim A$. Define $\tilde{h} = \theta h + (1- \theta) L_0 = L_0 + \theta (h - L_0)$, for which: $(1)$ $\tilde{h}|_{0.9Q} = h|_{0.9Q}$, $(2)$ $\|\tilde{h}\|_{L^{2,p}(\R^2)} \lesssim A$, and $(3)$ $|\nabla \tilde{h}(x) - \nabla \tilde{h}(x_0)| \leq \frac{1}{10}$ for all $x \in \R^2$. Where $(3)$ follows for $x \in Q$ by the SET, and for $x \notin Q$ by the fact that $\nabla \tilde{h}(x) = \nabla h(x_0) = \nabla \tilde{h}(x_0)$ by construction.

Denote $e_2 = \nabla \tilde{h}(x_0)$, and let $e_1 = e_2^\perp$ be a unit vector perpendicular to $e_2$. Define Euclidean coordinates by $(u,v)_{uv} = x_0 + ue_1 + ve_2 \in \R^2$. We now consider the mapping $\Phi : \R^2 \rightarrow \R^2$ given by 
\begin{equation}
\Phi(u,v) = (u,\tilde{h}(u,v))_{uv} 
\label{formofphi} 
\end{equation}
We claim that $\Phi$ is invertible: From $P2$ of $\tilde{h}$, $\|\Phi\|_{L^{2,p}(\R^2)} = \|\tilde{h}\|_{L^{2,p}(\R^2)} \lesssim A$. While, the definition of $\Phi$ and $P3$ of $\tilde{h}$ imply
\begin{equation}
|\nabla \Phi(x) - Id| = |\nabla \Phi(x) - \nabla \Phi(x_0)| = |\nabla \tilde{h}(x) - \nabla \tilde{h}(x_0)| \leq \frac{1}{10},
\label{closeid}
\end{equation}
for all $x \in \R^2$. Thus, $|\nabla \Phi(x)|, |(\nabla \Phi(x))^{-1}| \leq 10$ for all $x \in \R^2$. Moreover, from \eqref{closeid}, and the form of $\Phi$ in \eqref{formofphi}, it follows that $\Phi$ is bijective onto $\R^2$. Having verified the hypotheses of the inverse function theorem, we find  $\|\Phi^{-1}\|_{L^{2,p}(\R^2)} \lesssim A$. Also, from $P1$ of $\tilde{h}$,
\begin{align*}
\gamma & = \{x \in 0.9 Q: \tilde{h}(x)=0\} \subset \{x \in \R^2 : \tilde{h}(x)=0 \} = \{(u,v) \in \R^2: \Phi(u,v) = (u,0) \} \\
& = \{\Phi^{-1}(t,0): t \in \R \} = \{(t,\varphi(t)) : t \in \R \},
\end{align*}
with $\varphi(t) = (\Phi^{-1})^2(t,0)$. Thus, $\|\varphi\|_{\bes(\R)} \lesssim \|(\Phi^{-1})^2\|_{L^{2,p}(\R^2)} \lesssim A$. By the definition of Besov semi-norm of a set, we find $\|\gamma\|_{\bes} \lesssim A$, and the first half of the lemma is proven.

For the second half of the lemma: Let $\gamma$ be given with $\gamma \subset 0.9Q$ and $\|\gamma\|_{\bes} \leq c$ for $c$ sufficiently small. Let $A = \|\gamma\|_{\bes}$. Since $\#(\gamma) \geq 2$, we may fix distinct $x_0,y_0 \in \gamma$, and choose Euclidean coordinates $(u,v)$ so that $v(x_0) = v(y_0) = 0$ . Therefore, $(u,v)$ satisfy the hypothesis of Lemma \ref{nicecurve}. By choosing $c$ sufficiently small, Lemma \ref{nicecurve} applies and we find $\varphi \in \besfull(\R)$ with $\gamma \subset \{(u,\varphi(u)) : u \in \R \}$ and $\|\varphi\|_{\besfull(\R)} \lesssim A$.

Through application of Proposition \ref{trthm}, there exists $\phi \in W^{2,p}(\R^2)$ satisfying: $(1)$ $\phi(u,0) = \varphi(u)$ for $u \in \R$, and $(2)$ $\|\phi\|_{W^{2,p}(\R^2)} \lesssim A$. From the SET, $|\nabla \phi(x)| \lesssim \|\phi\|_{W^{2,p}(\R^2)} \lesssim A$ for any $x \in \R^2$. By choosing $c$ sufficiently small, $A$ can be taken small enough so that $|\nabla \phi(x)| \leq \frac{1}{10}$ for all $x \in \R^2$. Define $\Phi : \R^2 \rightarrow \R^2$ by $\Phi(u,v) = (u, v + \phi(u,v))$, for which
\begin{enumerate}
\item $\{\Phi(u,0):u \in \R\} = \{(u, \phi(u,0)):u \in \R\} = \{(u,\varphi(u)): u \in \R\} \supset \gamma$;
\item $|\nabla \Phi - Id| \leq \frac{1}{10}$, and thus
\item $|\nabla \Phi(x)|, \; |(\nabla \Phi(x))^{-1}| \leq 10$ for all $x \in \R^2$;
\item $\|\Phi\|_{L^{2,p}(\R^2)} = \|\phi\|_{L^{2,p}(\R^2)} \lesssim A$.
\end{enumerate}
As before, these properties and the definition of $\Phi$ imply that $\Phi$ is bijective. Thus, from the inverse function theorem we also find $(5)$ $\|\Phi^{-1}\|_{L^{2,p}(\R^2)} \lesssim A$. 

Let $h_1(u,v) = (\Phi^{-1})^2(u,v)$. $P1$ of $\Phi$ implies $h_1(x) = 0$ for any $x \in \gamma$, $P2$ of $\Phi$ implies $|\nabla h_1(x_0)| \geq \frac{1}{2}$, and $P5$ of $\Phi$ implies $\|h_1\|_{L^{2,p}(Q)} \leq \|\Phi^{-1}\|_{L^{2,p}(Q)} \lesssim A$. Thus, the function $h = 2h_1$ satisfies the desired properties, and the second half of the lemma has been proven.
\end{proof}

From the proof of the second half of the implicit function theorem, the map $\Phi^{-1}:\R^2 \rightarrow \R^2$ was shown to satisfy certain properties which we list in the following lemma (notice that $\Phi$ below corresponds to $\Phi^{-1}$ from the proof).

\begin{lem}[Straightening lemma]
Let $Q=[-1/2,1/2]^2$. There exists a sufficiently small universal constant $c>0$ so that the following holds: Let $\gamma \subset 0.9 Q$ be given with $\|\gamma\|_{\bes} \leq c$. There exist Euclidean coordinates $(u,v)$ on $\R^2$ and a diffeomorphism $\Phi : \R^2 \rightarrow \R^2$ so that
\begin{align*}
&\Phi(u,v) = (u,0) \; \mbox{for all} \; (u,v) \in \gamma, \\
&\|\Phi\|_{L^{2,p}(\R^2)} \lesssim \|\gamma\|_{\bes}, \\
&\|\Phi^{-1}\|_{L^{2,p}(\R^2)} \lesssim  \|\gamma\|_{\bes}, \mbox{and} \\
&|\nabla \Phi(x)|, |\nabla \Phi^{-1}(x)| \leq 10 \; \mbox{for all} \; x \in \R^2. \\
\end{align*}
\label{strlem}
\end{lem}

Such diffeomorphisms preserve the inhomogeneous $W^{2,p}(\R^2)$ Sobolev norm, as stated in the following result.

\begin{lem}
Suppose that $F \in W^{2,p}(\R^2)$, and $\Phi: \R^2 \rightarrow \R^2$ is a diffeomorphism with $\|\Phi\|_{L^{2,p}(\R^2)} \lesssim 1$, and $|\nabla \Phi|$, $|\nabla \Phi^{-1}| \lesssim 1$. Then $F \circ \Phi \in W^{2,p}(\R^2)$ with $\|F \circ \Phi\|_{W^{2,p}(\R^2)} \approx \|F\|_{W^{2,p}(\R^2)}$.
\label{sobdiff}
\end{lem}
\begin{proof}
We compute the second partials of $F \circ \Phi = F(\Phi^1,\Phi^2)$, to find that
\begin{equation}
\partial_{ij} (F \circ \Phi) = \sum_{k,l \in \{1,2\}} c_{kl} \partial_i \Phi^k \; \partial_j \Phi^l \; \partial_{kl} F \circ \Phi + \sum_{k \in \{1,2\}} \partial_{ij} \Phi^k \; \partial_k F \circ \Phi,
\label{hess1}
\end{equation}
with $c_{kl} \in \R$ independent of $F$ and $\Phi$. The SET implies that $\|\nabla F\|_{L^{\infty}(\R^2)} \lesssim \|F\|_{W^{2,p}(\R^2)}$. Raising both sides to the $p$'th power, integrating, and applying this fact along with the hypotheses on $\Phi$ yields:
\begin{equation*}
\|\nabla^2(F \circ \Phi)\|^p_{L^{p}(\R^2)} \lesssim \|(\nabla^2 F) \circ \Phi \|^p_{L^p(\R^2)} + \|F\|^p_{W^{2,p}(\R^2)}.
\end{equation*}
After changing variables using $\tilde{x} = \Phi(x)$, and noting that Jacobian of this coordinate change is bounded thanks to the assumption that $|(\nabla \Phi)^{-1}| \lesssim 1$, we find that $\|\nabla^2 (F \circ \Phi)\|_{L^p(\R^2)} \lesssim \|F\|_{W^{2,p}(\R^2)}$. In the same way we can bound the lower order derivatives of $F \circ \Phi$, and so $\|F \circ \Phi\|_{W^{2,p}(\R^2)} \lesssim \|F\|_{W^{2,p}(\R^2)}$. Finally, the assumptions on $\Phi$ imply through the Inverse Function Theorem that $\Phi^{-1}$ satisfies the same, and thus the above argument shows that $\|F\|_{W^{2,p}(\R^2)} \lesssim |F \circ \Phi\|_{W^{2,p}(\R^2)}$.
\end{proof}

In this paper we reduce a two dimensional Sobolev extension problem to a family of one dimensional Besov extension problems, which can be solved thanks to the next proposition.

\begin{prop}[1D linear Besov extensions]
Suppose that $E_1 \subset \R$ is finite, with $\mbox{diam}(E_1) \lesssim 1$, and $g:E_1 \rightarrow \R$ is given. There exists a bounded linear extension operator $T_b: \besfull(\R)|_{E_1} \rightarrow \besfull(\R)$, and linear functionals $\{\lambda_i(g)\}_{i=1}^{N_0}$ with $N_0 \lesssim (\#E_1)^2$ so that
\begin{enumerate}
\item $T_b g |_{E_1} = g$
\item $\|T_b g \|^p_{\besfull(\R)} \approx \|g\|^p_{\besfull(\R)|_{E_1}}\approx \sum_{i} |\lambda_i(g)|^p$
\end{enumerate}
\label{lbe}
\end{prop}

\begin{proof}

If $\#(E_1) \leq 1$ then the proposition is trivial, thus we may suppose that $\#(E_1) \geq 2$. We write $E_1 = \{x_1,\cdots,x_N\}$.

For each $x_k \in E_1$, define $x_{\nu(k)} \in E_1$ to be a nearest neighbor of $x_k$. For $1 \leq k \leq m$, define $m_k = \frac{g(x_k) - g(x_{\nu(k)})}{x_k - x_{\nu(k)}}$, $L_k(x) = g(x_k) + m_k(x-x_k)$, and $I_k = [x_k,x_{k+1}]$. Additionally, define $I_0 = (-\infty, x_1]$, $I_{N} = [x_N,\infty)$, and $\Delta_k = |x_k - x_{k+1}|$ for $1 \leq k \leq N-1$. From the classical Whitney extension theorem (see \cite{St1}), it is simple to find $F_k \in \dot{C}^{1,1}(I_k)$ ($0 \leq k \leq N)$ with
\begin{itemize}
\item $F_k(x) = L_k(x)$ for $x_k \leq x \leq x_k + \frac{1}{10}\Delta_k$, and $F_k(x) = L_{k+1}(x)$ for $x_{k+1} - \frac{1}{10}\Delta_{k} \leq x \leq x_{k+1}$;
\item $\|F_k\|_{\dot{C}^{1,1}(I_k)}$ is C-optimal with respect to the above property ($1 \leq k \leq N-1)$;
\item $F_k$ depends linearly on $g$;
\item $F_0(x) = f(x_1) + m_1(x-x_1) \; \mbox{for} \;x \in I_0$, and $F_N(x) = f(x_N) + m_N(x-x_N) \; \mbox{for} \; x \in I_N$.
\end{itemize}
Recall that $\dot{C}^{1,1}(I)$ is the space of functions semi-normed by $\|F\|_{\dot{C}^{1,1}} = \ \sup_{x,y \in I} |F'(x)-F'(y)|/|x-y|$. The classical Whitney extension theorem also gives a formula for $M_k \geq 0$ ($0 \leq k \leq N$) that satisfies $M_k \approx \|F_k\|_{\dot{C}^{1,1}(I_k)}$:
\begin{align}
\label{piecenorm} 
M_k & = |m_{k+1} - m_k|\Delta_k^{-1} + |L_k(x_{k+1}) - g(x_{k+1})| \Delta_k^{-2}, \; \mbox{for} \; 1 \leq k \leq N-1; \\
M_0 & = M_{N} = 0. \notag{}
\end{align}
Now, define $F \in C^{1,1}(\R)$ by $F(x) = F_k(x)$ for $x \in I_k$. Let $\displaystyle A_{kl} = \int_{I_k} \int_{I_l} \frac{1}{|x-y|^{p}} dx dy$, and set
$$M^p = \sum_{k=1}^{N-1} M^p_k \Delta^{2}_k + \sum_{0 \leq k < l \leq N} |m_{k+1} - m_l|^p A^p_{kl}.$$
\textbf{Claim 1:} $\|F\|_{\bes(\R)} \lesssim M$

To prove Claim 1 we must bound
\begin{equation}
\|F\|^p_{\bes(\R)} = \sum_{k=0}^{N} \int_{I_k} \int_{I_k} \frac{|F'(x)-F'(y)|^p}{|x-y|^p} dx dy + 2\sum_{0 \leq k < l \leq N} \int_{I_k} \int_{I_l} \frac{|F'(x) - F'(y)|^p}{|x-y|^p} dx dy.
\label{beseq}
\end{equation}

We now analyze each sum in \eqref{beseq} separately, initially focusing on the first. By the Lipschitz control on the derivative of $F=F_k$ on $I_k$ ($0 \leq k \leq N$),
\begin{equation}
\sum_{k=0}^{N} \int_{I_k} \int_{I_k} \frac{|F'(x)-F'(y)|^p}{|x-y|^{p}} dx dy \lesssim \sum_{k=1}^{N-1} M^p_k \Delta^2_k.
\label{besineq1}
\end{equation}
Now, for an individual term from the second sum in \eqref{beseq},
\begin{align*}
\int_{I_k} \int_{I_l} \frac{|F'(x) - F'(y)|^p}{|x-y|^p} dy dx & \lesssim \int_{I_k} \int_{I_l} \frac{|F'(x_{k+1}) - F'(x_l)|^p}{|x-y|^p} dy dx + \int_{I_k} \int_{I_l} \frac{|F'(x) - F'(x_{k+1})|^p}{|x-y|^p} dy dx \\
& + \int_{I_k} \int_{I_l} \frac{|F'(y) - F'(x_l)|^p}{|x-y|^p} dy dx \\
& \lesssim |m_{k+1} - m_l|^p A_{kl} + M^p_k \Delta^p_k \int_{x_k}^{x_{k+1} - \frac{1}{10}\Delta_k} \int_{I_l} \frac{1}{|x-y|^p} dy dx \\
& + M^p_l \Delta^p_l \int_{x_l + \frac{1}{10}\Delta_l}^{x_{l+1}} \int_{I_k} \frac{1}{|x-y|^p} dx dy.
\end{align*}
Here, we have used  that $F(x) = L_{k+1}(x)$ for $x \in [x_{k+1} - \frac{1}{10}\Delta_k,x_{k+1}]$, and $F(x) = L_{l}(x)$ for $x \in [x_{l},x_{l} + \frac{1}{10}\Delta_l]$. Summing over $0 \leq k < l \leq N$ we obtain
\begin{equation}
\label{besineq2}
\sum_{0 \leq k < l \leq N} \int_{I_k} \int_{I_l} \frac{|F'(x) - F'(y)|^p}{|x-y|^p} dx dy \lesssim \sum_{0 \leq k < l \leq N} |m_{k+1}-m_l|^p A_{kl} + \sum_{k=0}^{N-1} M^p_k \Delta^2_k + \sum_{l=1}^{N} M^p_l \Delta^2_l.
\end{equation}
Claim 1 now follows from \eqref{beseq}, \eqref{besineq1}, and \eqref{besineq2}.

\textbf{Claim 2:} Suppose that $F \in \bes([a,b])$. Then $$\int_a^b \frac{|F'(x) - F'(a)|^p}{|x-a|^{p-1}} dx \lesssim \|F\|^p_{\bes(I)}.$$

To prove Claim 2, we may suppose without loss of generality that $[a,b] = [0,1]$ through scale invariance. For $k \geq 0$, define $I'_{k} = [2^{-k-1},2^{-k}]$. Notice that $(0,1] = \bigcup I'_k$, and the intervals $I'_k$ intersect only at their endpoints. We write 
\begin{align*}
\int_0^1 \frac{|F'(x) - F'(0)|^p}{x^{p-1}} dx \lesssim \sum_{k \geq 0} \int_{I'_k} \frac{|F'(2^{-k}) - F'(0)|^p}{x^{p-1}} dx + \sum_{k \geq 0} \int_{I'_k} \frac{|F'(x) - F'(2^{-k})|^p}{x^{p-1}} dx.
\end{align*}
For any $x \in I'_k$, we have $x \approx 2^{-k}$. From the BET, $|F'(x) - F'(2^{-k})|^p \lesssim \|F\|^p_{\bes(I'_k)} 2^{-k(p-2)}$ for $x \in I'_k$. Thus,
\begin{align*}
\int_0^1 \frac{|F'(x) - F'(0)|^p}{x^{p-1}} dx & \lesssim \sum_{k \geq 0} |F'(2^{-k}) - F'(0)|^p 2^{-k(2-p)} + \sum_{k \geq 0} \|F\|^p_{\bes(I'_k)} \lesssim \sum_{k \geq 0} |F'(2^{-k}) - F'(0)|^p 2^{-k(2-p)} \\
& + \|F\|_{\bes([0,1])}.
\end{align*}
Finally, we must show that $ \sum_{k \geq 0} |F'(2^{-k}) - F'(0)|^p 2^{k(p-2)} \lesssim \|F\|_{\bes([0,1])}$, which will finish off the proof of Claim 2. From $F \in \bes([0,1]) \subset \dot{C}^{1,\alpha}([0,1])$, we find $\displaystyle \lim_{x \rightarrow 0} F'(x) = F'(0)$. Thus, $F'(2^{-k}) - F'(0) = \sum_{l \geq k} \left[ F'(2^{-l}) - F'(2^{-l-1}) \right]$. Let $\epsilon < \frac{p-2}{p}$, and through H\"older's inequality we bound
\begin{align*}
\sum_{k \geq 0} |F'(2^{-k}) - F'(0)|^p 2^{k(p-2)} & = \sum_{k \geq 0} \left| \sum_{l \geq k} (F'(2^{-l}) - F'(2^{-l-1})) 2^{l \epsilon} 2^{-l \epsilon} \right|^p 2^{k(p-2)} \\
& \leq \sum_{k \geq 0}  2^{k(p-2)}  \sum_{l \geq k} |F'(2^{-l}) - F'(2^{-l-1})|^p 2^{l \epsilon p}  \left[ \sum_{l \geq k} 2^{-l p' \epsilon} \right]^{p/p'} \\
& \lesssim \sum_{k \geq 0}  2^{k(p-2)}  \sum_{l \geq k} |F'(2^{-l}) - F'(2^{-l-1})|^p 2^{l \epsilon p}  2^{-k p \epsilon} \\
& = \sum_{l \geq 0} |F'(2^{-l}) - F'(2^{-l-1})|^p 2^{l \epsilon p} \sum_{k \leq l} 2^{k(p-2) - k p \epsilon} \\
& \lesssim \sum_{l \geq 0} |F'(2^{-l}) - F'(2^{-l-1})|^p 2^{l(p-2)} \lesssim \sum_{l \geq 0} \|F\|^p_{\bes(I'_l)} \leq \|F\|^p_{\bes([0,1])},
\end{align*}
which completes the proof of Claim 2.

Having already established Claim 1, the proof of Proposition \ref{lbe} will nearly be complete once we establish Claim 3 below.

\textbf{Claim 3:} Given any $\tilde{F} \in \bes(\R)$ with $\tilde{F}|_{E_1} = g$, we have $M \lesssim \|\tilde{F}\|_{\bes(\R)}$, and thus $M \lesssim \|g\|_{\bes(\R)|_{E_1}}$.

For $1 \leq k \leq N$, define $\delta_k = |x_k - x_{\nu(k)}|$. Note that $\delta_k, \delta_{k+1} \leq \Delta_k$ for each $1 \leq k \leq N-1$. Also, let $J_k$ be the interval with endpoints $x_k$ and $x_{\nu(k)}$. Recall that $m_k = \frac{g(x_k) - g(x_{\nu(k)})}{x_k - x_{\nu(k)}} = \frac{\tilde{F}(x_k) - \tilde{F}(x_{\nu(k)})}{x_k - x_{\nu(k)}}$. The mean value theorem implies the existence of $x_k^* \in J_k$ with $\tilde{F}'(x_k^*) = m_k$. From the BET,
\begin{equation}
|\tilde{F}'(x_k) - m_k| = |\tilde{F}'(x_k) -\tilde{F}'(x^*_k)| \lesssim \|\tilde{F}\|_{\bes(J_k)} (\delta_k)^\alpha.
\label{closeder}
\end{equation}
We now examine a single term in first sum in the definition of $M^p$. Recall \eqref{piecenorm}, which implies for $1 \leq k \leq N-1$ that
\begin{align*}
M_k & \approx |m_{k+1} - m_k|\Delta_k^{-1} + |L_k(x_{k+1}) - g(x_{k+1})| \Delta_k^{-2} \lesssim |m_{k+1} - \tilde{F}'(x_{k+1})| \Delta_k^{-1} \\
& + |\tilde{F}'(x_k) - \tilde{F}'(x_{k+1})| \Delta^{-1}_k + |m_k - \tilde{F}'(x_k)| \Delta_k^{-1} \\
& + |J_{x_k} \tilde{F}(x_{k+1}) - g(x_{k+1})| \Delta_k^{-2} + |J_{x_k} \tilde{F} (x_{k+1}) - L_k(x_{k+1})| \Delta_k^{-2}.
\end{align*}
Using that $|J_{x_k} \tilde{F} (x_{k+1}) - L_k(x_{k+1})| \Delta_k^{-2} = |\tilde{F}'(x_k) -  m_k|\Delta^{-1}_k$, the BET, and \eqref{closeder}, this implies
\begin{align*}
M_k & \lesssim  \|\tilde{F}\|_{\bes(J_{k+1})} \delta_{k+1}^\alpha \Delta_k^{-1} + \|\tilde{F}\|_{\bes(I_k)} \Delta^{\alpha-1}_{k} + \|\tilde{F}\|_{\bes(J_{k})} (\delta_{k})^\alpha \Delta_k^{-1} \\
& \lesssim  \|\tilde{F}\|_{\bes(J_{k+1})} \Delta_k^{\alpha-1} + \|\tilde{F}\|_{\bes(I_k)} \Delta^{\alpha-1}_{k} + \|\tilde{F}\|_{\bes(J_{k})} \Delta_k^{\alpha-1}.
\end{align*}
Thus,
\begin{equation*}
M^p_k \lesssim \Delta_k^{p(\alpha-1)} \left[ \|\tilde{F}\|^p_{\bes(J_{k+1})} + \|\tilde{F}\|^p_{\bes(I_k)} + \|\tilde{F}\|^p_{\bes(J_{k})} \right].
\end{equation*}
Then, summing over $1 \leq k \leq N-1$, and using the fact that $p(\alpha-1) = -2$, we find
\begin{equation}
\sum_{k=1}^{N-1} M^p_k \Delta^{2}_k \lesssim \|\tilde{F}\|_{\bes(\R)}.
\label{byo1}
\end{equation}
Fix $1 \leq k, l \leq N$ with $l \geq k+2$, and consider a term in the second double sum in the expression for $M^p$ : $|m_{k+1} - m_l|^p A^p_{kl}$. For $x \in I_k$, and $y \in I_l$, we have
\begin{align*}
\frac{|m_{k+1} - m_l|^p}{|x-y|^p} & \lesssim \frac{|m_{k+1} - \tilde{F}'(x_{k+1})|^p}{|x-y|^p} + \frac{|\tilde{F}'(x_{k+1}) - \tilde{F}'(x)|^p}{|x-y|^p} + \frac{|\tilde{F}'(x) - \tilde{F}'(y)|^p}{|x-y|^p} \\
& +  \frac{|\tilde{F}'(y) - \tilde{F}'(x_l)|^p}{|x-y|^p} + \frac{|\tilde{F}'(x_l) - m_l|^p}{|x-y|^p}.
\end{align*}
Integrating this over $x \in I_k$ and $y \in I_l$ lends us
\begin{align*}
|m_{k+1} - m_l|^p A^p_{kl} & \lesssim \int_{x_k}^{x_{k+1}} \int_{x_l}^{x_{l+1}} \frac{|m_{k+1} - \tilde{F}'(x_{k+1})|^p}{|x-y|^p} dy dx +  \int_{x_k}^{x_{k+1}} \int_{x_l}^{x_{l+1}} \frac{|\tilde{F}'(x_{k+1}) - \tilde{F}'(x)|^p}{|x-y|^p} dy dx \notag{} \\
& + \int_{x_k}^{x_{k+1}} \int_{x_l}^{x_{l+1}} \frac{|\tilde{F}'(x) - \tilde{F}'(y)|^p}{|x-y|^p} dy dx + \int_{x_k}^{x_{k+1}} \int_{x_l}^{x_{l+1}} \frac{|\tilde{F}'(y) - \tilde{F}'(x_l)|^p}{|x-y|^p} dy dx \\
&  + \int_{x_k}^{x_{k+1}} \int_{x_l}^{x_{l+1}} \frac{|\tilde{F}'(x_l) - m_l|^p}{|x-y|^p} dy dx.
\end{align*}
Thus,
\begin{align}
\sum_{0 \leq k < l \leq N} |m_{k+1} - m_l|^p A^p_{kl} & \lesssim \sum_{k=0}^{N-2} \int_{x_k}^{x_{k+1}} \int_{x_{k+2}}^{\infty} \frac{|m_{k+1} - \tilde{F}'(x_{k+1})|^p}{|x-y|^p} dy dx  \notag{}\\ 
& + \sum_{k=0}^{N-2} \int_{x_k}^{x_{k+1}} \int_{x_{k+2}}^{\infty} \frac{|\tilde{F}'(x_{k+1}) - \tilde{F}'(x)|^p}{|x-y|^p} dy dx \notag{} \\
& + \int_{\R} \int_{\R} \frac{|\tilde{F}'(x) - \tilde{F}'(y)|^p}{|x-y|^p} dy dx + \sum_{l=2}^{N} \int_{x_l}^{x_{l+1}} \int_{-\infty}^{x_{l-1}} \frac{|\tilde{F}'(y) - \tilde{F}'(x_l)|^p}{|x-y|^p} dy dx \notag{} \\
& + \sum_{l=2}^{N} \int_{x_l}^{x_{l+1}} \int_{-\infty}^{x_{l-1}} \frac{|\tilde{F}'(x_l) - m_l|^p}{|x-y|^p} dy dx 
\label{besineq3}
\end{align}
Notice that the first two terms in the above sum are the same as the last two, except with a different index, and a reversed orientation. We now work on establishing a useful bound for the first two terms in \eqref{besineq3} (the same bound will apply for the last two). For the first term, recall \eqref{closeder}, which implies that
\begin{align}
\int_{x_k}^{x_{k+1}} \int_{x_{k+2}}^{\infty} \frac{|m_{k+1} - \tilde{F}'(x_{k+1})|^p}{|x-y|^p} dy dx \lesssim \|\tilde{F}\|^p_{\bes(J_{k+1})} \delta^{p-2}_{k+1} |x_{k+1}-x_k| \cdot |x_{k+1} - x_{k+2}|^{1-p}\leq \|\tilde{F}\|^p_{\bes(J_{k+1})}.
\label{besineq4}
\end{align}
Here, the last inequality follows since $\delta_{k+1} \leq |x_{k+1} - x_{k+2}|, |x_{k+1} - x_k|$. Now, consider a term in the second sum on the RHS of \eqref{besineq3}. Using Claim 2,
\begin{align}
\label{besineq6}
\int_{x_k}^{x_{k+1}} \int_{x_{k+2}}^{\infty} \frac{|\tilde{F}'(x_{k+1}) - \tilde{F}'(x)|^p}{|x-y|^p} & dy dx = C(p) \int_{x_k}^{x_{k+1}} \frac{|\tilde{F}'(x) - \tilde{F}'(x_{k+1})|^p}{|x-x_{k+2}|^{p-1}} dx \\
& \lesssim  \int_{x_k}^{x_{k+1}} \frac{|\tilde{F}'(x) - \tilde{F}'(x_{k+1})|^p}{|x-x_{k+1}|^{p-1}} dx \lesssim \|\tilde{F}\|^p_{\bes(I_k)}. \notag{} 
\end{align}
Thus, from  \eqref{besineq3}, \eqref{besineq4}, and \eqref{besineq6}, we have
\begin{align}
\sum_{0 \leq k < l \leq N} |m_{k+1} - m_l|^p A^p_{kl} & \lesssim \sum_{k=0}^{N-2} \|\tilde{F}\|^p_{\bes(J_{k+1})} +  \sum_{k=0}^{N-2} \|\tilde{F}\|^p_{\bes(I_k)} + \|\tilde{F}\|^p_{\bes(\R)} + \sum_{l=2}^{N} \|\tilde{F}\|^p_{\bes(I_l)}  + \sum_{l=2}^{N}\|\tilde{F}\|^p_{\bes(J_l)} \notag{}\\
& \lesssim \|\tilde{F}\|^p_{\bes(\R)}.
\label{byo2}
\end{align}

Together, \eqref{byo1} and \eqref{byo2} imply Claim 3. The $F$ we constructed obviously satisfies $F|_{E_1} = g$, while Claim 1 and Claim 3 imply $\|F\|_{\bes(\R)} \approx \|g\|_{\bes(\R)|_{E_1}} \approx M$. This is not good enough, since the lemma requires a bound on the inhomogeneous $\besfull(\R)$ norm.

To finish the proof, we define $\widehat{F} = F \theta$, for $\theta \in C^\infty_c(\R)$ satisfying: $(1)$ $|d^k/dx^k \theta| \lesssim 1$ on $\R$ ($k \leq 2$), $(2)$ $\theta \equiv 1$ on $E_1$, and $(3)$ $\mbox{supp} (\theta) \subset [a_1,b_1]$ with $|a_1-b_1|\lesssim 1$. The existence of such a $\theta$ follows since $\mbox{diam}(E_1) \lesssim 1$. From the aforementioned properties of $F$ and $\theta$, along with the BET and mean value theorem,
\begin{enumerate}
\item $\widehat{F}|_{E_1} = g$, and
\item $\|\widehat{F}\|^p_{\besfull(\R)} \lesssim M^p + |g(x_1)-g(x_2)|^p/|x_1-x_2|^p + |g(x_1)|^p$.
\end{enumerate}
Moreover, from the BET and the mean value theorem, each of the terms on the RHS of $P2$ of $\widehat{F}$ is bounded by $\|g\|_{\besfull(\R)|_{E_1}}$. Let $T_b(g) = \widehat{F}$, which satisfies the desired properties.
\end{proof}

\section{A Calder\'on-Zygmund Decomposition}
\label{CZDecomp}
\subsection{OK squares \& CZ squares}
\label{CZ}

Given a finite $E \subset \R^2$, we now fix $Q^\circ \subset \R^2$ to be any square centered at the origin such that $E \subset \frac{1}{10} Q^\circ$. Fix $c_1 > 0$ to be some universal constant, whose precise value is yet to be determined. Our only assumption on $c_1$ is that it is taken to be sufficiently small in order to allow the arguments of this paper go through. A square $Q \subset Q^\circ$ that arises from $Q^\circ$ by repeated bisection will be termed dyadic. A useful concept is that of

\begin{defn}[OK squares]
A dyadic square $Q \subset Q^\circ$ is OK if and only if $$\|3Q \cap E\|_{\bes} \leq c_1 \delta_Q^{2/p-1}.$$
\end{defn}

\begin{rek} From the definition of the Besov semi-norm of a set, for any OK square $Q$ the following is true: Define the rescaled set $\overline{E} = \frac{1}{\delta_Q}(3Q \cap E - c_Q) + c_Q$, then $\|\overline{E}\|_{\bes} \leq c_1$.
\end{rek}

We now decompose $Q^\circ$ into finitely many dyadic squares, pairwise disjoint except for their boundaries, using a \textit{Calder\'on-Zygmund decomposition}.

\begin{description}
\item[CZ Cutting Procedure:] Given a dyadic sub-square $Q \subset Q^\circ$, proceed as follows: If $Q$ is OK, then return the singleton collection $\Lambda_Q=\{Q\}$. Otherwise, return the collection $$\bigcup_{Q' : (Q')^+ = Q} \Lambda_{Q'}.$$
\end{description}

\begin{lem}[Process Terminates]
Let $E \subset \frac{1}{10}Q^\circ$ be finite. Then $\Lambda_{Q^\circ}$ contains finitely many squares.
\end{lem}

\begin{proof}
Let $\epsilon = \frac{1}{100}\inf \{|x-y|: x,y \in E, x \neq y\}$. If $Q \subset Q^\circ$ is any dyadic square with $\delta_Q \leq \epsilon$, then $\#(3Q \cap E) \leq 1$, and so $\|3Q \cap E\|_{\bes} = 0$. This proves that dyadic squares with $\delta_Q \leq \epsilon$ are OK. Since there only a finite number of squares $Q \subset Q^\circ$ with $\delta_Q \geq \epsilon/2$, the Cutting Procedure eventually terminates.
\end{proof}

We denote $\Lambda = \Lambda_{Q^\circ} = \CZ$, and for $Q_\nu \in \Lambda$ we set $\delta_\nu = \delta_{Q_\nu}$. The squares $Q \in \Lambda$ are called Calder\'on-Zygmund squares (or for short, CZ squares).

A collection of sets, $\Pi$, is said to satisfy the \textit{Bounded Intersection Property} if there exists a universal constant constant $C$, so that for any $S \in \Pi$, there are at most $C$ elements of $\Pi$ intersecting $S$. We call $C$ the \textit{intersection constant} of $\Pi$.

\begin{lem}[Good Geometry]
\label{goodgeom}
Given $Q,Q' \in \Lambda$, the following holds:
\begin{enumerate}
\item $Q \leftrightarrow Q'$ $\Rightarrow$ $\frac{1}{2} \delta_{Q'}\leq \delta_Q \leq 2 \delta_{Q'}$;
\item $Q \cap Q' = \emptyset$ $\Rightarrow$ $\tilde{Q} \cap \tilde{Q}' = \emptyset$; and
\item $Q \cap Q' = \emptyset$ $\Rightarrow$ $d(Q,Q') \geq \frac{1}{10} \max\{\delta_{Q},\delta_{Q'}\}$.
\end{enumerate}
As a consequence of the first two properties, the collection $\{\tilde{Q}_\nu\}_{\nu=1}^K$ satisfies the Bounded Intersection Property with constant $13$.
\end{lem}

\begin{rek}
Henceforth, we refer to these conclusions regarding $\Lambda$ as the Good Geometry of the squares in $\Lambda$.
\end{rek}

\begin{proof}
We now proceed with the proof of $(1)$ from the lemma. This is sufficient, since $(2)$ and $(3)$ follow directly from $(1)$.

Suppose that $Q,Q' \in \Lambda$ satisfy $Q \leftrightarrow Q'$. Without loss of generality we may assume that $\delta_Q \leq \delta_{Q'}$. For sake of contradiction, suppose that $\delta_Q \leq \frac{1}{4} \delta_{Q'}$. Thus, $\delta_{Q^+} = 2 \delta_Q \leq \frac{1}{2} \delta_{Q'}$, and $3Q^+ \subset 3Q'$. Since $Q^+$ is not OK, and $2/p - 1 < 0$, we have
$$\|3Q' \cap E\|_{\bes} \geq \|3Q^+ \cap E\|_{\bes} > c_1\delta_{Q+}^{2/p - 1} \geq c_1\delta_{Q'}^{2/p -1}.$$ 
But this contradicts the fact that $Q'$ is a CZ square, and therefore OK.
\end{proof}

\subsection{Keystone squares} 
\label{alg}

We will use the following definition momentarily:
\begin{defn}[Roughness Property]
Let $c,c',c''>0$ be arbitrary constants. Let $Q$ be a square in $\R^2$, and $E_0$ a finite subset of $\R^2$. We say that $Q$ is \textit{Rough} relative to $E_0$ for the constants $c,c',c''$ (or rather Q satisfies \textbf{R}$(c,c',c'')$ relative to $E_0$) iff
\begin{description}
\item[(R1)] There exist $x_1,x_2,y_1,y_2 \in E_0 \cap Q$ with $x_1 \neq x_2$, $y_1 \neq y_2$, and so that $$\min \left\{ \left|\frac{x_1-x_2}{|x_1 - x_2|} - \frac{y_1-y_2}{|y_1 - y_2|} \right|, \left| \frac{x_1-x_2}{|x_1 - x_2|} + \frac{y_1-y_2}{|y_1 - y_2|} \right| \right\} > c''$$ OR
\item[(R2)] $c \delta_Q^{2/p-1} \leq \|E_0 \cap Q\|_{\bes} \leq c' \delta_Q^{2/p-1}$.
\end{description}
\end{defn}

\begin{lem}
\label{kap}
Suppose that $Q$ satisfies \textbf{R}$(c,c',c'')$ relative to $E_0$, for some universal constants $c,c',c''>0$. Then $\|Q \cap E_0\|_{\bes} \gtrsim \delta_Q^{2/p-1}$.
\end{lem}

\begin{proof}
First, suppose that $Q$ satisfies \textbf{(R2)} relative to $E_0$. In this case, $\|E_0 \cap Q\|_{\bes} \geq c \delta^{2/p-1}_Q$, and we are done.

Alternatively, $Q$ satisfies \textbf{(R1)} relative to $E_0$. Let $\varphi \in \bes(\R)$, and $(u,v)$ Euclidean coordinates, be jointly given with $E_0 \cap Q \subset \{(u,\varphi(u)): u \in \R\}$. From $\textbf{(R1)}$, we find that $|\varphi'(u_1) - \varphi'(u_2)| \gtrsim c''$ for some $u_1,u_2 \in \R$ with $|u_1-u_2| \lesssim \delta_Q$. From the BET,
$$\|\varphi\|_{\bes(\R)} \gtrsim |\varphi'(u_1)-\varphi'(u_2)| |u_1-u_2|^{2/p-1} \gtrsim c'' \delta^{2/p-1}_Q.$$ 
Thus, $\|Q \cap E_0\|_{\bes} = \inf\{\|\varphi\|_{\bes(\R)} : (u,v) \; \mbox{and} \; \varphi \;\mbox{as above}\} \gtrsim  \delta^{2/p-1}_Q$.
\end{proof}

We now let $c_2, c_3 > 0$ be small universal constants which are to be determined, moreover, they are assumed to satisfy the following.

\textbf{Order Remark (OR)}: $c_1$ and $c_2$ are chosen sufficiently small, and are allowed to depend on $c_3$. $c_3$ is a small universal constant to be determined at a later time.

We set $\Lambda^\# = \{Q^\# \in \Lambda : Q \in \Lambda \;\mbox{and} \; Q \cap 100Q^\# \neq \emptyset \; \Rightarrow \; \delta_Q \geq \delta_{Q^\#} \}$. As defined, these are the CZ squares which are local minima of the sidelength function. We call $\Lambda^\#$ the \textit{Keystone squares}. An ordered list of squares $\{Q'_i\}^k_{i=1}$ is called a \textit{path} iff $Q'_1 \leftrightarrow Q'_2 \leftrightarrow \cdots \leftrightarrow Q'_k$.

If $\Lambda$ is the singleton $\{Q^\circ\}$, i.e., $Q^\circ$ is OK, Theorem \ref{mainthm} will be easy to prove after we develop some machinery. Henceforth, we assume that $\Lambda \neq \{Q^\circ\}$. The remaining case, when $\Lambda = \{Q^\circ\}$, will be handled in Section \ref{lastcase}. For $Q^\#$ Keystone, we notice that $10Q^\#$ intersects a bounded number of squares, ensuring that the $10Q^\# \cap E$ will have uncomplicated geometry. In fact, this is not true for a general CZ square $Q$, for which $10Q$ may intersect squares much smaller than itself, leading to the formation of singularities in the set $E \cap Q$ on a much smaller lengthscale than the sidelength of $Q$; nevertheless, we use this to our advantage, building an association between CZ squares and Keystone squares in the next proposition.

\begin{prop}
\label{iskey}
As long as $\Lambda \neq \{Q^\circ\}$,
\begin{description}
\item[(K1)] $\forall Q \in \Lambda$, $\exists Q^\# \in \Lambda^\#$ with a path $$Q^\# = Q'_1 \leftrightarrow Q'_2 \leftrightarrow \cdots \leftrightarrow Q'_m = Q$$ with $\{Q'_k\}_{k=1}^m \subset \Lambda$. Moreover, for all $k_2 > k_1$ we have
$$ \delta_{Q'_{k_2}} \leq C (1-c)^{k_2-k_1} \delta_{Q'_{k_1}}$$
for some universal constants $C>0$ and $0<c<1$.
\item[(K2)] $\forall Q^\# \in \Lambda^\#$, $9Q^\#$ satisfies \textbf{R}$(c_1,c_3,c_2)$ relative to $E$.
\item[(K3)] $\{10Q^\#_\mu\}^{K^\#}_{\mu=1}$ satisfies the Bounded Intersection Property.
\end{description}
\end{prop}

The proof of Proposition \ref{iskey} can be found in Section \ref{proofiskey}. We now fix a labeling, and write $\Lambda^\# = \{Q^\#_\mu\}_{\mu=1}^{K^\#}$.

Let $Q_\nu \in \Lambda$ be given, then Proposition \ref{iskey} implies the existence of $\mu(\nu) \in \{1,2,\cdots, K^\#\}$ with $Q^\#_{\mu(\nu)} \in \Lambda^\#$ connected to $Q_\nu$ by a path as in \textbf{(K1)}. Fix such a choice of $\mu(\nu)$ for each $\nu \in \{1,2,\cdots,K\}$. We denote $\delta^\#_\mu = \delta_{Q^\#_\mu}$ for $\mu=1,2,\cdots,K^\#$. The following lemma is an immediate consequence of \textbf{(K1)}.

\begin{lem}
\label{stdkeyprops}
For any $\nu \in \{1,2,\cdots,K\}$ we have $d(Q_\nu,Q^\#_{\mu(\nu)}) \lesssim \delta_\nu$ and $\delta_{Q^\#_{\mu(\nu)}} \lesssim \delta_\nu$.
\end{lem}

\subsection{Fixing a few points}
\label{points}
Let $Q_\nu \in \Lambda$ be given. We now show how to produce $x_\nu \in \frac{1}{2}Q_\nu$ with $d(x_\nu,E) \geq \frac{1}{5} \delta_\nu$. 

Since $\|3Q_\nu \cap E\|_{\bes} \leq c_1 \delta^{2/p-1}_{Q}$, it follows that $3Q_\nu \cap E \subset \gamma = \{(u,\varphi(u)):u \in \R\}$, where $(u,v)$ are some Euclidean coordinates, and $\|\varphi\|_{\bes(\R)} \leq 2c_1\delta^{2/p-1}_Q$. The BET implies that $\varphi'$ varies by no more than $2Cc_1 \delta^{2/p-1}_Q \delta^{1-2/p}_{Q}$ on the interval $I=\mbox{proj}_{\{(u,0):u \in \R\}} (3Q)$, for some universal constant $C$. Choosing $c_1 < \frac{1}{2000C}$ we can arrange that $\varphi'$ varies by no more than $\frac{1}{100}$ on $I$. We have just shown that $3Q_\nu \cap E$ lies on the graph of a function with derivative varying by no more than $\frac{1}{100}$. Thus, there exists $x_\nu \in \frac{1}{2}Q_\nu$, with $d(x_\nu,E) \geq \frac{1}{5}\delta_\nu$, which we now fix for the remainder of the paper. We denote $E' = \{x_\nu\}_{\nu=1}^K$, which are the collection of \textit{CZ representative points}.

Additionally, for each Keystone square $Q^\#_\mu \in \Lambda^\#$, we have $Q^\#_\mu = Q_{\nu(\mu)}$ for some $1 \leq \nu(\mu) \leq K$. We define $x^\#_\mu = x_{\nu(\mu)}$, and denote $E^\# = \{x^\#_\mu\}_{\mu=1}^{K^\#} \subset E'$ which are the collection of \textit{Keystone representative points}.

\begin{lem}
\label{keyptlem}
The previously constructed $E'$ satisfies $E' \subset 0.99Q^\circ$.
\end{lem}
\begin{proof}
It will be shown that $\frac{1}{2}Q_\nu \subset 0.99 Q^\circ$ for all $\nu$. From this it easily follows that $x_\nu \in \frac{1}{2}Q_\nu \subset 0.99Q^\circ$

For any $1 \leq\nu\leq K$, either
\begin{description}
\item[(A)] $Q_\nu$ intersects the boundary of $Q^\circ$ OR
\item[(B)] $Q_\nu \subset \mbox{int}(Q^\circ)$
\end{description}
Suppose that $Q_\nu$ satisfies \textbf{(A)}, we then claim that $\delta_\nu \geq \frac{1}{32}\delta_{Q^\circ}$. Otherwise, suppose that $\delta_\nu \leq \frac{1}{64}Q^\circ$. Since $Q_\nu$ intersects the boundary of $Q^\circ$, and $E \subset \frac{1}{10}Q^\circ$, one finds that $3(Q_\nu)^+ \cap E \subset 9Q_\nu \cap E = \emptyset$. This implies that $(Q_\nu)^+$ is OK, which contradicts that $Q_\nu$ is CZ. Thus, $\delta_\nu \geq \frac{1}{32}\delta_{Q^\circ}$ for all $Q_\nu$ satisfying \textbf{(A)}.

Alternatively, suppose that $Q_\nu$ satisfies \textbf{(B)}. Then $Q_\nu \subset 0.99 Q^\circ$ thanks to the analysis from Case \textbf{(A)}, which provides a buffer of width at least $\frac{1}{32}\delta_{Q^\circ}$ between the boundary of $Q^\circ$ and squares $Q_\nu$ as in \textbf{(B)}.

In either case, $\frac{1}{2}Q_\nu \subset 0.99Q^\circ$.
\end{proof}

\subsection{Proof of Proposition \ref{iskey}}
\label{proofiskey}
We first check \textbf{(K1)}: 

We set $Q'_1 = \overline{Q}_0 = Q$, and construct the desired path iteratively starting with $\overline{Q}_0$. From the definition of the Keystone squares, either
\begin{description}
\item[Case 1:] $\overline{Q}_0 \in \Lambda^\#$ OR
\item[Case 2:] There exists $\overline{Q} \in \Lambda$ with $\overline{Q} \cap 100 \overline{Q}_0 \neq \emptyset$ and $\delta_{\overline{Q}} \leq \frac{1}{2} \delta_{\overline{Q}_0}$.
\end{description}
If Case 1 is satisfied, then the path of length 1 given by $Q \leftrightarrow Q^\# :=Q$ trivially satisfies the conditions in \textbf{(K1)}. Alternatively, suppose Case 2. Then, we may choose $\overline{Q}_1 \in \Lambda$ with
\begin{enumerate}
\item $\overline{Q}_1 \cap 100 \overline{Q}_0 \neq \emptyset$ and $\delta_{\overline{Q}_1} \leq \frac{1}{2} \delta_{\overline{Q}_0}$, and
\item $d(\overline{Q}_1,\overline{Q}_0)$ minimal among CZ squares satisfying the above property.
\end{enumerate}

From the Good Geometry of the CZ squares, there is a path: $Q = \overline{Q}_0 = Q'_1 \leftrightarrow Q'_2 \leftrightarrow \cdots Q'_{k_1} = \overline{Q}_1$, with $Q'_k \in \Lambda$, $k_1$ bounded by an absolute constant, and $\delta_{Q'_k} \approx \delta_{\overline{Q}_0}$ for all $1 \leq k \leq k_1$. To see this, connect $\overline{Q}_0$ and $\overline{Q}_1$ with a shortest line segment $\mathfrak{l}$, and take $\{Q'_i\}^{k_1}_{i=1}$ to be the collection of squares in $\Lambda$ which intersect $\mathfrak{l}$, indexed to form a path. By construction of $Q'_i$, $d(Q'_i,\overline{Q}_0) \leq d(Q'_i,\overline{Q}_1)$ for all $1 \leq i \leq k_1 -1$. Thus, $P1$-$2$ of $\overline{Q}_1$ imply that $\delta_{Q'_i} \geq \delta_{\overline{Q}_0}$ for all $2 \leq i \leq k_1$, which is enough to imply the afformentioned properties of the path $\{Q'_i\}^{k_1}_{i=1}$.

Continuing iteratively, this leads to a sequence of ``marker" squares $\overline{Q}_j \in \Lambda$ with $\delta_{\overline{Q}_j} \leq 2^{-j-1}\delta_{\overline{Q}_0}$, as well as a sequence of ``intermediary" squares $Q'_k \in \Lambda$, and the following path:
\begin{equation}
\label{e6}
Q = \overline{Q}_0 = Q'_1 \leftrightarrow \cdots \overline{Q}_j=Q'_{k_j} \leftrightarrow \cdots \overline{Q}_{j+1} = Q'_{k_{j+1}} \leftrightarrow \cdots
\end{equation}
Here, $k_{j+1}-k_j$ is bounded by an absolute constant. Since the sidelength of $\overline{Q}_k$ is exponentially decreasing in $k$, this path must terminate due to the finiteness of $\Lambda$. For the path to terminate, there must be a square $Q'_{k_n} = \overline{Q}_{n} \in \Lambda^\#$ for some $n$ (i.e., we eventually fall into Case 1). We set $Q^\# = Q'_{k_n}$. From the Good Geometry of $\Lambda$, we find that $\delta_{Q'_k} \approx \delta_{Q'_{k_j}}$ for all $k_j \leq k \leq k_{j+1}$. Along with the control on $k_{j+1} - k_j$ by an absolute constant, and $\delta_{Q'_{k_j}}  \leq 2^{-j} \delta_Q$, this proves that the path given in \eqref{e6} satisfies \textbf{(K1)}.

We now check \textbf{(K2)}: Fix $Q^\# \in \Lambda^\#$. To check that $9Q^\#$ is Rough relative to $E$, we set $E_0 = 9Q^\# \cap E$, and assume that $9Q^\#$ does not satisfy \textbf{(R1)} relative to $E_0$ (for otherwise $9Q^\#$ would be Rough relative to $E$, and we would be done). That is, for all $x_1,x_2,x_3,x_4 \in E_0$,
\begin{equation}
\min \left\{\left| \frac{x_1 - x_2}{|x_1 - x_2|} - \frac{x_3 - x_4}{|x_3 - x_4|} \right|, \left|\frac{x_1 - x_2}{|x_1 - x_2|} + \frac{x_3 - x_4}{|x_3 - x_4|} \right| \right\} \leq c_2
\label{lipprop}
\end{equation}
for a sufficiently small universal constant $c_2$. From here, our goal will be to show that $9Q^\#$ satisfies \textbf{(R2)} relative to $E_0$, which will imply \textbf{(K2)}.

Consider $Q \in \Lambda$ with $Q \cap 9Q^\# \neq \emptyset$. From the definition of the Keystone squares, $\delta_Q \geq \delta_{Q^\#}$. We now show that $\delta_Q \leq 100 \delta_{Q^\#}$: For sake of contradiction, suppose that $\delta_Q \geq 100 \delta_{Q^\#}$. Then $Q \cap 9Q^\# \neq \emptyset$ implies $\tilde{Q} \cap \tilde{Q}^\# \neq \emptyset$, and thus by the Good Geometry of the CZ squares, $Q \leftrightarrow Q^\#$. Again by the Good Geometry, we have $\delta_Q \leq 2\delta_{Q^\#}$, which is clearly a contradiction. Thus, $\delta_Q \leq 100 \delta_{Q^\#}$ for all $Q$ with $Q \cap 9Q^\# \neq \emptyset$.

Since any $Q \in \Lambda$ with $Q \cap 9Q^\# \neq \emptyset$ satisfies $\delta_Q \geq \delta_{Q^\#}$, there are at most 200 squares $Q \in \Lambda$ satisfying $Q \cap 9Q^\# \neq \emptyset$. Denote the collection of all such squares by $Q^1,\cdots,Q^n \in \Lambda$ ($n \leq 200$). Consider those squares $Q^i$ with $E_0 \cap Q^i \neq \emptyset$; we suppose that these squares have been labeled as: $Q^1 \cdots Q^m$, for $m \leq n \leq 200$. Since $Q^i$ is OK for each $i = 1,2,\cdots,m$, we have $\|3Q^i \cap E_0\|_{\bes} \leq \|3Q^i \cap E\|_{\bes} \leq c_1 \delta^{2/p-1}_{Q^i} \leq c_1 \delta^{2/p-1}_{Q^\#}$.

We now recollect our current setting:

\begin{description}
\item[(A1)] $E_0 \subset 9Q^\#$ finite, with $Q^\# \in \Lambda^\#$;
\item[(A2)] $\min\{|\frac{x_1 - x_2}{|x_1 - x_2|} - \frac{x_3 - x_4}{|x_3 - x_4|}|, |\frac{x_1 - x_2}{|x_1 - x_2|} + \frac{x_3 - x_4}{|x_3 - x_4|}|\} \leq c_2$ for all $x_1,x_2,x_3,x_4 \in E_0$ with $x_1 \neq x_2$, and $x_3 \neq x_4$;
\item[(A3)] $9Q^\# \subset \displaystyle \bigcup_{i=1}^n Q^i$, with $Q_i \in \Lambda$, and $n \leq 200$;
\item[(A4)] $\|3Q^i \cap E_0\|_{\bes} \leq c_1 \delta^{2/p-1}_{Q^\#}$ for $i=1,\cdots,m$, and $3Q^i \cap E_0 = \emptyset$ for $i=m+1,\cdots,n$;
\item[(A5)] $\delta_{Q^\#} \leq \delta_{Q^i} \leq 100 \delta_{Q^\#}$ for $i=1,\cdots,n$.
\end{description}

From \textbf{(A1-5)}, we will prove that $9Q^\#$ satisfies \textbf{(R2)} with constants $c=c_1$ and $c'=c_3$ (as long as $c_1$ and $c_2$ are sufficiently small depending on $c_3$). Through rescaling \textbf{(A1-5)} and the desired conclusion, we may assume that $\delta_{Q^\#} = 1$; this assumption will be dropped after we prove \textbf{(R2)}.

First, we show that $\|E_0\|_{\bes} \leq 9^{2/p-1} c_3$. If $E_0$ contains at most one point, then this statement is trivial. Thus, we assume that $E_0$ contains at least two distinct points. Fix distinct $y_1, y_2 \in E_0$, and choose unique Euclidean coordinates $(z_1,z_2)$ so that $y_1,y_2 \in \{z_2 = 0\}$.

Set $E'_i = 3Q^i \cap E_0$, and $\overline{E}'_i = \mbox{proj}_{\{z_2=0\}} E'_i$ for $i = 1,2,\cdots,m$. We will construct interpolating Besov curves with small semi-norm through $E'_i$ for each $i$ using the coordinate system $(z_1,z_2)$.

To do so, note from \textbf{(A4)} that $\|E'_i\|_{\bes(\R)} \leq c_1$, while \textbf{(A2)} and the definition of $(z_1,z_2)$ imply that $(z_1,z_2)$ satisfies the hypotheses of Lemma \ref{nicecurve} for the set $E'_i$ with constant $\kappa_2 = 100c_2$. Thus, Lemma \ref{nicecurve} applies and gives us $\varphi_i \in \besfull(\R)$ with
\begin{enumerate}
\item $E'_i \subset \{(z_1,\varphi_i(z_1))_{z_1z_2} : z_1 \in \R\}$, and
\item $\|\varphi_i\|_{\besfull(\R)} \lesssim c_1 + c_2$.
\end{enumerate}
Now, set $E_j = E_0 \cap Q^j$, $\overline{E}_j = \mbox{proj}_{\{z_2=0\}} E_j$, and $\overline{E}_0 = \mbox{proj}_{\{z_2=0\}} E_0$. Let $I_{j}$ be the convex hull of $\overline{E}_j$ contained in the line $\{z_2 = 0\}$. For $I_j$ taking the form $[a_j,b_j]$, we set $\tilde{I}_j = [a_j - \frac{1}{10}\delta_{Q^j}, b_j + \frac{1}{10}\delta_{Q^j}]$. From \textbf{(A5)}, we find that $|\tilde{I}_j| \geq \frac{1}{5} \delta_{Q^j} \geq \frac{1}{5}$. Furthermore, note that $\overline{E}_0 \subset \bigcup_j \overline{E}_j \subset \bigcup_j I_j$.

For $j=1,\cdots, m$, we let $\theta_j \in C_c^{\infty}(\tilde{I}_j)$ be a family of bump functions which satisfy $(1)$ $|\frac{d^k}{dx^k} \theta_j| \lesssim 1$ for $0 \leq k \leq 2$, and $(2)$ $\displaystyle \sum_{j=1}^m \theta_j(\overline{x}_0) =  1$ for $\overline{x}_0 \in \overline{E}_0$.

We briefly comment on the existence of a family of this kind: Initially, we let $\tilde{\theta}_j \in C^\infty_c (\tilde{I}_j)$ be any function which satisfies $(1)$ $\tilde{\theta}_j \equiv 1$ on $I_j$, and $(2)$ $| d^k/dx^k \tilde{\theta}_j| \lesssim 1$ for $k \leq 2$. Notice that $\psi = \sum_{j=1}^m \tilde{\theta}_j$ satisfies $|\psi| \lesssim 1$, and $\psi \geq 1$ on $\bigcup I_j$. Now, fix $\eta \in C^\infty(\R)$ with $\eta(w) \equiv w$ for $w \geq 1$, and $\eta \geq \frac{1}{2}$ for $w \leq 1$. For $j=1,\cdots, m$, define $\theta_j = \tilde{\theta}_j / \eta \circ \psi$, which satisfies the required properties.

Having collected the necessary materials, we now set $\displaystyle \varphi = \sum_{j=1}^{m} \theta_j \varphi_j$, and show that
\begin{enumerate}
\item $\|\varphi\|_{\bes(\R)} \lesssim c_1 + c_2$, and
\item $E_0 \subset \{(z_1,\varphi(z_1)): z_1 \in \R\}$.
\end{enumerate}

$P1$ of $\varphi$ follows from $P1$ of $\theta_j$, the fact that $m \leq 200$, and $P2$ of $\varphi_j$, by the following bound (note that control on the $\besfull(\R)$ norm of $\varphi_j$ is crucially used in the second ``$\lesssim$").
\begin{equation*}
\|\varphi\|_{\bes(\R)} \lesssim \sum_{j=1}^m \|\theta_j \varphi_j\|_{\bes(\R)} \lesssim \sum_{j=1}^m \|\varphi_j\|_{\besfull(\R)} \lesssim c_1 + c_2.
\end{equation*}
We now verify $P2$ of $\varphi$: Pick an arbitrary $x_0 \in E_0$, and write $x_0 = (\overline{x}_0,\overline{y}_0)_{z_1z_2}$. Alternatively, $\overline{x}_0$ is the projection of $x_0$ onto the line $\{z_2=0\}$. Suppose that the following claim holds.

\textbf{Claim 1}: If $\overline{x}_0 \in \tilde{I}_j$ for some $1 \leq j \leq m$, then $x_0 \in 3Q^j$.

We will return to the proof of Claim 1 in a moment, but first, we use it to prove $P2$ of $\varphi$. Consider $\varphi(\overline{x}_0) = \sum_{j} \theta_j(\overline{x}_0) \varphi_j(\overline{x}_0)$, where the sum need only be taken over those $j$ with $\overline{x}_0 \in \tilde{I}_j$ (this follows from the support properties of $\theta_j$). Pick $j$ with $\overline{x}_0 \in \tilde{I}_j$, and apply Claim 1 to find that $x_0 \in 3Q^j$, and so $x_0 \in 3Q^j \cap E_0 = E'_j$. Thus, $P1$ of $\varphi_j$ implies that $\varphi_j(\overline{x}_0) = \overline{y}_0$. Consequently, 
$$\varphi(\overline{x}_0) = \sum_{j} \theta_j(\overline{x}_0) \varphi_j(\overline{x}_0) = \sum_{j} \theta_j(\overline{x}_0) \overline{y}_0 = \overline{y}_0,$$
with the last ``$=$" following from $P2$ of $\theta_j$. Therefore, $x_0 = (\overline{x}_0,\overline{y}_0) = (\overline{x}_0,\varphi(\overline{x}_0))$, which proves $P2$ of $\varphi$.

For the proof of Claim 1, suppose that $x_0=(\overline{x}_0,\overline{y}_0) \in E_0$ satisfies $\overline{x}_0 \in \tilde{I}_j = [a_j-\frac{1}{10}\delta_{Q^j}, b_j + \frac{1}{10}\delta_{Q^j}]$, where $I_j=[a_j,b_j]$ is the convex hull of $\overline{E}_j$. Let $\overline{x}_j \in \overline{E}_j$ be an endpoint of $I_j$ which is closest to $\overline{x}_0$, and thus satisfies 
$$|\overline{x}_0 - \overline{x}_j| \leq \max\{\mbox{diam}(I_j)/2 , \frac{1}{10} \delta_{Q^j} \} \leq \frac{\sqrt{2}}{2}\delta_{Q^j}.$$
Since $\overline{E}_j =  \mbox{proj}_{\{z_2=0\}} E_j$, we find $\overline{y}_j \in \R$ for which $x_j = (\overline{x}_j,\overline{y}_j) \in E_j = E_0 \cap Q^j$. By choosing $c_2$ sufficiently small, and examining the definition of the coordinates $(z_1,z_2)$, \textbf{(A2)} implies that $E_0$ lies on the graph of a Lipschitz curve in $(z_1,z_2)$ coordinates: $E_0 \subset \{(z_1,\varphi_0(z_1))_{z_1z_2}:z_1 \in \R\}$ with $\|\varphi_0\|_{\mbox{Lip}(\R)} \leq \frac{1}{1000}$. Thus,
$$|\overline{y}_0 - \overline{y}_j| \leq \frac{1}{1000}|\overline{x}_0 - \overline{x}_j| \leq \frac{1}{1000} \delta_{Q^j} \leq \frac{1}{10},$$ 
since $\delta_{Q^j} \leq 100$. Therefore,
$$|x_0 - x_j|^2 \leq \frac{1}{2} \delta^2_{Q^j} + \frac{1}{100} \leq \delta^2_{Q^j},$$ 
since $1 \leq \delta_{Q^j}$. Thus, $|x_0 - x_j| \leq \delta_{Q^j}$ with $x_j \in Q^j$. Therefore $x_0 \in 3Q^j$. This completes the proof of Claim 1.

Both properties of $\varphi$ have been established. From $P1$ and $P2$ of $\varphi$, we find that $\|E_0\|_{\bes} \lesssim c_1 + c_2$. Recall (see \textbf{OR}) that $c_1$ and $c_2$ are allowed to depend on $c_3$, and thus by choosing $c_1$ and $c_2$ sufficiently small, we can arrange for $\|E_0\|_{\bes} \leq 9^{2/p-1} c_3$ as desired. This establishes half of property \textbf{R2}. For the other half, we notice that 
$$\|9Q^\# \cap E\|_{\bes} \geq \|3(Q^\#)^+ \cap E\|_{\bes} \geq c_1 \delta^{2/p-1}_{(Q^\#)^+} \geq 9^{2/p-1} c_1,$$
since $(Q^\#)^+$ is not OK. Thus, $9Q^\#$ satisfies \textbf{R}$(c_1,c_3,c_2)$ relative to $E_0$. This concludes the section where we assume $\delta_{Q^\#}=1$. The proof of \textbf{(K2)} is now complete.

We finish by establishing \textbf{(K3)}: Suppose that $10Q^\# \cap 10\overline{Q}^\# \neq \emptyset$ for some $Q^\#,\overline{Q}^\# \in \Lambda^\#$. Without loss of generality we may suppose that $\delta_{Q^\#} \geq \delta_{\overline{Q}^\#}$. Thus, $\overline{Q}^\# \cap 100Q^\# \neq \emptyset$. From the definition of $\Lambda^\#$ this implies $\delta_{\overline{Q}^\#} \geq \delta_{Q^\#}$, and so $\delta_{Q^\#} = \delta_{\overline{Q}^\#}$. Because the interiors of the squares in $\Lambda$ are disjoint, there can be no more than a universal constant number of squares $\overline{Q}^\# \in \Lambda^\#$ that satisfy $10\overline{Q}^\# \cap 10Q^\# \neq \emptyset$ and $ \delta_{\overline{Q}^\#} = \delta_{Q^\#}$. This completes the proof of \textbf{(K3)}. The proof of Proposition \ref{iskey} is complete.
\section{A Modified Extension Problem}
\label{mep}

\subsection{An inequality}

We start by establishing an inequality, which generalizes the Sobolev Embedding Theorem to a more global setting. Its proof ties in closely to the geometry of $\Lambda$ and $\Lambda^\#$ established in Proposition \ref{iskey}.

\begin{lem}[Global Sobolev inequality]
\label{globsobineq}
Suppose that $F \in \WR$. Then
\begin{align*}
&\sum_{\nu=1}^{K} |\nabla F(x_\nu) - \nabla F(x^\#_{\mu(\nu)})|^p \delta_\nu^{2-p} \lesssim \|F\|^p_{\WR},\\
&\sum_{\nu=1}^{K} |F(x_\nu) - J_{x^\#_{\mu(\nu)}} F(x_\nu)|^p \delta_\nu^{2-2p} \lesssim \|F\|^p_{\WR}.
\end{align*}
\end{lem}

\begin{proof}
We recall \textbf{(K1)}, which led us to define the map $\mu:\{1,\cdots,K\} \rightarrow \{1,\cdots,K^\#\}$ with: For each $\nu \in \{1,2,\cdots,K\}$, there exist indicies $k^\nu_1, k^\nu_2, \cdots k^\nu_{N_\nu} \in \{1,2,\cdots K\}$ with
\begin{align}
\label{expdec1}
Q_\nu = Q_{k^\nu_1} \leftrightarrow \cdots \leftrightarrow Q_{k^\nu_{N_\nu}} = Q^\#_{\mu(\nu)} \mbox{, and} \notag{} \\
\delta_{k^\nu_j} \lesssim (1-c)^{j-i} \delta_{k^\nu_i} \; \mbox{for all} \; 1\leq i < j \leq N_\nu,
\end{align}
for some $0<c<1$.

We now derive an auxiliary inequality from which the lemma follows almost immediately. Let $A_{\nu} \in \R^d$ be given for $\nu \in \{1,2,\cdots,K\}$, and $d=1$ or $2$. Let $\beta>0$ be a universal constant. Denote $A_{\nu,n} = A_{k^\nu_n}$, and $\delta_{\nu,n} = \delta_{k^\nu_n}$, for $\nu \in \{1,2,\cdots,K\}$ and $1 \leq n \leq N_\nu$. We establish the following bound:
\begin{equation}
\label{globsobineq0}
\sum_{\nu=1}^{K} |A_{\nu, N_\nu} - A_{\nu, 1}|^p \delta_\nu^{-\beta} \lesssim \sum_{k \leftrightarrow k'} |A_k - A_{k'}|^p \delta_k^{-\beta}.
\end{equation}
Let $\epsilon>0$ be a sufficiently small constant depending on $\beta$ and $p$. Then
\begin{align}
\label{globsobineq1}
\sum_{\nu=1}^{K} |A_{\nu, N_\nu} - A_{\nu, 1}|^p \delta_\nu^{-\beta} & = \sum_{\nu=1}^{K} \left| \sum_{n=1}^{N_\nu-1}\left[ A_{\nu,n} - A_{\nu,(n+1)} \right] \right|^p \delta_\nu^{-\beta}\\
& = \sum_{\nu=1}^{K} \left| \sum_{n=1}^{N_\nu-1}\left[ A_{\nu,n} - A_{\nu ,(n+1)} \right] \delta^{-\epsilon}_{\nu , n} \delta^\epsilon_{\nu , n} \right|^p \delta_\nu^{-\beta} \notag{}\\
& \lesssim \sum_{\nu=1}^{K}  \delta_\nu^{-\beta}  \sum_{n=1}^{N_\nu-1} |A_{\nu,n} - A_{\nu ,(n+1)}|^p \delta^{-\epsilon p}_{\nu,n}  \left[ \sum_{n=1}^{N_\nu - 1} \delta^{\epsilon p'}_{\nu, n} \right]^{p/p'}. \notag{}
\end{align}
Now, for a fixed $\nu \in \{1,2,\cdots,K\}$, \eqref{expdec1} implies that
\begin{align*}
\sum_{n=1}^{N_\nu - 1} \delta^{\epsilon p'}_{\nu, n} \lesssim \sum_{n=1}^{N_\nu - 1}  (1-c)^{n-1} \delta^{\epsilon p'}_{\nu, 1} \lesssim \delta^{\epsilon p'}_{\nu,1} = \delta^{\epsilon p'}_\nu.
\end{align*}
Hence the right hand side of \eqref{globsobineq1} may be bounded by
\begin{equation}
\label{globsobineq2}
\sum_{\nu} \delta_\nu^{-\beta + \epsilon p} \sum_{n=0}^{N_\nu-1} |A_{\nu,n} - A_{\nu ,(n+1)}|^p \delta^{-\epsilon p}_{\nu,n} \lesssim \sum_{k \leftrightarrow k'} |A_k - A_{k'}|^p \delta_k^{-\epsilon p} Y_k,
\end{equation}
with $\displaystyle Y_k = \sum_{\nu \in I_k} \delta^{-\beta+\epsilon p}_{\nu}$, and 
$$I_k = \{ \nu : \exists \; n \in \{1,2,\cdots,N_\nu\} \; \mbox{with} \; k = k^\nu_n \}.$$ 
\eqref{globsobineq2} followed from the previous estimate by changing the order of summation, since for each $\nu$ the list $k^\nu_1,\cdots,k^\nu_{N_\nu}$ has at most an absolute constant number of repetitions (a consequence of \eqref{expdec1}). Consider $\nu \in I_k$, and choose $1 \leq n \leq N_\nu$ with $k=k^\nu_n$. Then, \textbf{(K1)} implies that 
$$d(Q_\nu,Q_k) \lesssim \sum_{m=2}^{n-1} \mbox{diam}(Q_{k^\nu_m}) \lesssim \sum_{m=2}^{n-1} (1-c)^{m-1} \delta_{Q_{k^\nu_1}} \lesssim \delta_{\nu}, $$
and hence for $C_1$ a sufficiently large universal constant we have
\begin{equation}
\label{gsob1}
Q_\nu \subset \{x \in \R^2 : d(x,Q_k) \leq C_1 \delta_\nu\}.
\end{equation}
In particular, we find for each $Q_k$, and $\delta>0$, there are at most $\tilde{C}$ squares $Q_\nu$ with $\nu \in I_k$ and $\delta_\nu = \delta$. Moreover, we find from \eqref{gsob1} and the Good Geometry of the CZ squares that there are no such squares with $\delta \leq \frac{1}{4C_1}\delta_k$.

Let $\epsilon>0$ be sufficiently small so that $-\beta + \epsilon p < 0$. Using the preceding remarks,
\begin{equation}
\label{expsum}
Y_k = \sum_{\nu \in I_k} \delta^{-\beta+\epsilon p}_{\nu} \leq \sum_{ \substack{j \geq -\log{(4C_1)} \\ \delta = 2^j \delta_k}}\sum_{\substack{\nu \in I_k \\ \delta_{\nu}=\delta}} \delta^{-\beta+\epsilon p} \lesssim \delta^{-\beta + \epsilon p}_k.
\end{equation}
Finally, \eqref{expsum} implies that the right hand side of \eqref{globsobineq2} may be bounded by $\sum_{k \leftrightarrow k'} |A_k - A_{k'}|^p \delta^{-\beta}_k$. This completes the proof of \eqref{globsobineq0}. We are now ready to prove the two main statements of the lemma.

\textbf{Statement 1:}
We let $A_\nu = \nabla F(x_\nu) \in \R^2$, and set $\beta = p-2$. From \eqref{globsobineq0}, one finds that
\begin{align}
\label{globsobineq4}
\sum_{\nu=1}^{K} |\nabla F(x_\nu) & - \nabla F(x^\#_{\mu(\nu)})|^p \delta_\nu^{2-p}  \lesssim \sum_{k \leftrightarrow k'} |\nabla F(x_k) - \nabla F(x_{k'})|^p \delta^{2- p}_k \\
& \lesssim \sum_{k \leftrightarrow k'} \|F\|^p_{L^{2,p}(\tilde{Q}_k \cup \tilde{Q}_{k'})} \lesssim \sum_k \|F\|^p_{L^{2,p}(\tilde{Q}_k)}  \lesssim \|F\|^p_{L^{2,p}(\R^2)}. \notag{}
\end{align}
Here, the second ``$\lesssim$" follows from the SET for the domain $\tilde{Q}_k \cup \tilde{Q}_{k'}$ (see Remark \ref{twosquares}), and the third and fourth ``$\lesssim$" follow from the Bounded Intersection Property of the squares $\{\tilde{Q}_\nu\}^K_{\nu=1}$. This completes the proof of the first inequality in the statement of the lemma.

\textbf{Statement 2:}
Now, let $A_\nu = J_{x_\nu} F(x^\#_{\mu(\nu)})$, and $\beta = 2p-2$. Notice that
\begin{align*}
\sum_{\nu=1}^{K} |F(x_\nu) - J_{x^\#_{\mu(\nu)}} F(x_\nu)|^p \delta_\nu^{2-2p} & \lesssim \sum_{\nu=1}^{K} |J_{x_\nu}F(x^\#_{\mu(\nu)}) - F(x^\#_{\mu(\nu)})|^p \delta_\nu^{2-2p} \\
& + \sum_{\nu=1}^{K} |(\nabla F(x_\nu) - \nabla F(x^\#_{\mu(\nu)}))(x^\#_{\mu(\nu)} - x_\nu)|^p \delta_\nu^{2-2p} \\
& \lesssim \sum_{\nu=1}^{K} \left[ |J_{x_\nu}F(x^\#_{\mu(\nu)}) - F(x^\#_{\mu(\nu)})|^p \delta_\nu^{2-2p} + |\nabla F(x_\nu) - \nabla F(x^\#_{\mu(\nu)})|^p \delta_\nu^{2-p} \right],
\end{align*}
with the last inequality following from Lemma \ref{stdkeyprops}, which implies that $|x^\#_{\mu(\nu)} - x_\nu| \lesssim \delta_\nu$. Upon applying \eqref{globsobineq0} and \eqref{globsobineq4} we can bound the right hand side by
\begin{align*}
\sum_{\nu=1}^{K} |J_{x_\nu}F(x^\#_{\mu(\nu)}) & - F(x^\#_{\mu(\nu)})|^p \delta_\nu^{2-2p} + \|F\|^p_{L^{2,p}(\R^2)} \lesssim \sum_{k \leftrightarrow k'} |J_{x_k} F(x^\#_{\mu(k)}) - J_{x_{k'}} F(x^\#_{\mu(k')})|^p \delta^{2-2p}_k + \|F\|^p_{L^{2,p}(\R^2)} \\
& \lesssim \sum_{k \leftrightarrow k'} |J_{x_k} F(x_k) - J_{x_{k'}} F(x_k)|^p \delta^{2-2p}_k \notag{} + \sum_{k \leftrightarrow k'} |\nabla F(x_k) - \nabla F(x^\#_{\mu(k)})|^p |x_k - x^\#_{\mu(k)}|^p \delta^{2-2p}_k \\
& + \|F\|^p_{L^{2,p}(\R^2)} \lesssim \sum_{k \leftrightarrow k'} \|F\|^p_{L^{2,p}(\tilde{Q}_k \cup \tilde{Q}_{k'})} + \|F\|^p_{\WR} \lesssim \|F\|^p_{\WR}.
\end{align*}
This completes the proof of the second statement, and thus the lemma as well.
\end{proof}

\subsection{Choosing an extension whose jets are ``constant'' along paths}

Let $L^\#=(L^\#_x)_{x \in E^\#} \in Wh(E^\#)$ be given. Define the \textit{constant-path extension} of $L^\#$ to be $L \in Wh(E')$ given by
\begin{equation*}
L_{x_\nu} = L^\#_{x^\#_{\mu(\nu)}} \quad \forall \nu = 1,2,\ldots,K,
\end{equation*}
with $\mu(\nu)$ defined as in Section \ref{CZ}. For $1\leq\mu\leq K^\#$, and $1 \leq \nu \leq K$, we denote $L^\#_\mu = L^\#_{x^\#_\mu}$, and $L_\nu = L_{x_\nu}$

\textbf{A remark on our notation:} The passage between a Whitney Field with and without a $\#$ always means that the one Whitney Field is related to the other through its constant-path extension. Also, the indexing present in the above definition extends for all jets in either $Wh(E')$ of $Wh(E^\#)$.

One of the key ideas in this paper is the addition of additional linear conditions on proposed Sobolev interpolants of $f$, whose purpose is mod out certain degrees of freedom in the search for our extension. It is important for these extra conditions to be natural, in the sense that the optimal norm of an extension which satisfies them is not much larger than the optimal norm of an interpolant of $f$. We say that $F \in L^{2,p}(\R^2)$ satisfies the \textit{constant-path property} (or CPP) iff $J_{x_\nu}F = J_{x^\#_{\mu(\nu)}}F $ for all $1 \leq \nu \leq K$.

\begin{lem}
\label{modlem}
Let $f:E \rightarrow \R$. There exists $\widehat{F} \in \WR$ with 
\begin{enumerate}
\item $\widehat{F}|_E=f$;
\item $\|\widehat{F}\|_{\WR} \lesssim \|f\|_{\WR|_E}$;
\item $\widehat{F}$ satisfies the constant-path property.
\end{enumerate}
\end{lem}
\begin{rek}
Equivalently, the above lemma claims the existence of $\widehat{F} \in \WR$ with $\widehat{F}|_E=f$, $\|\widehat{F}\|_{\WR} \lesssim \|f\|_{\WR|_E}$, and $J_{E'}\widehat{F}=L$ for $L$ the constant-path extension of $L^\#=J_{E^\#} \widehat{F}$.
\end{rek}

\begin{proof}
From the definition of the trace semi-norm, we may fix $F \in \WR$ with: $(1)$ $F|_E = f$, and $(2)$ $\|F\|_{\WR} \leq 2 \|f\|_{\WR|_E}$. Our goal will be to modify $F$ near $E'$ so as to ensure it satisfies the CPP without disturbing the values of $F$ on $E$. Let $L^\# = J_{E^\#}F$, with $L \in Wh(E')$ the constant-path extension of $L^\#$.

For $1 \leq \nu \leq K$, we choose $\theta_\nu \in C_c^{\infty}(B(x_\nu,\frac{1}{20}\delta_\nu))$ which satisfies: $(1)$ $0 \leq \theta_\nu \leq 1$, $(2)$ $\theta_\nu \equiv 1$  on $B(x_\nu, \frac{1}{40}\delta_\nu)$, and $(3)$ $|\partial^{\alpha} \theta_\nu| \leq \delta_\nu^{-|\alpha|}$ for each $|\alpha| \leq 2$.

We recall that $x_\nu \in \frac{1}{2}Q_\nu$ and $d(x_\nu,E) \geq \frac{1}{10} \delta_Q$. Along with $\mbox{supp}(\theta_\nu) \subset B(x_\nu,\frac{1}{20}\delta_\nu) \subset Q_\nu$, this implies
\begin{align*}
&\nu \neq \nu' \Rightarrow \mbox{supp}(\theta_\nu) \cap \mbox{supp}(\theta_{\nu'}) = \emptyset;
&\theta_\nu|_E = 0 \; \mbox{for all} \; \nu \in \{1,2,\cdots,K\}.
\end{align*}

Define $h_\nu = \theta_\nu (L_\nu - J_{x_\nu}F)$, which satisfies: $(1)$ $h_{\nu}|_E = 0$, $(2)$ $J_{x_\nu} h_\nu = L_\nu - J_{x_\nu}F$, $(3)$ $\mbox{supp}(h_\nu) \subset Q_\nu$, and $(4)$ $\|h_\nu\|^p_{L^{2,p}(Q_\nu)} \lesssim |L_\nu(x_\nu) - J_{x_\nu}F(x_\nu)|^p \delta^{2-2p}_\nu + |\nabla L_\nu - \nabla F(x_\nu)|^p \delta^{2-p}_\nu$.
 
Finally, define $\widehat{F} = F + \sum_{\nu} h_\nu$. Since the supports of the functions $h_\nu$ ($1 \leq \nu \leq K$) are disjoint and contained within $Q_\nu$, and by Lemma \ref{globsobineq} as well as $P2$ of $F$,
\begin{align*}
\|\widehat{F}\|^p_{\WR} & \lesssim \|F\|^p_{\WR} + \|\sum_{\nu} h_\nu\|^p_{L^{2,p}(Q_\nu)} = \|F\|^p_{\WR} + \sum_{\nu} \|h_\nu\|^p_{L^{2,p}(Q_\nu)} \\ 
& \lesssim \|F\|^p_{\WR} + \sum_{\nu} \left[ |J_{x^\#_{\mu(\nu)}}F(x_\nu)  - J_{x_\nu}F(x_\nu)|^p \delta^{2-2p}_\nu + |\nabla F(x^\#_{\mu(\nu)}) - \nabla F(x_\nu)|^p \delta^{2-p}_\nu \right] \\
& \lesssim \|F\|^p_{\WR} \leq 2^p \|f\|_{\WR|_E}.
\end{align*}
Using the fact that $x_\nu \in \mbox{supp}(h_{\nu'})$ iff $\nu'=\nu$, and $P1$-$3$ of $h_\nu$ as well as $P1$ of $F$,
\begin{align*}
&J_{x_\nu} \widehat{F} = J_{x_\nu} F + J_{x_\nu} h_\nu = L_\nu; \\
&\widehat{F}|_E = F|_E + \sum_{\nu} h_\nu|_E = F|_E + 0 = f.
\end{align*}
Since $L_\nu$ is the constant-path extension of $L^\#$, $\widehat{F}$ satisfies the CPP.
\end{proof}

Since the additional assumption that our interpolant of $f$ satisfies the CPP does not make its optimal norm any worse, we might as well assume it from here on out. This leads us to consider the following extension problem, which has become natural in light of the previous result.

\textbf{(Modified Extension Problem)}
Given a finite $E \subset \R^2$, $f: E \rightarrow \R$, and $L^\# \in Wh(E^\#)$. Set $L \in Wh(E')$ to be the constant-path extension of $L^\#$. We say that $\widehat{F} \in \WR$, and $\widehat{M}(f,L^\#) \in [0,\infty)$ solve the \textbf{(MEP)} with data $(f,L^\#) \in \WR|_E \times Wh(E^\#)$ if and only if
\begin{description}
\item[(MEPa)] $\widehat{F}|_E = f$;
\item[(MEPb)] $J_{E'} \widehat{F} = L$;
\item[(MEPc)] $\|\widehat{F}\|_{\WR}$ is C-optimal among $\WR$-functions satisfying the two properties above.
\item[(MEPd)] $\widehat{M}(f,L^\#) \approx \|\widehat{F}\|_{\WR}$
\end{description}

From \textbf{(MEPc)} and \textbf{(MEPd)}, we find that $\widehat{M}(f,L^\#) \approx \inf\{\|F\|_{\WR} : F|_E=f, J_{E'} F = L\}$ (as usual, $L$ is the constant-path extension of $L^\#$). The \textbf{(MEP)} interests us because of the following lemma.

\begin{lem}
Let $E \subset \R^2$ be finite, and $f:E \rightarrow \R$ be given. Then $$\|f\|_{\WR|_E} \approx \inf \{ \widehat{M}(f, L^\#) : L^\# \in Wh(E^\#) \}.$$ 
\label{minilem}
\end{lem}
\begin{proof}
The ``$\lesssim$'' direction is trivial because of the extra conditions present in the $\textbf{(MEP)}$. The ``$\gtrsim$'' direction is a consequence of Lemma \ref{modlem}, which states that there exist $C$-optimal interpolants of $f$ which satisfy the CPP.
\end{proof}

Lemma \ref{minilem} suggests a first step towards constructing C-optimal interpolants of $f$: For any $L^\# \in Wh(E^\#)$, in Section \ref{patch}, we construct a global solution to the \textbf{(MEP)} with data $(f,L^\#)$ with an approximate formula for its $\WR$ norm given by $\widehat{M}(f,L^\#)$. Later, in Section \ref{findwhitfield}, we produce a Whitney field $\tilde{L}^\# \in Wh(E^\#)$ depending linearly on $f$ which is an essential infimum for $\widehat{M}(f,\cdot)$ in the sense that $\widehat{M}(f,\tilde{L}^\#) \lesssim \widehat{M}(f,L^\#)$ for all $L^\# \in Wh(E^\#)$. Finally, the solution to the Sobolev interpolation problem for $f$ is taken to be the solution of the \textbf{(MEP)} with data $(f,\tilde{L}^\#)$.

\section{A Local Modified Extension Problem}
\label{LMEP}

We solve the \textbf{(MEP)} by first solving a collection of local extension problems. The local extension problem will be one with function values specified on the set $\tilde{Q}_\nu \cap E$, and a jet specified at $x_\nu$. The geometry of the set for each local piece is quite simple, which makes these local extension problems easy to solve. We now detail the geometric properties of each local piece that we require for the arguments of this section to follow through. Let $c_4>0$ be a sufficiently small universal constant (independent of $c_1$, $c_2$, and $c_3$). The following assumptions are in effect for the remainder of this section.

\noindent\textbf{Geometric Assumptions:}
\begin{itemize}
\item $Q$ is an arbitrary square in $\R^2$, and $E_0 \subset 0.9Q$ is finite with $\|E_0\|_{\bes} \leq c_4 \delta_Q^{2/p-1}$.
\item $x_0 \in \frac{1}{2}Q$ satisfies $d(x_0,E_0) \geq \frac{1}{100} \delta_{Q}$.
\end{itemize}

\begin{thm}[Local Modified Extension Operator]
There exists a sufficiently small universal constant $c_4>0$ such that the following holds: For any $Q$, $E_0$, and $x_0$ that satisfy the Geometric Assumptions (with constant $c_4>0$), there exists a linear operator $\widehat{T}_0: L^{2,p}(Q)|_{E_0} \times P \rightarrow L^{2,p}(Q)$ and a non-negative real number $\widehat{M}_Q(\cdot,\cdot)$ which satisfy
\begin{enumerate}
\item $\widehat{T}_0(f_0,L_0)|_{E_0} = f_0$;
\item $J_{x_0}\widehat{T}_0(f_0,L_0) = L_0$;
\item $\|\widehat{T}_0(f_0,L_0)\|_{L^{2,p}(Q)} \approx \widehat{M}_Q(f_0,L_0)$ is C-optimal with respect to the previous two properties.
\item There exist linear functionals $\{\lambda_i\}_{i=1}^{N_0}$ with $N_0 \lesssim (\#E_0)^2$ and $\widehat{M}_Q(f_0,L_0)^p = \sum_{i=1}^{N_0} |\lambda_i(f_0,L_0)|^p$.
\end{enumerate}

\label{localthm}
\end{thm}

\begin{rek} It is simple to check that $\widehat{M}_Q(\cdot,\cdot)$ is $C$-equivalent to a semi-norm on $L^{2,p}(Q)|_{E_0} \times P$, and is therefore essentially subadditive in the sense that $\widehat{M}_Q(f_0+f_1,L_0+L_1) \lesssim \widehat{M}_Q(f_0,L_0) + \widehat{M}_Q(f_1,L_1)$.
\end{rek}

\begin{proof}

Through a standard rescaling argument, without loss of generality $Q=[-1/2,1/2]^2$. Since $E_0 \subset 0.9Q$, a standard cutoff function argument yields
\begin{align}
\label{locglob}
\|f_0\|_{L^{2,p}(Q)|_{E_0}} &\approx \|f_0\|_{L^{2,p}(\R^2)|_{E_0}}; \\
\|f_0\|_{W^{2,p}(Q)|_{E_0}} &\approx \|f_0\|_{W^{2,p}(\R^2)|_{E_0}}. \notag{}
\end{align}

Since $E_0 \subset 0.9Q$, and $\|E_0\|_{\bes} \leq c$ with $c$ sufficiently small, Lemma \ref{strlem} gives us a diffeomorphism $\Phi:\R^2 \rightarrow \R^2$, and Euclidean coordinates $(u,v)$ on $\R^2$, with
\begin{enumerate}
\item $\Phi(u,v) = (u,0)$ for all $(u,v) \in E_0$;
\item $\|\Phi\|_{L^{2,p}(\R^2)} \lesssim c$;
\item $|\nabla \Phi(x)|,$ $|(\nabla \Phi(x))^{-1}| \lesssim 1$ for all $x \in \R^2$.
\end{enumerate}
Let $E_1 = \Phi(E_0) \subset \R$, and $f_1 = f_0 \circ \Phi^{-1}|_{E_1}$. The following equivalence of norms follows from Lemma \ref{sobdiff}.
\begin{align}
\label{sobdiff2}
\|f_0\|_{W^{2,p}(\R^2)|_{E_0}} & = \inf\{\|F\|_{W^{2,p}(\R^2)} : F|_{E_0} = f_0 \} \\
& \approx \inf \{ \|F \circ \Phi^{-1}\|_{W^{2,p}(\R^2)} : F|_{E_1} = f_1\} = \|f_1\|_{W^{2,p}(\R^2)|_{E_1}}. \notag{}
\end{align}  
Also, from Proposition \ref{trthm} we have
\begin{align}
\label{treq}
\|f_1\|_{W^{2,p}(\R^2)|_{E_1}} &= \inf \{\|F\|_{W^{2,p}(\R^2)} : F|_{E_1} = f_1\} \\
& \approx \inf \{\|g\|_{\besfull(\R)} : g|_{E_1} = f_1\} = \|f_1\|_{\besfull(\R)|_{E_1}}. \notag{}
\end{align}
Applying Proposition \ref{lbe} to $f_1:E_1 \subset \R \rightarrow \R$ we find $g \in \besfull(\R)$ with: $(1)$ $g|_{E_1} = f_1$, $(2)$ $\|g\|_{\besfull(\R)} \lesssim \|f_1\|_{\besfull(\R)|_{E_1}}$, and $(3)$ $g$ depends linearly on $f_1$. Proposition \ref{trthm} then applies to produce $F_1 \in W^{2,p}(\R^2)$ with: $(1)$ $F_1|_{\R} = g$, $(2)$ $\|F_1\|_{W^{2,p}(\R^2)} \lesssim \|g\|_{\besfull(\R)}$, and $(3)$ $F_1$ depends linearly on $g$. From the linear dependence of $g$ on $f_1$, $F_1$ also depends linearly on $f_1$. From $P1$ of $g$ and $P1$ of $F_1$ it follows that $F_1|_{E_1} = g|_{E_1} = f_1$. From  $P2$ of $g$, $P2$ of $F_1$, and \eqref{treq},
\begin{equation}
\|F_1\|_{W^{2,p}(\R^2)} \lesssim \|g\|_{\besfull(\R)} \lesssim \|f_1\|_{\besfull(\R)|_{E_1}} \approx \|f_1\|_{W^{2,p}(\R^2)|_{E_1}}.
\label{nbo1}
\end{equation}
Let $F_2 := F_1 \circ \Phi$, for which: $(1)$ $F_2|_{E_0} = F_1 \circ \Phi |_{E_0} = f_1 \circ \Phi|_{E_0} = f_0$.
Lemma \ref{sobdiff} implies that $\|F_2\|_{W^{2,p}(\R^2)} \approx \|F_1\|_{W^{2,p}(\R^2)}$. Similarly, \eqref{sobdiff2} implies that $\|f_1\|_{W^{2,p}(\R^2)|_{E_1}} \approx \|f_0\|_{W^{2,p}(\R^2)|_{E_0}}$. Therefore, from \eqref{nbo1}: $(2)$ $\|F_2\|_{W^{2,p}(\R^2)} \lesssim \|f_0\|_{W^{2,p}(\R^2)|_{E_0}}$.

Choose $\theta \in C_c^{\infty}(B(x_0,\frac{1}{150}))$, with: $(1)$ $\theta \equiv 1$ on $B(x_0,\frac{1}{200})$, and $(2)$ $|\partial^{\alpha} \theta| \lesssim 1$ for $|\alpha| \leq 2$. Since $d(x_0,E_0) \geq \frac{1}{100}$, we have $\theta|_{E_0} = 0$. Define $F_3 = (F_2 - \theta J_{x_0}F_2)|_{Q}$. We show that
\begin{enumerate}
\item $F_3|_{E_0} = F_2|_{E_0} - 0 = f_0$;
\item $J_{x_0} F_3 = J_{x_0}F_2 - J_{x_0} F_2 = 0$;
\item $\|F_3\|_{L^{2,p}(Q)}$ is C-optimal with respect to the two properties above, with $\|F_3\|_{L^{2,p}(Q)} \approx \|f_0\|_{W^{2,p}(Q)|_{E_0}}$.
\end{enumerate}
Indeed, $P1$ and $P2$ follow immediately as above. From the SET, $P2$ of $F_2$, and \eqref{locglob},
\begin{equation*}
\|F_3\|_{W^{2,p}(Q)} = \|F_2 - \theta J_{x_0} F_2\|_{W^{2,p}(Q)} \lesssim \|F_2\|_{W^{2,p}(\R^2)} \lesssim \|f_0\|_{W^{2,p}(\R^2)|_{E_0}} \approx \|f_0\|_{W^{2,p}(Q)|_{E_0}}
\end{equation*}
Since $F_3$ interpolates $f_0$, $\|F_3\|_{W^{2,p}(Q)}$ is $C$-optimal with respect to $P1$-$2$ above, and moreover $\|F_3\|_{W^{2,p}(Q)} \approx \|f_0\|_{W^{2,p}(Q)|_{E_0}}$. For the moment, consider an arbitrary $F$ with $F|_{E_0} = f_0$, and $J_{x_0}F = 0$. The SET implies that $\|F\|_{W^{2,p}(Q)} \lesssim \|F\|_{L^{2,p}(Q)}$, and so the $W^{2,p}(Q)$ and $L^{2,p}(Q)$ norms are $C$-equivalent for such an $F$. This argument implies that $\|F_3\|_{L^{2,p}(Q)}$ is also C-optimal with respect to $P1$-$2$ of $F_3$, and that $\|F_3\|_{L^{2,p}(Q)} \approx \|F_3\|_{W^{2,p}(Q)} \approx \|f_0\|_{W^{2,p}(Q)|_{E_0}}$. This establishes $P3$ of $F_3$.

We finally set $F_4 = F_3 + L_0$, for which: $(1)$ $F_4|_{E_0} = f_0 + L_0|_{E_0} = f_0$, and $(2)$ $J_{x_0} F_4 = 0 + L_0 = L_0$. From $P3$ of $F_3$, and the fact that the Sobolev semi-norm is invariant under addition of affine functions, $\|F_4\|_{L^{2,p}(Q)}$ is C-optimal with respect to $(1)$ and $(2)$. One sees from the above construction that $F_4$ depends linearly on $f_0$ and $L_0$, and so $\widehat{T}_0(f_0,L_0) = F_4$ is the desired linear extension operator. It only remains to find an approximate formula for $\|F_4\|_{L^{2,p}(Q)}$ that is given by a sum of linear functionals raised to the $p$'th power. Through inequalities derived above,
\begin{align*}
\|F_4\|^p_{L^{2,p}(Q)} = \|F_3\|^p_{L^{2,p}(Q)} & \approx \|f_0\|_{W^{2,p}(Q)|_{E_0}} \approx \|f_0\|_{W^{2,p}(\R^2)|_{E_0}} \\
& \approx \|f_1\|^p_{W^{2,p}(\R^2)|_{E_1}} \approx \|f_1\|^p_{\besfull(\R)|_{E_1}} \approx \sum_{i=1}^{N_0} |\lambda_i(f_0,L_0)|^p,
\end{align*}
with $N_0 \lesssim (\#E_0)^2$. The last $\approx$ comes from the formulation of the Besov trace norm in Proposition \ref{lbe}. This completes the proof of Theorem \ref{localthm}.
\end{proof}

\section{Patching the local solutions}
\label{patch}
In this section, along with the already fixed $E \subset \frac{1}{10} Q^\circ$ and $f:E \rightarrow \R$, we also fix $L^\# \in Wh(E^\#)$. As usual, we let $L \in Wh(E')$ be the constant-path extension of $L^\#$. Let $E_\nu = 1.1 Q_\nu \cap E$, and $f_\nu = f|_{E_\nu}$. From the OK property of $Q_\nu$ we have $\|E_\nu\|_{\bes} \leq c_1 \delta^{2/p-1}_\nu$, and obviously $E_\nu \subset 1.1Q_\nu \subset 0.9 \tilde{Q}_\nu$.

Since we are free to choose $c_1$ smaller than $c_4$, we can arrange that $\|E_\nu\|_{\bes} \leq c_4 \delta^{2/p-1}_\nu$ for all $1\leq \nu \leq K$. Since $d(x_\nu,E_\nu) \geq d(x_\nu,E) \geq \frac{1}{10}\delta_\nu$, we find that $x_\nu$, $E_\nu$, and $\tilde{Q}_\nu$ satisfy the Geometric Assumptions from Section \ref{LMEP} for each $1 \leq \nu \leq K$. Thus, by Theorem \ref{localthm}, there exists $\widehat{T}_\nu:L^{2,p}(\tilde{Q}_\nu)|_{E_\nu} \times P \rightarrow L^{2,p}(\tilde{Q}_\nu)$ with 
\begin{enumerate}
\item $\widehat{T}_\nu(f_\nu,L_\nu)|_{E_\nu} = f_\nu$;
\item $J_{x_\nu} \widehat{T}_\nu(f_\nu,L_\nu) = L_\nu$;
\item $\|\widehat{T}_\nu(f_\nu,L_\nu)\|_{L^{2,p}(\tilde{Q}_\nu)}$ is C-optimal with respect to previous two properties.
\end{enumerate} 
Moreover, there exists a non-negative real number of the form $\widehat{M}_\nu (f_\nu,L_\nu)^p =  \sum^{N_\nu}_{i=1} |\lambda^\nu_i(f_\nu,L_\nu)|^p$ with $N_\nu \lesssim (\#E_\nu)^2$ and $\widehat{M}_\nu(f_\nu,L_\nu) \approx \|\widehat{T}_\nu(f_\nu,L_\nu)\|_{L^{2,p}(\tilde{Q}_\nu)}$. The solution to the \textbf{(MEP)} through a linear extension operator follows.

\begin{prop}
Let $E \subset \frac{1}{10}Q^\circ$ be finite. There exists a linear operator $\widehat{T}:\WR|_E \times Wh(E^\#) \rightarrow \WR$, and a non-negative real number $\widehat{M}(\cdot,\cdot)$ so that the following holds: For any $f:E \rightarrow \R$, and $L^\# \in Wh(E^\#)$, we have
\begin{enumerate}
\item $\widehat{T}(f,L^\#)|_{E} = f$;               
\item $J_{E'}\widehat{T}(f,L^\#) = L$;
\item $\|\widehat{T}(f,L^\#)\|_{\WR} \approx \widehat{M}(f,L^\#)$ is C-optimal with respect to the two properties above.
\end{enumerate}
Moreover, we may take
\begin{equation*}
\widehat{M}(f,L^\#)^p = \sum_{\nu} \widehat{M}_\nu(f_\nu, L_\nu)^p + \sum_{\nu \leftrightarrow \nu'} \left[ |\nabla L_\nu -\nabla L_{\nu'}|^p \delta^{2-p}_\nu + |L_\nu(x_\nu) - L_{\nu'}(x_\nu)|^p \delta^{2-2p}_\nu \right].
\end{equation*}
\label{mainprop}
\end{prop}

\subsection{Proof of Proposition \ref{mainprop}}

For ease of notation, we let $\widehat{F}_\nu = \widehat{T}_\nu(f_\nu,L_\nu) \in L^{2,p}(\tilde{Q}_\nu)$. For each $1 \leq \nu \leq K$, fix $\theta_\nu \in C^\infty_c(1.1Q_\nu)$ which satisfies: $(1)$ $0 \leq \theta_\nu \leq 1$ for all $1 \leq \nu \leq K$, $(2)$ $\theta_\nu \equiv 1$ on $0.9 Q_\nu$, $(3)$ $\sum_{\nu} \theta_{\nu} \equiv 1$ on $Q^{\circ}$, and $(4)$ $|\frac{\partial^{\alpha}}{\partial x^\alpha} \theta_\nu| \lesssim \delta_\nu^{-|\alpha|}$ for all $|\alpha| \leq 2$. Define
\begin{equation*}
\overline{F} = \sum_{\nu} \theta_\nu \widehat{F}_\nu.
\end{equation*}
Clearly $\overline{F}$ depends linearly on $f$ and $L^\#$ since it is a linear combination of such functions. Using $P1$-$3$ of $\theta_\nu$, $P1$-2 of $\widehat{T}_\nu$, $E_\nu \subset 1.1Q_\nu$, and $x_\nu \in 0.9Q_\nu$, we have $\overline{F}|_{E} = f$, and $J_{E'} \overline{F} = L$. We now estimate $\|\overline{F}\|_{L^{2,p}(Q^\circ)}$.

\begin{lem}[$L^{2,p}(Q^\circ)$ Norm of $\overline{F}$]
We have
\begin{equation}
\|\overline{F}\|^p_{L^{2,p}(Q^\circ)} \lesssim \sum_{\nu} \widehat{M}^p_\nu + \sum_{\nu \leftrightarrow \nu'} \left[|\nabla L_\nu - \nabla L_{\nu'}|^p \delta^{2-p}_\nu + |L_\nu(x_\nu) - L_{\nu'}(x_\nu)|^p \delta^{2-2p}_\nu \right].
\label{fullbound1}
\end{equation}
\label{upperboundlem}
\end{lem}

\begin{proof}
It follows from $P3$ of $\theta_\nu$ that for any $x \in Q^\circ$, and $\alpha$ with $0 < |\alpha| \leq 2$,
\begin{equation*}
\sum_{\nu} \frac{\partial^{\alpha}}{\partial x^\alpha} \theta_\nu (x) = \frac{\partial^{\alpha}}{\partial x^\alpha}  \sum_{\nu} \theta_\nu (x) = \frac{\partial^{\alpha}}{\partial x^\alpha} 1 =  0.
\end{equation*}
Fix $\nu' \in \{1,2,\cdots,K\}$, then using the previous identity one can rewrite 
\begin{equation}
\label{hesseq}
\nabla^2 \overline{F} = \sum_{\nu} \nabla^2 (\widehat{F}_\nu) \theta_\nu + 2 \sum_{\nu} \nabla(\widehat{F}_\nu - \widehat{F}_{\nu'}) \otimes \nabla(\theta_\nu) + \sum_{\nu} (\widehat{F}_\nu - \widehat{F}_{\nu'}) \nabla^2(\theta_\nu),
\end{equation}
with $(v^1 \otimes v^2)_{ij} = (1/2)(v^1_i v^2_j + v^1_j v^2_i)$ the symmetrized tensor product.

Our plan is to first get a bound on $\|\overline{F}\|^p_{L^{2,p}(Q_{\nu'})}$, and then to sum this bound over $\nu' \in \{1,2,\ldots,K\}$ to recover a bound on the full integral norm as in \eqref{fullbound1}. For $x \in Q_{\nu'}$, if $x \in \mbox{supp} (\theta_\nu) \subset 1.1Q_\nu$ for some $\nu$ then $x \in \tilde{Q}_\nu \cap Q_{\nu'}$. From the Good Geometry of the CZ squares it follows that $\nu \leftrightarrow \nu'$; moreover, for each fixed $\nu'$. this may occur for at most a bounded number of $\nu$. Thus, by \eqref{hesseq},
\begin{align}
\label{term1} 
\|\nabla^2 \overline{F}\|^p_{L^p( Q_{\nu'})} &\lesssim \sum_{\nu \leftrightarrow \nu'} \| \nabla^2 (\widehat{F}_\nu) \theta_\nu \|^p_{L^p( Q_{\nu'})} \\
& + \sum_{\nu \leftrightarrow \nu'} \| \nabla(\widehat{F}_\nu - \widehat{F}_{\nu'}) \otimes \nabla(\theta_\nu) \|^p_{L^p( Q_{\nu'})} \label{term2} \\
& + \sum_{\nu \leftrightarrow \nu'} \| (\widehat{F}_\nu - \widehat{F}_{\nu'}) \nabla^2(\theta_\nu)\|^p_{L^p( Q_{\nu'})}
\label{term3}
\end{align}
We start with an estimation of a term in the first sum on the RHS of \eqref{term1}. For any $\nu$ with $\nu \leftrightarrow \nu'$,
\begin{equation}
\| \nabla^2 (\widehat{F}_\nu) \theta_\nu \|^p_{L^p( Q_{\nu'})} \leq \| \nabla^2 (\widehat{F}_\nu) \|^p_{L^p( {\tilde{Q}_\nu})} \approx \widehat{M}^p_\nu. \label{term1a}
\end{equation}
Here, the first inequality follows from $\mbox{supp}(\theta_\nu) \subset \tilde{Q}_\nu$ and $P1$ of $\theta_\nu$, and the ``$\approx$" from $P3$ of $\widehat{T}_\nu$. We now examine a term in the sum given in \eqref{term2}: Recall that $\nabla F(x_\nu) = \nabla L_\nu$ for all $1 \leq \nu \leq K$. $P1$ and 4 of $\theta_\nu$, as well as the SET imply that for any $\nu$ with $\nu \leftrightarrow \nu'$,
\begin{align}
\label{term2a} 
\| \nabla(\widehat{F}_\nu - \widehat{F}_{\nu'}) \otimes \nabla(\theta_\nu) \|^p_{L^p( Q_{\nu'})} & \lesssim \| (\nabla L_\nu - \nabla L_{\nu'}) \otimes \nabla \theta_\nu \|^p_{L^p( Q_{\nu'})} + \| (\nabla \widehat{F}_\nu - \nabla L_\nu) \otimes \nabla \theta_\nu\|^p_{L^p( Q_{\nu'})} \\
& \qquad\qquad + \| (\nabla L_{\nu'} - \nabla \widehat{F}_{\nu'}) \otimes \nabla \theta_\nu \|^p_{L^p( Q_{\nu'})} \notag{} \\
&\lesssim \delta^{2-p}_\nu |\nabla L_\nu - \nabla L_{\nu'}|^p + \delta^{-p}_\nu \| \nabla \widehat{F}_\nu - \nabla L_\nu\|^p_{L^p( \tilde{Q}_\nu)} \notag{} \\
& \qquad\qquad + \delta^{-p}_\nu \|\nabla L_{\nu'} - \nabla \widehat{F}_{\nu'}\|^p_{L^p( \tilde{Q}_{\nu'})} \notag{} \\
& \lesssim \delta^{2-p}_\nu |\nabla L_\nu - \nabla L_{\nu'}|^p + \|\widehat{F}_\nu\|^p_{L^{2,p}(\tilde{Q}_\nu)} + \|\widehat{F}_{\nu'}\|^p_{L^{2,p}(\tilde{Q}_{\nu'})} \notag{} \\
& \approx \delta^{2-p}_\nu |\nabla L_\nu - \nabla L_{\nu'}|^p + M^p_\nu + M^p_{\nu'}. \notag{}
\end{align}
Finally, we examine a term in \eqref{term3}. As before, for any $\nu$ with $\nu \leftrightarrow \nu'$, we find
\begin{align}
\label{term3a} 
\| (\widehat{F}_\nu - \widehat{F}_{\nu'}) \nabla^2(\theta_\nu)\|^p_{L^p( Q_{\nu'})} & \lesssim \delta^{-2p}_\nu \| L_\nu - L_{\nu'}\|^p_{L^p( Q_{\nu'})} + \|\widehat{F}_\nu\|^p_{L^{2,p}(\tilde{Q}_\nu)} + \|\widehat{F}_{\nu'}\|^p_{L^{2,p}(\tilde{Q}_{\nu'})}\\
& \lesssim \delta^{2-2p}_\nu | L_\nu (x_{\nu}) - L_{\nu'} (x_{\nu})|^p + \delta^{2-p}_\nu |\nabla L_\nu - \nabla L_{\nu'}|^p +  M^p_\nu + M^p_{\nu'}, \notag{}
\end{align}
with the last inequality following since
\begin{equation*}
\| L_\nu - L_{\nu'}\|^p_{L^p( Q_{\nu'})} \lesssim \delta^2_\nu |L_\nu(x_{\nu}) - L_{\nu'}(x_{\nu}) |^p + \delta^{2+p}_\nu  |\nabla L_\nu - \nabla L_{\nu'}|^p.
\end{equation*}
Inserting the bounds \eqref{term1a},\eqref{term2a},\eqref{term3a} into \eqref{term1},\eqref{term2},\eqref{term3}, we find
\begin{equation*}
\|\nabla^2 \overline{F}\|^p_{L^p( Q_{\nu'})} \lesssim \sum_{\nu:\nu \leftrightarrow \nu'} \widehat{M}^p_\nu + \sum_{\nu: \nu \leftrightarrow \nu'} \left[ |\nabla L_\nu - \nabla L_{\nu'}|^p \delta_\nu^{2-p} + |L_\nu(x_{\nu}) - L_{\nu'}(x_{\nu})|^p  \delta^{2-2p}_\nu \right].
\end{equation*}
And finally, summing over $\nu' \in \{1,2,\cdots,K\}$,
\begin{equation*}
\|\overline{F}\|^p_{L^{2,p}(Q^\circ)} \lesssim \sum_{\nu} \widehat{M}^p_{\nu} + \sum_{\nu \leftrightarrow \nu'} \left[ |\nabla L_\nu - \nabla L_{\nu'}|^p \delta^{2-p}_\nu + |L_\nu(x_\nu) - L_{\nu'}(x_\nu)|^p \delta^{2-2p}_\nu \right].
\end{equation*}
Here, we have used that for each $\nu'$ there are at most $C$ indices $\nu$ with $\nu \leftrightarrow \nu'$. This implies that each pair $(\nu,\nu')$ with $\nu \leftrightarrow \nu'$ is over-counted at most $2C$ times in a sum of the form $\displaystyle \sum_{\nu'} \sum_{\nu : \nu \leftrightarrow \nu'}$. This completes the proof of Lemma \ref{upperboundlem}.
\end{proof}

Having just established Lemma \ref{upperboundlem}, we now extend $\overline{F}$ to a function in $L^{2,p}(\R^2)$ without increasing its norm by more than a constant factor. As usual, all methods will be linear in the initial data. 

Recall that $\overline{F}|_E = f$, and $J_{E'} \overline{F} = L$. Let $\theta \in C^\infty_c(Q^\circ)$ satisfy $(1)$ $\theta \equiv 1$ on $0.99 Q^\circ$, and $(2)$ $|\partial^{\alpha} \theta| \lesssim \delta_{Q^\circ}^{-|\alpha|}$ for $|\alpha| \leq 2$. Let $L^\circ= J_0 \overline{F}$, which obviously depends linearly on $f$ and $L^\#$. Now, define 
$$\widehat{F} = \theta \overline{F} + (1 - \theta)L^\circ = \theta (\overline{F} - L^\circ) + L^\circ.$$ 
Clearly $\widehat{F}$ depends linearly on $f$ and $L^\#$. Recall that $E \subset \frac{1}{10} Q^\circ \subset 0.99 Q^\circ$, and Lemma \ref{keyptlem} implies that $E' \subset 0.99Q^\circ$. Therefore, $P1$ of $\theta$ implies that $\widehat{F}|_E = f$, and $J_{E'} \widehat{F} = L$. Now, the SET, $\mbox{supp}(\theta) \subset Q^\circ$, and $P2$ of $\theta$ give
\begin{align}
\label{bound2}
\|\widehat{F}\|_{L^{2,p}(\R^2)} = \|\theta (\overline{F} - L^\circ)\|_{L^{2,p}(Q^\circ)} & \lesssim \|\nabla^2 \theta (\overline{F}-L^\circ)\|_{L^p(Q^\circ)} + \|\nabla \theta \otimes \nabla(\overline{F}-L^\circ)\|_{L^p(Q^\circ)} \\
& \qquad\qquad + \|\theta \nabla^2(\overline{F}-L^\circ)\|_{L^p(Q^\circ)} \lesssim \|\overline{F}\|_{L^{2,p}(Q^\circ)}. \notag{} 
\end{align}
Thus, \eqref{fullbound1} and \eqref{bound2} imply that
\begin{equation}
\label{bound3}
\|\widehat{F}\|^p_{\WR} \lesssim \sum_{\nu} \widehat{M}^p_\nu +  \sum_{\nu \leftrightarrow \nu'} \left[ |\nabla L_\nu - \nabla L_{\nu'}|^p \delta^{2-p}_\nu + |L_\nu(x_\nu) - L_{\nu'}(x_\nu)|^p \delta^{2-2p}_\nu \right].
\end{equation}
Therefore $\widehat{F}$ satisfies \textbf{(MEPa-b)} with the above control on the $\WR$-norm. The following result states that any function satisfying \textbf{(MEPa-b)} must have $\WR$-norm at \textit{least} as large as the quantity on the RHS of \eqref{bound3}. This will imply that $\|\widehat{F}\|_{\WR}$ is C-optimal with respect to \textbf{(MEPa-b)}, and moreover that $\|\widehat{F}\|_{\WR} \approx \widehat{M}(f,L^\#)$ with $\widehat{M}(f,L^\#)$ given by the RHS of \eqref{bound3}.

\begin{lem}
For any $F \in \WR$ with $F|_E = f$, and $J_{E'} F = L$, we have
\begin{equation*}
\|F\|^p_{\WR} \gtrsim \sum_{\nu} \widehat{M}^p_\nu + \sum_{\nu \leftrightarrow \nu'} \left[|\nabla L_\nu - \nabla L_{\nu'}|^p \delta^{2-p}_\nu + |L_\nu(x_\nu) - L_{\nu'}(x_\nu)|^p \delta^{2-2p}_\nu \right].
\end{equation*}
\label{lowerboundlem}
\end{lem}

\begin{proof}
First, we note that
\begin{equation}
\label{ineq20}
\sum_{\nu} \widehat{M}_\nu^p \approx \sum_{\nu} \int_{\tilde{Q}_\nu} |\nabla^2 \widehat{F}_\nu(x)|^p dx \lesssim \sum_{\nu} \int_{\tilde{Q}_\nu} |\nabla^2 F(x)|^p dx \lesssim \int_{\R^2} |\nabla^2 F(x)|^p dx = \|F\|^p_{\WR}. \notag{}
\end{equation}
Here, the first ``$\lesssim$" follows from the essential optimality of the functions $\widehat{F}_\nu$, while the second follows from the Bounded Intersection Property of $(\tilde{Q}_\nu)_{\nu=1}^{K}$.

Applying the SET on the domain $\tilde{Q}_\nu \cup \tilde{Q}_{\nu'}$ for $\nu \leftrightarrow \nu'$ (see Remark \ref{twosquares}), we find
\begin{align}
\label{ineq21}
\sum_{\nu \leftrightarrow \nu'} \left[|\nabla L_\nu - \nabla L_{\nu'}|^p \delta^{2-p}_\nu + |L_\nu(x_\nu) - L_{\nu'}(x_\nu)|^p \delta^{2-2p}_\nu \right] & \lesssim \sum_{\nu \leftrightarrow \nu'} \|F\|^p_{L^{2,p}(\tilde{Q}_\nu \cup \tilde{Q}_{\nu'})} \\
& \lesssim \sum_{\nu} \|F\|^p_{L^{2,p}(\tilde{Q}_\nu)} \lesssim \|F\|^p_{\WR}. \notag{}
\end{align}
Here, the last two inequalities follow from the Bounded Intersection Property of $\{\tilde{Q}_\nu\}_{\nu=1}^K$. Together, \eqref{ineq20} and \eqref{ineq21} imply the lemma.
\end{proof}

Lemma \ref{upperboundlem} and Lemma \ref{lowerboundlem} imply that the function $\widehat{F} = \widehat{T}(f,L^\#)$ is a solution to the \textbf{(MEP)} with $L^{2,p}(\R^2)$-norm given up to a multiplicative universal constant by the RHS of \eqref{bound3}; Thus, we may take $\widehat{M}(f,L^\#)$ as in Proposition \ref{mainprop}. This completes the proof of Proposition \ref{mainprop}.

\section{Determining the optimal Whitney Field $L^\#$}
\label{findwhitfield}
The goal of this section is to find a jet $L^\# \in Wh(E^\#)$ depending linearly on $f$ which approximately minimizes the formula for $\widehat{M}$ from Proposition \ref{mainprop}. Recall that an extension of $f$ with essentially optimal $L^{2,p}(\R^2)$-norm is related to a solution of the \textbf{(MEP)} by Lemma \ref{minilem}, which states that
\begin{equation*}
\|f\|_{L^{2,p}(\R^2)|_E} \approx \inf_{L^\#} \widehat{M}(f,L^\#).
\end{equation*}
We will choose such an $L^\#$ through the following lemma. Afterward, we prove that it is an approximate minimum for $\widehat{M}(f,\cdot)$.
\begin{lem}
There exists a sufficiently small universal constant $c'>0$ such that the following holds: Let $E_0 \subset 0.9Q$, and $x_0 \in \frac{1}{2}Q$ satisfy $d(x_0,E_0) \geq \frac{1}{100}\delta_Q$. Moreover, suppose that $Q$ satisfies \textbf{R}$(c,c',c'')$ relative to $E_0$ for some universal constants $c,c''>0$. Then there is a linear operator $T_1:L^{2,p}(Q)|_{E_0} \rightarrow P$ with $T_1(f_0) \in \Gamma_Q(f_0,x_0,C \|f_0\|_{L^{2,p}(Q)|_{E_0}})$ for some large universal constant $C$, and all $f_0:E_0 \rightarrow \R$.
\label{keylem}
\end{lem}

\begin{proof}
Through a standard rescaling argument we may assume that $Q=[-1/2,1/2]^2$. Let $f_0:E_0 \rightarrow \R$ be given. Since $Q$ satisfies \textbf{R}$(c,c',c'')$ relative to $E_0$, there are two cases to consider.

\noindent\textbf{Case 1:} There exist two unit vectors $v_1 = \frac{x_1-x_2}{|x_1 - x_2|}$, and $v_2 = \frac{y_1 - y_2}{|y_1 - y_2|}$ with $x_i,y_i \in E_0$ and so that $\min \{|v_1 - v_2|,|v_1 + v_2|\} > c''$. 

We form the matrix with rows given by $v_1$ and $v_2$:
\[ M_1 = \left( \begin{array}{ccc}
\; & v_1 & \; \\
\; & v_2 & \; \end{array} \right) \]
Notice that the condition $\min \{|v_1 - v_2|,|v_1 + v_2|\} > c''$ implies that $M_2 = M_1^{-1}$ has entries bounded by a universal constant.

Set $m_1 = \frac{f_0(x_1) - f_0(x_2)}{|x_1 - x_2|}$, and $m_2 = \frac{f_0(y_1) - f_0(y_2)}{|y_1 - y_2|}$. Choose $A \in \R^2$ given by $A^T = M_2 (m_1,m_2)^T$. Let $T_1(f_0) = L_1 \in P$ be the unique affine function satisfying $\nabla L_1 = A$ and $L_1(x_1) = f_0(x_1)$. Clearly, $L_1$ depends linearly on $f_0$. We now show that $L_1$ is in the desired $\Gamma_Q$.

Let $F \in L^{2,p}(Q)$ satisfy: $(1)$ $F|_{E_0} = f_0$, and $(2)$ $\|F\|_{L^{2,p}(Q)} \leq 2 \|f_0\|_{L^{2,p}(Q)|_{E_0}}$. The mean-value theorem implies the existence of two points $x^*, y^* \in Q$ with $v_1 \cdot \nabla F(x^*) = m_1$, and $v_2 \cdot \nabla F(y^*) = m_2$. Thus, from the SET,
$$|v_1 \cdot \nabla F(x_0) - m_1|, |v_2 \cdot \nabla F(x_0) - m_2| \lesssim \|F\|_{L^{2,p}(Q)}.$$
Equivalently, we may write $|M_1 (\nabla F(x_0))^T - (m_1,m_2)^T| \lesssim \|F\|_{L^{2,p}(Q)}$. Since the entries of $M_2=M^{-1}_1$ are bounded by a universal constant, this implies $|(\nabla F(x_0))^T - M_2 (m_1,m_2)^T| \lesssim \|F\|_{L^{2,p}(Q)}$. Thus, 
\begin{equation}
|\nabla F(x_0) - \nabla L_1| = |\nabla F(x_0) - A| = |\nabla F(x_0) - (m_1,m_2) M^T_2| \lesssim \|F\|_{L^{2,p}(Q)}.
\label{yoyo1}
\end{equation}
From this it follows that
\begin{align}
\label{yoyo2}
|F(x_0) - L_1(x_0)| & \leq |F(x_1) - L_1(x_1)| + |(F(x_0) - L_1(x_0)) - (F(x_1) - L_1(x_1)))| \\
& = |(F(x_0) - F(x_1)) + (A \cdot(x_1-x_0))| \leq |F(x_0) - J_{x_0}F(x_1) + A \cdot(x_1-x_0)| \notag{} \\
& + |J_{x_0}F(x_1) - F(x_1)| \lesssim \|F\|_{L^{2,p}(Q)}, \notag{}
\end{align}
with the ``$=$" following since $F(x_1)=f_0(x_1)=L_1(x_1)$, and the ``$\lesssim$" following from $F(x_0) - J_{x_0}F(x_1) = - \nabla F(x_0) \cdot (x_1-x_0)$, the SET, and \eqref{yoyo1}. Note that \eqref{yoyo1} and \eqref{yoyo2} imply $\|J_{x_0}F - L_1\|_{W^{2,p}(Q)} \lesssim \|F\|_{L^{2,p}(Q)}$. 

Now, choose a cutoff function $\theta_1 \in C_c^\infty(B(x_0,\frac{1}{100}))$ with: $(1)$ $\theta_1 \equiv 1$ on $B(x_0,\frac{1}{200})$, and $(2)$ $|\partial^{\alpha} \theta_1| \lesssim 1$ for all $|\alpha| \leq 2$. Since $d(x_0,E_0) \geq \frac{1}{100}$, one has $\theta_1|_{E_0} = 0$. Let $G = F + \theta_1 (L_1 - J_{x_0}F)$, for which
\begin{enumerate}
\item $G|_{E_0} = F|_{E_0} + 0 = f_0$;
\item $\|G\|_{L^{2,p}(Q)} \lesssim \|F\|_{L^{2,p}(Q)} \approx \|f_0\|_{L^{2,p}(Q)|_{E_0}}$;
\item $J_{x_0} G = J_{x_0} F + L_1 - J_{x_0} F = L_1$.
\end{enumerate}
Thus, $L_1 \in \Gamma_Q(f_0,x_0,C\|f_0\|_{L^{2,p}(Q)|_{E_0}})$ for some large universal constant $C$.

\textbf{Case 2:} $\|E_0\|_{\bes} \leq c'$, with $c'$ sufficiently small as mentioned in the hypotheses.

By choosing $c'$ sufficiently small, one can arrange for the Geometric Assumptions from Section \ref{LMEP} to be satisfied. Thus Theorem \ref{localthm} applies, and so there exists a bounded linear operator $\widehat{T}_0:L^{2,p}(Q)|_{E_0} \times P \rightarrow L^{2,p}(Q)$, and a non-negative real number $\widehat{M}_Q(\cdot,\cdot)$, with $(1)$ $\widehat{T}_0(f_0,L_0)|_{E_0} = f_0$, $(2)$ $J_{x_0}\widehat{T}_0(f_0,L_0) = L_0$, $(3)$ $\|\widehat{T}_0(f_0,L_0)\|^p_{L^{2,p}(Q)}$ is $C$-optimal with respect the the two properties above, and $(4)$ $\|\widehat{T}_0(f_0,L_0)\|^p_{L^{2,p}(Q)} \approx \widehat{M}_Q(f_0,L_0)^p  = \sum_{i=1}^{N_0} |\lambda_i(f_0,L_0)|^p$.

From the defining characteristics of $\widehat{M}_Q$ given above ($P3$ and $P4$), we find that $\inf \{\widehat{M}_Q(f_0,L_0): L_0 \in P\} \approx \|f_0\|_{L^{2,p}(Q)}$. In fact, a stronger statement is true: For any $L_1 \in P$,
\begin{equation}
\label{bnorm0}
\widehat{M}_Q(f_0,L_1) \approx \inf \{\widehat{M}_Q(f_0,L_0): L_0 \in P\} \; \Rightarrow \; L_1 \in \Gamma_Q(f_0,x_0,C\|f_0\|_{L^{2,p}(Q)|_{E_0}}),
\end{equation}
for some universal constant $C$. Suppose that could find $L_1 \in P$ with
\begin{equation}
\label{bnorm1}
\sum_{i=1}^{N_0} |\lambda_i(f_0,L_1)|^p \approx \inf \{\sum_{i=1}^{N_0} |\lambda_i(f_0,L_1)|^p: L_0 \in P\}.
\end{equation}
From \eqref{bnorm0}, and the given form of $\widehat{M}_Q$, such an $L_1$ would lie in the desired $\Gamma_Q$. Thus, we are lead to the problem of minimizing variant $l^p$-norms such as those appearing on the RHS of \eqref{bnorm1}. The following claim will be useful for this purpose.

\textbf{Claim 1:}
Let $\beta=(\beta_i)_{i=1}^{N_0} \in \R^{N_0}$ be given. Then $S: \R^{N_0} \rightarrow \R$ defined by $\displaystyle S(z) = \frac{1}{\|\beta\|^p_{l^p}}\sum_{j=1}^{N_0} \frac{z_j}{\beta_j} |\beta_j|^p$ for $\beta \neq 0$ (and $S(z)=0$ for $\beta=0$) satisfies:
\begin{equation}
\label{yodude2}
\sum^{N_0}_{i=1} |z_i - \beta_i S(z)|^p \lesssim \min_{a \in \R} \{\sum^{N_0}_{i=1} |z_i - \beta_i a|^p\}.
\end{equation}
In order to prove Claim 1, first note that we may assume without loss that $\beta_i \neq 0$ for all $i=1,\cdots,N_0$. This follows since the terms in \eqref{yodude2} for which $\beta_i=0$ have no effect on the minimization problem, and can be removed by projecting out coordinates. In the case when $\beta=0$, obviously $S(z)=0$ solves the problem of interest. The more general case follows from an application of H\"older's inequality, which we leave out for sake of brevity.

Expand an arbitrary $L \in P$ in coordinates as $L(u,v) = a_1 u + a_2 v + b$, for $(u,v)$ any Euclidean coordinates. As in the RHS of \eqref{bnorm1}, we would like to minimize $|\sum_{i} \lambda_i(f_0,a_1u + a_2v + b)|^p$ over choices of $a_1,a_2$, and $b$. We apply Claim 1 three consecutive times to solve for optimal choices for each of the coefficients $a_1,a_2$, and $b$ in terms of the remaining coefficients and the function $f_0$. In this way, we get a linear map $S: L^{2,p}(\R^2)|_{E_0} \rightarrow \R^3$ with $(\overline{a}_1,\overline{a}_2,\overline{b})=S(f_0)$ satisfying
\begin{equation*}
\sum^{N_0}_{i=1} |\lambda_i(f_0,\overline{a}_1 u + \overline{a}_2 v + \overline{b})|^p \approx \inf\{ \sum^{N_0}_{i=1} |\lambda_i(f_0, L_0)|^p : L_0 \in P\}.
\end{equation*}
Set $L_1 = \overline{a}_1 u + \overline{a}_2 v + b$, which by \eqref{bnorm0} and \eqref{bnorm1} satisfies $L_1 \in \Gamma_Q(f_0,x_0,C\|f_0\|_{L^{2,p}(Q)|_{E_0}})$ as desired. The proof of the lemma for Case 2 is now complete.
\end{proof}
The next result will be useful in our proof of optimality for the output of the preceding lemma.
\begin{lem}
Let $Q \subset \R^2$ and $E_0 \subset 0.9Q$ be given. Let $x_0 \in \frac{1}{2}Q$ satisfy $d(x_0,E_0) \geq \frac{1}{100}\delta_Q$. Let $L \in P$ be given. Recall that $\widehat{M}_Q(0,L) \approx \inf \{\|h\|_{L^{2,p}(Q)} : h|_{E_0}=0, J_{x_0} h = L\}$. Then 
\begin{equation*}
\widehat{M}_Q(0,L)^p \lesssim |L(x_0)|^p \delta_Q^{2-2p} + |\nabla L|^p \delta_Q^{2-p}
\end{equation*}
Moreover, suppose that  $\|E_0\|_{\bes} \geq c \delta_Q^{1-2/p}$ for some given universal constant $c>0$, then
\begin{equation*}
\widehat{M}_Q(0,L)^p \approx |L(x_0)|^p \delta_Q^{2-2p} + |\nabla L|^p \delta_Q^{2-p}.
\end{equation*}
\label{lem100}
\end{lem}
\begin{proof}
Through a standard rescaling argument we may assume that $Q=[-1/2,1/2]^2$.

Let $\theta \in C_c^{\infty} (B(x_0,\frac{1}{150}))$ satisfy $(1)$ $0 \leq \theta \leq 1$, $(2)$ $\theta \equiv 1$ on $B(x_0,\frac{1}{200})$, and $(3)$ $|\partial^{\alpha} \theta| \lesssim 1$ for all $\alpha$ with $|\alpha| \leq 2$. Then $\tilde{h} = \theta L$ satisfies the necessary properties to be included among the functions in the infimum defining $\widehat{M}_Q(0,L)$. Also, a simple calculation shows that $\|\tilde{h}\|^p_{L^{2,p}(Q)} \lesssim |L(x_0)|^p + |\nabla L|^p$. Thus,
\begin{equation}
\label{e99}
\widehat{M}_Q(0,L)^p \lesssim |L(x_0)|^p + |\nabla L|^p.
\end{equation}
We now assume that $\|E_0\|_{\bes} \geq c$, and show 
\begin{equation}
|L(x_0)|^p + |\nabla L|^p \lesssim \|h\|^p_{L^{2,p}(Q)} \approx \widehat{M}_Q(0,L)^p.
\label{e100}
\end{equation}
Let $h \in L^{2,p}(Q)$ essentially attain the infimum in the definition of $\widehat{M}_Q(0,L)$; that is:  $(1)$ $J_{x_0} h = L$, $(2)$ $h|_{E_0} = 0$, and $(3)$ $\|h\|_{L^{2,p}(Q)} \leq 2 \widehat{M}_Q(0,L)$.

If $L = 0$, then \eqref{e100} obviously holds, and we are done. Therefore we may assume that $L \neq 0$. As a consequence, we now show that $\widehat{M}_Q(0,L) \neq 0$. To see this, suppose for sake of contradiction that $\widehat{M}_Q(0,L)=0$. Then $P3$ of $h$ would imply that $h$ is affine, and in particular not equal to the zero function (by $P1$ of $h$). Thus, $P2$ of $h$ would imply that $E_0$ is contained in a line (the zero set of $h$), and so $\|E_0\|_{\bes}=0$. This contradicts the assumption that $\|E_0\|_{\bes} \geq c$. For the remainder, we now assume that $L \neq 0$ and $\widehat{M}_Q(0,L) \neq 0$.

Since $E_0 \neq \emptyset$, we may choose $\overline{x}_0 \in E_0$. The SET implies that
\begin{equation*}
|L(\overline{x}_0)| = |J_{x_0}h(\overline{x}_0) - h(\overline{x}_0)| \lesssim \|h\|_{L^{2,p}(Q)},
\end{equation*}
and so
\begin{equation}
|L(x_0)| \lesssim |L(\overline{x}_0)| + |\nabla L| \lesssim \|h\|_{L^{2,p}(Q)} + |\nabla L|
\label{e101}
\end{equation}
If we show that $|\nabla L| \lesssim \widehat{M}_Q(0,L)$, then \eqref{e100} will follow from \eqref{e101}. For sake of contradiction, suppose that $|\nabla L| > C_1 \widehat{M}_Q(0,L)$ for some sufficiently large universal constant $C_1$. The SET and the previously stated properties of $h$ imply that
\begin{equation*}
|\nabla h(\overline{x}_0)| \geq |\nabla h(x_0)| - C \|h\|_{L^{2,p}(Q)} \geq |\nabla L| - C' \widehat{M}_Q(0,L) > (C_1 - C')\widehat{M}_Q(0,L).
\end{equation*}
By choosing $C_1$ sufficiently large depending on $C_2$, we can arrange for $\overline{h} = \frac{h}{C_2 \widehat{M}_Q(0,L)}$ to satisfy: $(1)$ $\overline{h}|_{E_0} = 0$, in particular $h(\overline{x}_0) = 0$, $(2)$ $|\nabla \overline{h}(\overline{x}_0)| \geq 1$, and $(3)$ $\|\overline{h}\|_{L^{2,p}(Q)} \leq 2/C_2$. We now apply Lemma \ref{impft} to $\gamma = \{x \in 0.9Q: \overline{h}(x)=0\}$, which implies $\|\gamma\|_{\bes} \lesssim 2/C_2$. By choosing $C_2$ sufficiently large, we can arrange
$$\|\gamma\|_{\bes} \leq \frac{c}{2}.$$
But recall that $E_0 \subset \gamma$ from $P1$ of $\overline{h}$, and so $\|E_0\|_{\bes} \leq \|\gamma\|_{\bes} \leq \frac{c}{2}$. But, this contradicts $\|E_0\|_{\bes} \geq c$. Therefore, we have shown $|\nabla L| \lesssim \widehat{M}_Q(0,L)$. Together with \eqref{e101}, this completes the proof of \eqref{e100}, and thus also the second half of the lemma.
\end{proof}

We set $E^\#_\mu = 9Q^\#_\mu \cap E$. Recall \textbf{(K2)}, which states that $9Q^\#_\mu$ satisfies \textbf{R}$(c_1,c_3,c_2)$ relative to $E^\#_\mu$. Equivalently, $10Q^\#_\mu$ satisfies \textbf{R}$(c_1(9/10)^{2/p-1},c_3(9/10)^{2/p-1},c_2)$ relative to $E^\#_\mu$. Also, the \textit{Keystone representative point} $x^\#_\mu$ (see Section \ref{points}) satisfies:
$$x^\#_\mu \in \frac{1}{2}Q^\#_\mu \; \mbox{and} \; d(x^\#_\mu,E^\#_\mu) \geq d(x^\#_\mu,E) \geq \frac{1}{5}\delta_{Q^\#_\mu} \geq \frac{1}{100} \delta_{10Q^\#_\mu}.$$

For $c_3$ small enough, Lemma \ref{keylem} applies, and so for each $1 \leq \mu \leq K^\#$ there exists a linear operator $T^\#_{\mu} : L^{2,p}(10Q^\#_\mu)|_{E^\#_\mu} \rightarrow P$ with
\begin{equation*}
\tilde{L}^\#_\mu = T^\#_\mu(f|_{E^\#_\mu}) \in \Gamma_{10Q^\#_\mu}(f|_{E^\#_\mu}, x^\#_\mu, C\|f|_{E^\#_\mu}\|_{L^{2,p}(10Q^\#_\mu)|_{E^\#_\mu}}).
\end{equation*}

\begin{lem}
\label{bestposslem}
Let $\tilde{L}^\#$ be defined as above. Then $\widehat{M}(f,\tilde{L}^\#) \lesssim \widehat{M}(f,L^\#)$ for all $L^\# \in Wh(E^\#)$.
\end{lem}

\begin{proof}
Fix an arbitrary $L^\# \in Wh(E^\#)$. Using the formula for $\widehat{M}(f,\tilde{L}^\#)$ from Proposition \ref{mainprop}, we write
\begin{align}
\label{e119}
\widehat{M}(f,\tilde{L}^\#)^p \approx & \sum_{\nu} \widehat{M}_\nu(f_\nu,\tilde{L}_\nu)^p + \sum_{\nu \leftrightarrow \nu'} \left[ |\nabla \tilde{L}_\nu - \nabla \tilde{L}_{\nu'}|^p \delta^{2-p}_\nu + |\tilde{L}_\nu(x_\nu) - \tilde{L}_{\nu'}(x_\nu)|^p \delta^{2-2p}_\nu \right]  \\
& \lesssim \notag{}\\
& \sum_{\nu} \widehat{M}_\nu(f_\nu,L_\nu)^p + \sum_{\nu \leftrightarrow \nu'} \left[|\nabla L_\nu - \nabla L_{\nu'}|^p \delta^{2-p}_\nu + |L_\nu(x_\nu) - L_{\nu'}(x_\nu)|^p \delta^{2-2p}_\nu \right] \notag{}\\
& + \notag{}\\
& \sum_{\nu} \widehat{M}_\nu(0,\tilde{L}_\nu - L_\nu)^p +  \sum_{\nu \leftrightarrow \nu'}  \left[|\nabla \tilde{L}_\nu - \nabla L_\nu|^p \delta^{2-p}_\nu + |(\tilde{L}_\nu - L_\nu)(x_\nu)|^p \delta^{2-2p}_\nu \right]. \notag{}
\end{align}
In the above we have used the subadditivity of $\widehat{M}_\nu(\cdot,\cdot)$ and the following estimation: 
$$|(\tilde{L}_{\nu'} - L_{\nu'})(x_\nu)|^p \delta^{2-2p}_\nu \lesssim |(\tilde{L}_{\nu'} - L_{\nu'})(x_{\nu'})|^p \delta^{2-2p}_{\nu'} + |\nabla \tilde{L}_{\nu'} - \nabla L_{\nu'}|^p \delta^{2-p}_{\nu'},$$
for any $\nu \leftrightarrow \nu$. Through use of the formula for $\widehat{M}(f,L^\#)$, a collapse of double sums into single sums using the finite intersection property of the squares $\CZ$, and Lemma \ref{lem100},
\begin{align}
\label{e120}
\widehat{M}(f,\tilde{L}^\#)^p \lesssim \widehat{M}(f,L^\#)^p + \sum_{\nu} \left[ \widehat{M}_\nu(0,L_\nu - \tilde{L}_\nu) + |\nabla L_\nu - \nabla \tilde{L}_\nu|^p \delta_\nu^{2-p} + |L_\nu(x_\nu) - \tilde{L}_\nu(x_\nu)|^p \delta_\nu^{2-2p}\right] \\
\lesssim \widehat{M}(f,L^\#)^p + \sum_{\nu} \left[ |\nabla L_\nu - \nabla \tilde{L}_\nu|^p \delta_\nu^{2-p} + |L_\nu(x_\nu) - \tilde{L}_\nu(x_\nu)|^p \delta_\nu^{2-2p} \right]. \notag{}
\end{align}

Set $X = \sum_{\nu} \left[ |\nabla L_\nu - \nabla \tilde{L}_\nu|^p \delta_\nu^{2-p} + |L_\nu(x_\nu) - \tilde{L}_\nu(x_\nu)|^p \delta_\nu^{2-2p} \right]$. Once we prove that $X \lesssim \widehat{M}(f,L^\#)^p$, the proposition will follow. Using the fact that $L \in Wh(E')$ is the constant-path extension of $L^\# \in Wh(E^\#)$,
\begin{equation}
\label{e121}
X = \sum_{\mu}\sum_{\nu:\mu(\nu)=\mu} \left[ |\nabla L^\#_\mu - \nabla \tilde{L}^\#_\mu|^p \delta_\nu^{2-p} + |L^\#_\mu(x_\nu) - \tilde{L}^\#_\mu(x_\nu)|^p \delta_\nu^{2-2p} \right]
\end{equation}
Now, Lemma \ref{stdkeyprops} implies that $|x_\nu - x^\#_{\mu(\nu)}| \lesssim \delta_{\nu}$, and so for $\mu = \mu(\nu)$,
$$|L^\#_\mu(x_\nu) - \tilde{L}^\#_\mu(x_\nu)|^p \lesssim |L^\#_\mu(x^\#_\mu) - \tilde{L}^\#_\mu(x^\#_\mu)|^p + |\nabla L^\#_\mu - \nabla \tilde{L}^\#_\mu|^p \delta_\nu^p,$$
Upon plugging this into \eqref{e121},
\begin{equation}
X \lesssim \sum_{\mu}\left[ |\nabla L^\#_\mu - \nabla \tilde{L}^\#_\mu|^p \sum_{\nu:\mu(\nu)=\mu} \delta_\nu^{2-p} + |L^\#_\mu(x^\#_\mu) - \tilde{L}^\#_\mu(x^\#_\mu)|^p \sum_{\nu:\mu(\nu)=\mu} \delta_\nu^{2-2p} \right]
\label{xeq}
\end{equation}
Recall that $\sum_{\nu:\mu(\nu)=\mu} \delta_\nu^{-\epsilon}$ behaves like a geometric series for any $\epsilon > 0$. In fact, by a special case of \eqref{expsum} which can be found in Lemma \ref{globsobineq}, we have
\begin{equation}
\label{geomser}
\sum_{\nu: \mu(\nu) = \mu} \delta^{-\epsilon}_{\nu} \lesssim (\delta^\#_{\mu})^{-\epsilon}.
\end{equation}
Applying \eqref{geomser} for $\epsilon=p-2$, and $\epsilon = 2p-2$, we transform \eqref{xeq} into
\begin{equation}
\label{xeq2}
X \lesssim \sum_{\mu}\left[ |\nabla L^\#_\mu - \nabla \tilde{L}^\#_\mu|^p(\delta^\#_\mu)^{2-p} + |L^\#_\mu(x^\#_\mu) - \tilde{L}^\#_\mu(x^\#_\mu)|^p (\delta^\#_\mu)^{2-2p} \right].
\end{equation}
Recall that $10Q^\#_\mu$ satisfies \textbf{R}$(c_1(9/10)^{2/p-1},c_3(9/10)^{2/p-1},c_2)$ relative to $E^\#_\mu$, which by Lemma \ref{kap} implies that $\|E^\#_\mu\|_{\bes} \gtrsim \delta_{Q^\#_\mu}^{2/p-1}$ (with constant dependent on $c_1$, $c_2$, and $c_3$). Thus Lemma \ref{lem100} applies, and 
\begin{equation}
\label{xeq3}
X \lesssim \sum_{\mu} \widehat{M}_{10Q^\#_\mu}(0,L^\#_\mu - \tilde{L}^\#_\mu)^p \lesssim \sum_{\mu} \left[\widehat{M}_{10Q^\#_\mu}(f|_{E^\#_\mu},L^\#_\mu)^p + \widehat{M}_{10Q^\#_\mu}(f|_{E^\#_\mu}, \tilde{L}^\#_\mu)^p\right],
\end{equation}
with the last inequality following from subadditivity of $\widehat{M}_Q$. Since $\tilde{L}^\#_\mu \in \Gamma_{10Q^\#_\mu}(f|_{E^\#_\mu},x^\#_\mu,C\|f|_{E^\#_\mu}\|)$, we have that $\widehat{M}_{10Q^\#_\mu}(f|_{E^\#_\mu}, \tilde{L}^\#_\mu) \lesssim \widehat{M}_{10Q^\#_\mu}(f|_{E^\#_\mu}, L^\#_\mu)$. Thus, \eqref{xeq3} implies that
\begin{equation}
\label{xeq4}
X \lesssim \sum_{\mu} \widehat{M}_{10Q^\#_\mu}(f|_{E^\#_\mu}, L^\#_\mu)^p.
\end{equation}
Finally, from the properties of $\widehat{M}_{10Q^\#_\mu}$ and $\widehat{M}$ in Theorem \ref{localthm} and Proposition \ref{mainprop}, and the Bounded Intersection Property of $\{10Q^\#_\mu\}_{\mu=1}^{K^\#}$ (see \textbf{(K3)}),
$$X \lesssim \widehat{M}(f,L^\#)^p.$$
As we already noted, the lemma follows from this fact.
\end{proof}

\section{Proof of Theorem \ref{mainthm} when $\Lambda \neq \{Q^\circ\}$}
Consider a finite $E \subset \frac{1}{10} Q^\circ \subset \R^2$, and $f:E \rightarrow \R$. We construct CZ squares $\Lambda$, and Keystone squares $\Lambda^\# \subset \Lambda$. In this section, we suppose that $\Lambda \neq \emptyset$, so that Proposition \ref{iskey} is valid. As described in Section \ref{points}, we construct the various representative points: $E' = \{x_\nu\}^K_{\nu=1}$, and $E^\# = \{x^\#_\mu\}^{K^\#}_{\mu=1}$. Let $E^\#_\mu = 9Q^\#_\mu \cap E$, $E_\nu = E \cap 1.1Q_\nu$, and $f_\nu = f|_{E_\nu}$. From the previous section we found a linear jet-valued map $\tilde{L}^\#_\mu = T^\#_\mu(f|_{E^\#_\mu}) \in \Gamma_{10Q^\#}(\cdots)$ for each $\mu=1,2,\cdots,K^\#$, which by Lemma \ref{bestposslem} and Lemma \ref{minilem} satisfies
\begin{equation*}
\widehat{M}(f,\tilde{L}^\#) \lesssim \inf\{ \widehat{M}(f,L^\#) : L^\# \in Wh(E^\#)\} \approx \|f\|_{L^{2,p}(\R^2)|_E}.
\end{equation*}
Note that the reverse inequality follows trivially from the fact that $\widehat{M}(f,\tilde{L}^\#)$ is approximately the norm of an interpolant of $f$ (see Proposition \ref{mainprop}), and so
\begin{equation}
\label{e149}
\widehat{M}(f,\tilde{L}^\#) \approx \|f\|_{L^{2,p}(\R^2)|_E}.
\end{equation}
Let $\tilde{L}$ be the constant-path extension of $\tilde{L}^\#$. From Proposition \ref{mainprop} we find a linear operator $\widehat{T}:\WR|_E \times Wh(E^\#) \rightarrow \WR$ and a positive real number $\widehat{M}(\cdot,\cdot)$ which for any $L^\# \in Wh(E^\#)$ satisfy: $(1)$ $\widehat{T}(f,L^\#)|_E = f$, $(2)$ $J_{E'} \widehat{T}(f,L^\#) = L$, and $(3)$ $\|\widehat{T}(f,L^\#)\|_{\WR} \approx \widehat{M}(f,L^\#)$ is C-optimal with respect to these two properties.

Define $T(f) = \widehat{T}(f,\tilde{L}^\#)$. From \eqref{e149}, and $P1$-$3$ of $\widehat{T}$, we have $(1)$ $T(f)|_E = \widehat{T}(f,\tilde{L}^\#)|_E = f$, and $(2)$ $\|T(f)\|_{\WR} = \|\widehat{T}(f,\tilde{L}^\#)\|_{\WR} \approx \widehat{M}(f,\tilde{L}^\#) \approx \|f\|_{\WR|_E}$.

Let $\widehat{M}_\nu(f_\nu,\tilde{L}_\nu) = \widehat{M}_{\tilde{Q}_\nu}(f_\nu,\tilde{L}_\nu)$. Using the formula for $\widehat{M}(f,\tilde{L}^\#)$ from Proposition \ref{mainprop}, the definition of $\tilde{L}$ as the constant-path extension of $\tilde{L}^\#$, and the fact that $|x_\nu - x^\#_{\mu(\nu)}| \lesssim \delta_\nu$ (from Lemma \ref{stdkeyprops}),
\begin{align}
\label{normform1}
\widehat{M}(f,\tilde{L}^\#) & = \sum_{\nu} \widehat{M}_{\nu}(f_\nu,\tilde{L}_\nu)^p \\
& + \sum^{K^\#}_{\mu,\mu'=1} \sum_{\substack{\nu \leftrightarrow \nu' \\ \mu(\nu) = \mu, \; \mu(\nu') = \mu'}} \left[ |\nabla \tilde{L}^\#_\mu - \nabla \tilde{L}^\#_{\mu'}|^p \delta^{2-p}_\nu + |\tilde{L}^\#_\mu(x^\#_\mu) - \tilde{L}^\#_{\mu'}(x^\#_\mu)|^p \delta^{2-2p}_{\nu} \right] \notag{} \\
& \approx \sum_{\nu} \widehat{M}_{\nu}(f_\nu,\tilde{L}_\nu)^p + \sum^{K^\#}_{\mu,\mu'=1}\left[ |\nabla \tilde{L}^\#_\mu - \nabla \tilde{L}^\#_{\mu'}|^p \Delta^{2-p}_{\mu \mu'} + |\tilde{L}^\#_\mu(x^\#_\mu) - \tilde{L}^\#_{\mu'}(x^\#_\mu)|^p \Delta^{2-2p}_{\mu \mu'} \right]. \notag{}
\end{align}
Here, the last ``$\approx$" follows by taking a dyadic sum, and setting $\Delta_{\mu \mu'} = \min\{\delta_{\nu} : \exists \nu' \leftrightarrow \nu \; \mbox{with} \; \mu(\nu)=\mu, \; \mu(\nu') = \mu'\}$.
From Theorem \ref{localthm}, and the linear dependence of $\tilde{L}_\nu$ on $f$ for each $\nu \in \{1,2,\cdots,K\}$, we find linear functionals $\lambda^\nu_1(f),\cdots,\lambda^\nu_{N_\nu}(f)$ with $N_\nu \lesssim (\#E_\nu)^2$, and
$$\widehat{M}_\nu(f_\nu,\tilde{L}_\nu)^p \approx \sum_{i=1}^{N_\nu} |\lambda^\nu_i(f)|^p.$$
Along with \eqref{normform1}, this implies
\begin{equation}
\label{normform2}
\widehat{M}(f,\tilde{L}^\#) \approx \sum_{\nu}\sum_{i=1}^{N_\nu} |\lambda^\nu_i(f)|^p + \sum^{K^\#}_{\mu,\mu'=1}\left[ |\nabla \tilde{L}^\#_\mu - \nabla \tilde{L}^\#_{\mu'}|^p \Delta^{2-p}_{\mu \mu'} + |\tilde{L}^\#_\mu(x^\#_\mu) - \tilde{L}^\#_{\mu'}(x^\#_\mu)|^p \Delta^{2-2p}_{\mu \mu'} \right].
\end{equation}
Finally, define 
\begin{equation}
\label{normform3}
M(f) = \sum_{\nu}\sum_{i=1}^{N_\nu} |\lambda^\nu_i(f)|^p + \sum^{K^\#}_{\mu,\mu'=1}\left[ |\nabla \tilde{L}^\#_\mu - \nabla \tilde{L}^\#_{\mu'}|^p \Delta^{2-p}_{\mu \mu'} + |\tilde{L}^\#_\mu(x^\#_\mu) - \tilde{L}^\#_{\mu'}(x^\#_\mu)|^p \Delta^{2-2p}_{\mu \mu'} \right],
\end{equation}
and note from \eqref{e149} and \eqref{normform2} that $M(f) \approx \|f\|_{\WR|_E}$. Using the Bounded Intersection Property of $\{\tilde{Q}_\nu\}$, and the bound $N_\nu \lesssim (\#(E_\nu))^2$, there are at most $\sum_\nu C(\#(1.1Q_\nu \cap E))^2 \lesssim N^2$ linear functionals used in the first sum in \eqref{normform3}.

From \textbf{(K2)} and the definition of property \textbf{R}, any $Q^\# \in \Lambda^\#$ satisfies $E \cap 9Q^\# \neq \emptyset$. For each $Q^\#_\mu$ we assign a point $y^\#_\mu \in E \cap 9Q^\#$. Note that each $y \in E$ has a bounded preimage under this assignment, since $\{10Q^\#_\mu\}^{K^\#}_{\mu=1}$ satisfies the Bounded Intersection Property from \textbf{(K3)}. Thus, the Keystone squares $\Lambda^\# = \{Q^\#_\mu\}^{K^\#}_{\mu=1}$ satisfy $K^\# \lesssim N$; from this, the second sum in \eqref{normform3} contains at most $C N^2$ terms, each a linear functional of $f$ raised to the $p$'th power. 

The two arguments above show that $M(f)^p$ is a sum of at most $C N^2$ linear functionals raised to the $p$'th power, and completes the proof of Theorem \ref{mainthm} when $\Lambda \neq \{Q^\circ\}$.

\section{Proof of Theorem \ref{mainthm} when $\Lambda = \{Q^\circ\}$}
\label{lastcase}

In this case, $Q^\circ$ is OK, and thus $\|E\|_{\bes} \leq c_1 \delta_{Q^\circ}^{2/p-1}$. Since $E \subset \frac{1}{10}Q^\circ$, we may fix $x_0 \in 0.9 Q^\circ$ satisfying $d(x_0,E) \geq 1/10 \delta_{Q^\circ}$. It is clear that $Q^\circ$, $E$, and $x_0$ satisfy the Geometric Assumptions from Section \ref{LMEP} as long as $c_1$ is sufficiently small. 

Thus, by Theorem \ref{localthm}, there exists a linear operator $\widehat{T}_0:L^{2,p}(Q^\circ)|_E \times P \rightarrow L^{2,p}(Q^\circ)$ and a non-negative real number $\widehat{M}(\cdot,\cdot)$ so that for any $L_0 \in P$, $(1)$ $\widehat{T}_0(f,L_0)|_E=f$, $(2)$ $J_{x_0} \widehat{T}_0(f,L_0) = L_0$, $\|\widehat{T}_0(f,L_0)\|_{L^{2,p}(Q^\circ)}$ is C-optimal with respect to these two properties, and $(4)$ $\|\widehat{T}_0(f,L_0)\|_{L^{2,p}(Q^\circ)} \approx \widehat{M}(f,L_0)$ with $\widehat{M}(f,L_0)^p = \sum_{i=1}^{N_1} |\lambda_i(f,L_0)|^p$ and $N_1 \lesssim (\# E)^2$.

As in the remarks proceeding Lemma \ref{upperboundlem}, we can extend $\widehat{T}_0(f,L_0)$ to a function $\widehat{T}(f,L_0) \in \WR$ without increasing its norm by more than a constant factor, and without ruining $P1$ and $P2$ above. Using Claim 1 in the proof of Lemma \ref{keylem}, we may find $L_1 \in P$ depending linearly on $f$ with $\widehat{M}(f,L_1) \lesssim \widehat{M}(f,L_0)$ for all $L_0 \in P$. Let $T(f) = \widehat{T}(f,L_1)$ and $M(f) = \widehat{M}(f,L_1)$, which satisfy $(1)$ $T(f)|_E = f$, $(2)$ $\|T(f)\|_{\WR}$ is C-optimal with respect to this property, and $(3)$ $\|T(f)\|_{\WR} \approx M(f)$ with $M(f)^p = \sum_{i=1}^{N_1} |\lambda_i(f,L_1)|^p$. Since $N_1 \lesssim (\#E)^2 = N^2$, Theorem \ref{mainthm} is proven for this case as well.

\end{document}